\documentclass[11pt, letterpaper, reqno]{amsart}
\usepackage[utf8]{inputenc} 	
\usepackage{microtype} 			
\usepackage{geometry} 
\usepackage{amsmath}			
\usepackage{amsthm}		 		
\usepackage{amssymb}
\usepackage{bm}					
\usepackage{mathrsfs}			
\usepackage{thmtools}
\usepackage{xcolor}				
\definecolor{darkred}{rgb}{0.6, 0.1, 0.1}
\definecolor{darkblue}{rgb}{0.2, 0.2, 0.6}
\definecolor{darkgreen}{rgb}{0.2, 0.4,.1}
\definecolor{mellowyellow}{rgb}{1,.8,.2}
\definecolor{bettercyan}{rgb}{0.1, 0.4, 0.7}
\definecolor{bettermagenta}{rgb}{0.6, 0.2, 0.6}
\usepackage{tikz}				
\usetikzlibrary{patterns}
\usetikzlibrary{arrows}
\usepackage[all]{xy}			
\usepackage{graphicx}			
\usepackage{pgfplots}			
\usepackage{enumitem} 			
\usepackage{array}
\usepackage{lmodern}			
\usepackage[T1]{fontenc}		
\usepackage{cancel}
\usepackage[breaklinks=true, colorlinks=true, linkcolor=darkred, citecolor=bettercyan, urlcolor=bettermagenta]{hyperref}
\usepackage{cleveref}
\usepackage[backend=biber, style=alphabetic, backref=true, firstinits=true, isbn=false, date=year]{biblatex}
\makeatletter
\renewbibmacro{in:}{}
\DeclareFieldFormat{pages}{#1}
\renewcommand*{\bibnamedash}{%
	\leavevmode\raise +0.6ex\hbox to 5.5ex{\hrulefill}.\space\space}

\InitializeBibliographyStyle{\global\undef\bbx@lasthash}

\newbibmacro*{bbx:savehash}{\savefield{fullhash}{\bbx@lasthash}}

\renewbibmacro*{author}{%
	\ifboolexpr{
		test \ifuseauthor
		and
		not test {\ifnameundef{author}}
	}
	{%
		\iffieldequals{fullhash}{\bbx@lasthash}
		{\bibnamedash\addcomma\space}
		{\printnames{author}}%
		\usebibmacro{bbx:savehash}%
		\iffieldundef{authortype}
		{}
		{%
			\setunit{\addcomma\space}%
			\usebibmacro{authorstrg}%
		}%
	}
	{\global\undef\bbx@lasthash}%
}
\makeatother

\newtheorem{propositionx}{Proposition}[section]
\newenvironment{proposition}
{\pushQED{\qed}\propositionx}
{\popQED\endpropositionx}

\newtheorem*{theorem*}{Theorem}
\newtheorem{theorem}{Theorem}

\newtheorem*{corollaryx}{Corollary}
\newenvironment{corollary}
{\pushQED{\qed}\corollaryx}
{\popQED\endcorollaryx}
\newtheorem{lemmax}[propositionx]{Lemma}
\newenvironment{lemma}
{\pushQED{\qed}\lemmax}
{\popQED\endlemmax}

\theoremstyle{remark}
\newtheorem{remark}[propositionx]{Remark}
\newtheorem*{remark*}{Remark}

\newcommand{\bbB}{\mathbb{B}}
\newcommand{\bbC}{\mathbb{C}}

\newcommand{\bbN}{\mathbb{N}}

\newcommand{\bbR}{\mathbb{R}}
\newcommand{\bbS}{\mathbb{S}}

\newcommand{\bbZ}{\mathbb{Z}}

\newcommand{\calA}{\mathcal{A}}

\newcommand{\calD}{\mathcal{D}}
\newcommand{\calE}{\mathcal{E}}
\newcommand{\calF}{\mathcal{F}}
\newcommand{\calG}{\mathcal{G}}

\newcommand{\calI}{\mathcal{I}}
\newcommand{\calJ}{\mathcal{J}}

\newcommand{\calL}{\mathcal{L}}

\newcommand{\calS}{\mathcal{S}}

\newcommand{\calV}{\mathcal{V}}

\newcommand{\calX}{\mathcal{X}}

\newcommand{\dd}{\,\mathrm{d}}
\makeatletter
\@namedef{subjclassname@2020}{\textup{2020} Mathematics Subject Classification}
\makeatother

\title[Asymptotics in all regimes for the Schr\"odinger equation]{Asymptotics in all regimes for the Schr\"odinger equation with time-independent coefficients}

\date{April 25th, 2026 (Last update).}
\author{Sam Looi and Ethan Sussman}
\email{looi@caltech.edu}
\email{ethan.sussman@northwestern.edu}
\address{The Division of Physics, Mathematics and Astronomy, California Institute of Technology, California, USA}
\address{Department of Mathematics, Northwestern University, Illinois, USA}

\begin{document}

\begin{abstract}
	Using the recent analysis of the output of the low-energy resolvent of short-range Schr\"odinger operators on asymptotically conic manifolds (including Euclidean space), we produce detailed asymptotic expansions for the solutions of the Schr\"odinger equation's initial-value problem. Asymptotics are calculated in all joint large-radii large-time regimes, these corresponding to the boundary hypersurfaces of a particular compactification of spacetime.
\end{abstract}

\maketitle

\tableofcontents

\section{Introduction}

In recent work \cite{HintzPrice}\cite{LooiXiong}, wave propagation on stationary, asymptotically flat spacetimes has been analyzed using Vasy's spectral methods \cite{Vasy,VasyLagrangian}, which involve low-energy resolvent estimates for Schr\"odinger operators (i.e.\ for the ``time-independent'' Schr\"odinger equation, in physicists' preferred terminology) and are closely related to earlier work of Guillarmou--Hassell--Sikora \cite{GH1, GH2, GHK}, with numerous precursors in the wider literature.
These tools apply to Schr\"odinger operators with short range potentials, which in \cite{HintzPrice}\cite{Full} means decaying cubically or faster. The quadratic case is similar, requiring only minor modifications. Long range potentials (e.g. any Coulomb-like potential, decaying like $\sim r^{-1}$) require serious modifications which we do not discuss here.

We focus on the $3$-dimensional case of most physical interest, for which the spectral methods are most developed.
Let $(X,g)$ denote an asymptotically conic manifold \cite{MelroseSC, MelroseGeometric} (see below for the precise assumptions), and let $\rho$ denote a boundary-defining function (bdf) on $X$.
Let $\calS(X) = \bigcap_{k\in \bbN} \rho^k C^\infty(X)$ denote the Fr\'echet space of Schwartz functions on $X$.

The reader is invited to consider the case of exact Euclidean space. Then,
\begin{equation}
X=\overline{\bbR^3} = \bbR^3\cup \infty \bbS^{2},
\end{equation}
which is the compactification of $\bbR^3$ constructed by adding on the 2-sphere $\infty \bbS^{2}$ at infinity, and $g$ is the exact Euclidean metric. The results below are novel even in this case, where the analysis sharpens the seminal results of Jensen--Kato \cite{JensenKato} by providing large-argument asymptotics in addition to long-time asymptotics.
Dispersive estimates for the Schr\"odinger equation (including scattering off obstacles, on manifolds, or with nonlinearities) have been a popular theme for a while now --- see e.g.\ \cite{JensenSolo}\cite{SoggeDispersive}\cite{StrichartzTeam}\cite{SchlagB}\cite{Magnetic}\cite{SchlagA} for a few examples; \cite{SchlagB} is a survey with many references --- and global dispersive estimates probe the behavior of solutions in all large-$t$ regimes.
Nevertheless, as far as we know the complete description of the asymptotics below is new even in the Euclidean case. There are three particular aspects, the compatibility of the asymptotic expansions at adjacent regimes, the Schwartz behavior of the scattering data, and the expansion in the parabolic scaling regime, where we are unaware of a reference showing that the result holds even to leading-order.

In the Euclidean case, a convenient choice of bdf is $\rho = 1/\langle r \rangle$, where $r$ is the Euclidean radial coordinate and $\langle r \rangle = (1+r^2)^{1/2}$ is the Japanese bracket. Also, $\calS(X) = \calS(\bbR^3)$ is just the usual set of Schwartz functions on $\bbR^3$.

Because of our familiarity with the Euclidean case, even when working with general $X$ it is usually easier to use the coordinate $r(x) = \rho(x)^{-1} \in C^\infty(X^\circ)$ in place of the bdf $\rho(x)$. For example, $g$ is given to leading-order (with respect to some collar neighborhood of the boundary) by $\mathrm{d} r^2 + r^{2} g_{\partial X}$ for $g_{\partial X}$ a Riemannian metric on $X$. This is the metric of an exact cone.
However, it should be kept in mind that, in the exact Euclidean case, $r(x) = \langle r \rangle$, where $r$ is the Euclidean radial coordinate. This slightly overloaded notation should not cause confusion. Whenever we use `$r$' below, we mean $\rho(x)^{-1}$, not the Euclidean radial coordinate.

In this paper, we apply Hintz--Vasy's low-energy toolkit to the Schr\"odinger initial value problem
\begin{equation*}\label{eq:Schrodinger}
\begin{cases}
 -i \partial_t u = \Delta_g u +  i  A\cdot  \nabla_g u + (V+2^{-1} i\nabla_g\cdot A) u, \\
 u(0,x) = f(x),
 \end{cases}
 \tag{IVP}
\end{equation*}
posed on the manifold $\bbR_t\times X^\circ$, where $f\in \calS(X)$,
\begin{itemize}
	\item $\Delta_g$ is the positive-semidefinite Laplace--Beltrami operator,
	\item $\nabla_g : C^\infty(X^\circ) \to \calV(X^\circ)$ is the gradient operator which is antisymmetric with respect to the $L^2(X, g)$-inner product,
	 $A\cdot \nabla_g u(x) = g(A(x),\nabla_g u(x))$, and $\nabla_g\cdot $ is the corresponding divergence operator, and
	\item  $A \in r^{-2}\calV(X;\bbR)$, and $V\in r^{-3} C^\infty(X;\bbR)$.
\end{itemize}
The vector field $A$ plays the role of the vector potential generated by a short-range magnetic field, e.g.\ that generated by a magnetic dipole, and $V$ plays the role of the voltage generated by a short-range electric field, such as that generated by an electric quadrupole. We exclude more slowly decaying fields.

Note that $L^2(X,g)$ is not just $L^2(X)$, where $X$ is considered as a compact manifold-with-boundary, because $g$ does not extend smoothly across the boundary of $X$. However, in the exact Euclidean case, $L^2(X,g) = L^2(\bbR^3)$. To emphasize the distinction, $L^2(\bbB^3)\cong L^2(\overline{\bbR^3})\neq L^2(\bbR^3)$, because e.g.\ the former contains all constant functions while the latter does not. Here, the isomorphism ``$\cong$'' is that induced by the diffeomorphism $\bbB^3 \cong \overline{\bbR^3}$.

The coefficients of the PDE are static, i.e.\ constant in the time coordinate $t$.  Moreover, the differential operator
\begin{equation}
P= \Delta_g +  i A\cdot \nabla_g + 2^{-1} i\nabla_g\cdot A + V \in \operatorname{Diff}^2(X^\circ)
\end{equation}
is formally symmetric with respect to the $L^2(X, g)$-inner product, $\langle \phi,\psi \rangle_{L^2(X,g)} = \int_X \phi^* \psi \dd \mathrm{Vol}_g$. This enables the application of spectral-theoretic tools which otherwise might not apply or would require some work to justify.
We also make the following assumptions, as in Hintz's work:
\begin{enumerate}[label=(\Roman*)]
	\item (No zero energy resonance or bound state.) The operator $P$ has trivial nullspace acting on $\calA^1(X)=r^{-1}\calA^0(X) \subset r^{-1} L^\infty $ \cite[Def.\ 2.8]{HintzPrice}.
	\item The high energy estimates stated in \cite[Def.\ 2.9]{HintzPrice} apply.
\end{enumerate}
There are many function spaces that can replace $\calA^1(X)$ in (I). The point is just to rule out the presence of nonzero solutions $u$ to $Pu=0$ decaying like $1/r$.
The high energy estimates apply whenever the metric $g$ is non-trapping or exhibits only normally hyperbolic trapping. In particular, (II) holds in the exact Euclidean case or in any sufficiently small perturbation thereof.

\begin{theorem}
	The solution of the initial-value problem \cref{eq:Schrodinger} is of exponential-polyhomogeneous type on the compactification $M\hookleftarrow (0,\infty)_t\times X$ given by the iterated blowup
	\begin{equation}
		M=[[ [0,\infty]_t\times X; \{\infty\}\times \partial X ]; \beta^{-1}(\{\infty\}\times \partial X)\cap  \operatorname{cl}\beta^{-1}( \{\infty\}\times X)  ],
	\end{equation}
	where $\beta: [ [0,\infty]_t\times X; \{\infty\}\times \partial X ] \to [0,\infty]_t\times X$ is the blowdown map.
	\label{thm:A}
\end{theorem}

We will explain the construction of $M$ in more detail below. For now, see \Cref{fig:M}, which describes $M$ near $\beta^{-1}([0,\infty] \times \partial X)$ via an atlas (ignoring the angular degrees of freedom associated with $\partial X$).
Also, see the discussion below for the definition of the notion of exponential-polyhomogeneity appearing in the theorem. It is a term used to state that asymptotic expansions hold without specifying the forms of those asymptotic expansions. More precise theorems (in particular, \Cref{thm:B}) appear later.

We will sketch the proof in \S\ref{subsec:sketch}. Analytically, \cite{Vasy, VasyLagrangian} contains the key estimate; \cite{HintzPrice}\cite{Full} parlay these estimates into asymptotic expansions for the limiting resolvent, or equivalently for the spectral projector. Our main contribution here is the deduction of spacetime asymptotics from the asymptotic expansions on the spectral side. This procedure is more involved than the corresponding procedure for the wave equation, which can be found already in \cite{HintzPrice}. The main theorem which accomplishes this task is labeled \Cref{thm:D} below. It essentially states that oscillatory integrals of a particular form are of exponential-polyhomogeneous type on $M$. The more precise theorem below, \Cref{thm:B}, requires additional argumentation, and this is found in \S\ref{sec:sad}.

One surprising feature of \Cref{thm:A} is the presence of an asymptotic regime corresponding to the parabolic scaling of the PDE. This is labeled ``parF'' in \Cref{fig:M}. No such regime is necessary for the wave equation. It is tempting to label this an obvious consequence of the parabolic scaling of the free Schr\"odinger equation, but in the free case the face parF can be blown-down. It is only in the variable coefficient case where it is necessary. As will become clear below, asymptotic expansions at parF are derived from spectral asymptotics at the face $\mathrm{tf}$ in \cite{Vasy, VasyLagrangian}\cite{HintzPrice}\cite{Full}. This asymptotic regime is a delicate transitional regime in which $\sigma\to 0^+$ and $r\to\infty$ at a comparable rate. Its importance was first recognized by Guillarmou--Hassell--Sikora in \cite{GH1, GH2, GHK}. So, it should not be entirely surprising that, on the spacetime side, the regime $\mathrm{parF}$ has received scant attention.

As far as we are aware, \Cref{thm:A} should hold if $d\neq 2$,  not just $d=3$, and for any $A\in r^{-1}\calV(X)$ and $V\in r^{-2} C^\infty(X)$, regardless of whether or not there is a zero energy resonance or bound state.
Our techniques are quite general, but the spectral side has yet to be developed sufficiently for our techniques to allow us to deduce such a result using the tools below. For example, we cite \cite{HintzPrice}\cite{Full} as a black box, but these works forbid the appearance of a zero energy resonance or bound state. Thus, one avenue for future work is extending the results here to more general settings. It should be remarked that \cite{Full} handles $d>3$ as well as $d=3$, so it actually would have been possible to treat $d>3$ in this work.

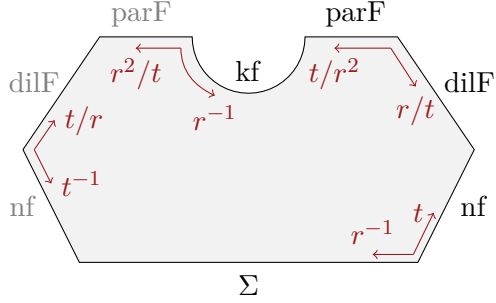
\begin{figure}
	\begin{center}
		\begin{tikzpicture}[scale=3]
		\filldraw[fill=lightgray!20] (-1,.5) -- (-.75,0) -- (.75,0) -- (1,.5) -- (.66,1) -- plot[domain=0:180] ({.25*cos(\x)}, {1-.25*sin(\x)}) -- (-.66,1) -- (-1,.5);
		\node at (1,.25) {$\mathrm{nf}$};
		\node at (.98,.8) {$\mathrm{dilF}$};
		\node at (.475,1.1) {$\mathrm{parF}$};
		\node[gray] at (-1,.25) {$\mathrm{nf}$};
		\node[gray] at (-.95,.8) {$\mathrm{dilF}$};
		\node[gray] at (-.475,1.1) {$\mathrm{parF}$};
		\node at (0,.85) {$\mathrm{kf}$};
		\node at (0,-.1) {$\Sigma$};
		\draw[->, color=darkred] (.73,.035) -- (.55,.035) node[above] {$r^{-1}$};
		\draw[->, color=darkred] (.73,.035) -- (.82,.22) node[left] {$t$};
		\draw[->, color=darkred] (-.95,.5) -- (-.875,.35) node[right] {$t^{-1}$};
		\draw[->, color=darkred] (-.95,.5) -- (-.86,.63) node[right] {$t/r$};
		\draw[->, color=darkred] (.63,.95) -- (.38,.95) node[below] {$t/r^2$};
		\draw[->, color=darkred] (.63,.95) -- (.74,.78) node[below] {$r/t$};
		\draw[->, color=darkred] (-.3,.95) -- (-.5,.95) node[below] {$r^2/t$};
		\draw[->, color=darkred] (-.3,.95) to[out=-90, in=155] (-.15,.74) node[below] {$r^{-1}$};
		\end{tikzpicture}
	\end{center}
	\caption{The mwc $M \supset \bbR_t\times X$, with an atlas of coordinate charts near $\{\rho=0\}\subset \partial M$ depicted. Here, $r = 1/\rho$. The Cauchy hypersurface $\mathrm{cl}_M\{t=0\}$ is $\Sigma$. Note that some of the faces appear disconnected only because the drawing is in 1+1D. Really, there is a $\partial X$ factor we are not drawing.}
	\label{fig:M}
\end{figure}

Finally, it is worth pointing out recent work \cite{Parabolicsc, HassellNL} that develops microlocal tools which can be used, among other things, to produce asymptotic expansions for the Schr\"odinger equation with \emph{time-dependent} coefficients that settle down to the constant-coefficient case. Their methods cannot handle coefficients that settle down to something non-constant. Conversely, our methods cannot handle any time-dependence in the coefficients whatsoever. So, there is no overlap between our results and those of Hassell et al., besides the constant-coefficient case that is included in both. Moreover, the microlocal tools required are very different.

In the future, it will hopefully be possible to combine the analysis here with that in \cite{Parabolicsc, HassellNL} to handle time-dependent coefficients which settle down to something stationary but with spatial variation. This would be the Schr\"odinger analogue to Hintz's recent work \cite{Hintz3b} on the wave equation combining \cite{BasinVasyWunsch}\cite{HintzPrice}.

\subsection{More precise form of the main theorem}

For a function $v \in C^\infty(M^\circ)$ to be \emph{of exponential-polyhomogeneous type} on the manifold-with-corners (mwc) $M$ means that $v$ can be written as a finite sum $v = \sum_{n=1}^N e^{i \theta_n} v_n$
for $\theta_n,v_n \in C^\infty(M^\circ)$ polyhomogeneous functions on $M$.
Polyhomogeneity is a widely used generalization of smoothness due to Melrose \cite{ MelroseCorners, MelroseAPS} which allows bounded powers of logarithms to appear in Taylor series. Such ``generalized Taylor series'' are called \emph{polyhomogeneous expansions}.
The specific combinations of logarithms and powers allowed are specified by an \emph{index set} $\calE_{\mathrm{f}}\subset \bbC\times \bbN$ at each boundary hypersurface $\mathrm{f}\subset M$ of $M$. (An index set consists of pairs of numbers that are used to characterize the singularities of a distribution or a solution to a PDE.) Schematically, a polyhomogeneous function $v:M^\circ\to \bbC$ with index set $\calE_{\mathrm{f}}$ at $\mathrm{f}$ admits the polyhomogeneous expansion
\begin{equation}
v \sim \sum_{(j,k)\in \calE_{\mathrm{f}}} v_{\mathrm{f};j,k}\varrho_{\mathrm{f}}^j \log^k(\varrho_{\mathrm{f}})
\label{eq:misc_003}
\end{equation}
at $\mathrm{f}$,
where $\varrho_{\mathrm{f}} \in C^\infty(M)$ is a boundary-defining function of $\mathrm{f}$ and the $v_{\mathrm{f};j,k} \in C^\infty(\mathrm{f}^\circ)$ are polyhomogeneous functions on $\mathrm{f}$, which itself is a mwc of one lower dimension.
Polyhomogeneity at corners guarantees the existence of \emph{joint} asymptotic expansions there. This is equivalent to saying that the polyhomogeneous expansions at adjacent boundary hypersurfaces are compatible. Concretely, this means that $(v_{\mathrm{f};j,k})_{\mathrm{f}\cap \mathrm{F};J,K} = (v_{\mathrm{F};J,K})_{\mathrm{f} \cap \mathrm{F};j,k}$
for any adjacent boundary hypersurfaces $\mathrm{f},\mathrm{F}$ and $(j,k)\in \calE_{\mathrm{f}}, (J,K)\in \calE_{\mathrm{F}}$.
So, the notion of exponential-polyhomogeneous type is a formalization of the notion of admitting a full atlas of (term-by-term differentiable, in all directions) asymptotic expansions in terms of elementary functions.
If $M$ is compact, then, in some sense, exponential-polyhomogeneity means that the provided atlas of asymptotic expansions is complete. (The opposite extreme is when $M$ has no boundary, in which case exponential-polyhomogeneity just means smoothness and therefore says nothing about asymptotics.)

We refer to the cited works \cite{MelroseAPS, MelroseCorners}\cite{GrieserBasics}\cite{HintzPrice}\cite{SherBessel} for further discussion of polyhomogeneity and the function spaces capturing it, as well as for the related notion of conormality. Our notational conventions mostly follow \cite{HintzPrice}\cite{Full} and are explained as needed. In particular, ``$\calA^{(\calE,\alpha)}$'' is used to refer to partially polyhomogeneous behavior with index set $\calE$ and a conormal error of order $\alpha \in \bbR$, and ``$\calA^\calE$'' means purely polyhomogeneous behavior.
One abbreviation used throughout is that, for $j\in \bbR$ and $k\in \bbN$, ``$(j,k)$'' means the index set $\{(j+n,\kappa)\in \bbC\times \bbN : n\in \bbN, \kappa \leq k \}$. Also, `$\infty$' means empty index set, i.e.\ Schwartz behavior. For instance, $\calA^\infty(X) = \calS(X)$, and $\calA^{(j,0)}(X) = r^{-j} C^\infty(X)$.

So, \Cref{thm:A} states that solutions of the Schr\"odinger equation (with Schwartz initial data) are governed by four asymptotic regimes (five if we include the Cauchy hypersurface $\Sigma = \{t=0\} \subset M$), one regime for each of the four boundary hypersurfaces $\mathrm{nf},\mathrm{dilF},\mathrm{parF},\mathrm{kf} \subset M$
of the mwc $M$ appearing in the theorem.
A more precise version of the theorem -- which, when combined with \Cref{prop:bound}, which studies the sum over eigenfunctions in \cref{eq:misc_loo}, yields \Cref{thm:A} -- is:
\begin{theorem}
	Let $\phi_1,\cdots,\phi_N \in \calS(X)$ denote the (automatically Schwartz) square-integrable eigenfunctions of $P$, so that $P\phi_n = -E_n \phi_n$ for some $E_n>0$. Let $f\in \calS(X)$.
	If $u \in C^\infty([0,\infty)_t\times X^\circ)$ solves the initial-value problem \cref{eq:Schrodinger}, and if $\chi\in C_{\mathrm{c}}^\infty(\bbR)$ satisfies $0\notin \operatorname{supp}(1-\chi)$, then
	\begin{equation}
	u(t,x) = \exp\Big[ -\frac{i (1-\chi(t))}{4 t \rho(x)^2} \Big] \frac{  u_{\mathrm{phg}}(t,x)}{(t+i\epsilon)^{3/2}} + \sum_{n=1}^N e^{-i E_n t} \phi_n(x) \langle \phi_n,f \rangle_{L^2(X,g)}
	\label{eq:misc_loo}
	\end{equation}
	for $u_{\mathrm{phg}} = u_{\mathrm{phg}}[\chi]$ polyhomogeneous on $M$, with
	\begin{equation}
		u_{\mathrm{phg}} \in \calA^{(0,0)\cup (1/2,0)\cup \calE, (0,0)\cup \calF,(0,0),\infty,(0,0)}(M)
	\end{equation}
	for some index sets $\calE \subset (2^{-1} \bbN^{\geq 2})\times \bbN$, $\calF\subset \bbN^{\geq 1}\times \bbN$, where the index sets are specified at $\mathrm{kf}$, $\mathrm{parF}$, $\mathrm{dilF}$, $\mathrm{nf}$, and $\Sigma$, respectively. So, the index set at $\mathrm{kf}$ is $(0,0)\cup (1/2,0)\cup \calE$, the index set at $\mathrm{parF}$ is $(0,0)\cup \calF$, the index set at $\mathrm{nf}$ is empty, and the remaining index sets are $(0,0)$.
	Moreover, the leading-order asymptotic at $\mathrm{kf}\cup \mathrm{parF}$ has the following form: for some polyhomogeneous $w\in \calA^0(X)$ and $v \in \calA^{1-}(X)$,
	\begin{equation}
	u_{\mathrm{phg}}(t,x) - \chi(r/t) w(x)-\chi(r^2/t) v(x) \in \calA^{(1/2,0)\cup \calE,\calF,(0,0) \cup \calF,\infty,(0,0)}(M).
	\label{eq:lo}
	\end{equation}
	The restriction $w|_{\partial X}(\theta)$ does not depend on $\theta$, and $w|_{\partial X}=\Lambda(f)$ for some linear functional $\Lambda:\calS(X)\to \bbC$.  In fact, 
	\begin{equation*} 
		w(x) = \frac{-\sqrt{\pi i}}{2\pi i} r P^{-1} f,
	\end{equation*} 
	where $P^{-1} f \in \calA^1(X)$ is the unique solution $w_+=P^{-1} f$ to $Pw_+ = f$ such that $w_+\in \calA^{0+}(X)$.
	The functional $\Lambda$ is given by
	\begin{equation}
		\Lambda(f) \propto \langle   \smash{u^{(0)},f }\rangle_{L^2(X,g)},
	\end{equation}
	where $\smash{u^{(0)}}\in C^\infty(X)$ is the unique smooth solution of $Pu^{(0)}=0$ that is identically $1$ at $\partial X$.
	Letting
	\begin{equation}
		L =-i \frac{\mathrm{d}}{\mathrm{d} \sigma} M_{\exp(- i \sigma r)} P M_{\exp(i \sigma r) } \Big|_{\sigma =0} \in \operatorname{Diff}^1(X^\circ),
		\label{eq:Ldef}
	\end{equation}
	in which $M_{\bullet}: w \mapsto \bullet w$ denotes a multiplication operator, the function $v$ is given by $v(x) = \frac{\sqrt{\pi i}}{2\pi i} P^{-1} (r  + L P^{-1} )f$.
	\label{thm:B}
\end{theorem}

\begin{remark*}
	Under the stated assumptions, it is the case \cite[\S2]{HintzPrice} that we have a well-defined one-sided inverse $P^{-1}: \calA^{2+\alpha}(X)\to \calA^{\alpha-}(X)$
	for any $\alpha \in (0,1)$, where $\calA^{\alpha-}(X) = \bigcap_{\epsilon>0} \calA^{\alpha-\epsilon}(X)$. So, $P^{-1} f \in \calA^{1-}(X)$. A standard argument lets us upgrade this to $P^{-1} f \in \calA^{(1,0),2-} (X)$; see \cite[Appendix A]{Full}, specifically the main theorem in that section.
	Because $L$, which is given by
	\[L= -2  r^{-1}(r \partial_r+1) \bmod r^{-2}\operatorname{Diff}^1_{\mathrm{b}}(X),
	\]
	(cf.\ \cite[below eq.\ 1.11]{HintzPrice}, our $L$ being related to $L(\sigma)$ there by $L=-iL'(\sigma)|_{\sigma=0
	}$)
	satisfies $L \rho \in \rho^3 C^\infty(X)$, it is the case that $LP^{-1} f \in \calA^3(X)$ and therefore that $P^{-1} L P^{-1} f \in \calA^{1-}(X)$. The same theorem in \cite[Appendix A]{Full} shows that $P^{-1} L P^{-1} f$ is polyhomogeneous. This justifies the description of the profiles $v,w$ in \Cref{thm:B}. We refer to \cite[\S3]{HintzPrice} for the details regarding the large $r$ asymptotics of $w$.
\end{remark*}

For the reader uncomfortable with the notion of polyhomogeneity, the following $L^\infty$-based corollary follows immediately from the theorem:
\begin{corollary}
	For any $K\in \bbN$,
	\begin{multline}
	u(t,x) =\underbrace{\exp\Big[ -\frac{i (1-\chi(t))}{4 t \rho(x)^2} \Big] \frac{ \chi(r/t) w(x) + \chi(r^2/t) v(x)}{(t+i\epsilon)^{3/2}}}_{\text{leading-order dispersive part}} + \underbrace{\sum_{n=1}^N e^{-i E_n t} \phi_n(x) \langle \phi_n,f \rangle_{L^2(X,g)}}_{\text{bound states}}  \\ +
	\underbrace{O \Big(\frac{1}{\langle t \rangle^{3/2}} \Big\langle \frac{r}{t} \Big\rangle^{-K}  \Big\langle \frac{t}{r} \Big\rangle^{-1/2} \Big)}_{\text{dispersive error}}.
	\label{eq:misc_acd}
	\end{multline}
	The big-O term in \cref{eq:misc_acd}, which is bounded above by $O(r^{1/2}/t^{2})$, is suppressed relative to the other terms as $t\to \infty$ in $t\gg r$. That is, for any $\epsilon>0$ the big-O term is $o(t^{-3/2})$ if $r=o(t)$.
	Moreover, for any $\varepsilon>0$,
	\begin{equation}
	u(t,x) =\exp\Big[ -\frac{i (1-\chi(t))}{4 t \rho(x)^2} \Big] \frac{   \Lambda(f)}{(t+i\epsilon)^{3/2}}  + O \Big( \Big\langle \frac{rt}{t+r^2} \Big\rangle^{-4+\varepsilon} \Big).
	\label{eq:misc_kjl}
	\end{equation}
	The big-O term in \cref{eq:misc_kjl} is suppressed relative to the other terms for $t \sim r^2$; if $cr^2 < t < C r^2$ for some $0<c<C$, then the big-O term is $O(t^{-2+\varepsilon/2})=O(r^{-4+\varepsilon})$.
\end{corollary}

So, the solution consists of a sum over bound states and a dispersive remainder, whose large-$t$ asymptotics have been given explicitly within $r/t \ll 1$.

\begin{remark*}
	One minor improvement of \Cref{thm:B} is that the index sets $\calE,\calF$ can be related to the polynomial decay rate of the coefficients of $P-\Delta_{g_0}$ for $g_0$ the exactly conic metric on which $g$ is modeled. The larger the degree, the smaller these index sets can be taken. This improvement follows from \Cref{thm:D} and the corresponding fact in the main theorem of \cite{Full}.
\end{remark*}

Using Duhamel's principle, it is straightforward to show that: 
\begin{theorem}
	The conclusions of the theorems above apply also to the inhomogeneous version of \cref{eq:Schrodinger} with forcing in $\calS(\bbR_t; \calS(X)) = \calS(\bbR_t\times X)$, except without the explicit formulas for the terms $w,v$ in the asymptotic profile.
	\label{thm:forced} 
\end{theorem}
\begin{remark*}
	One can still compute a formula for $w,v$ in this case, it is just slightly more complicated. 
\end{remark*}

The proof below of the theorems above is essentially constructive, in the sense that it yields an algorithm for computing the asymptotic expansions of $u_{\mathrm{phg}}$ in all possible regimes, not just $\mathrm{kf}\cup \mathrm{parF}$, and not just to leading-order. The algorithm can be extracted from the proof. It produces expansions on $M$ in terms of the coefficients of the expansions in \cite{Full}. Insofar as these coefficients are explicit, so too are the asymptotics on $M$.

\begin{remark*}
	It is worth mentioning that if \cite[Thm.\ 3.1]{HintzPrice} is used instead of \cite{Full}, then one gets \Cref{thm:B} except with only conormal estimates of the remainder in \cref{eq:lo}. In particular, the $L^\infty$-based corollary above can still be deduced.
	Moreover, this applies even if $V,A$ merely satisfy \emph{symbolic} estimates (so, do not necessarily extend smoothly to $\partial X$), except in this case $w,v$ are only going to be partially polyhomogeneous.
\end{remark*}

\begin{remark*}[Initial data with finite decay]
	In this paper, we only consider Schwartz initial data, but the same tools allow one to study the case where the initial data is polyhomogeneous on $X$ with some finite decay rate at $\partial X$. 
\end{remark*}

\subsection{Geometric setup and spacetime compactification}
\label{subsec:geometric_setup}

Concretely, that $(X,g)$ be an asymptotically conic manifold means that $X$ is a smooth manifold-with-boundary and $g$ is a Riemannian metric on $X^\circ$ satisfying the following: for some $\bar{\rho}>0$ and embedding $\iota: [0,\bar{\rho}]_\rho \times \partial X \to X$ satisfying $\iota(0,-) = \mathrm{id}_{\partial X}$ (that is, a collar neighborhood of the boundary), and for some Riemannian metric $g_{\partial X}$ on $\partial X$, the pullback $\iota^* g$ has the form\footnote{Note that this rules out long-range contributions to the metric, i.e.\ terms subleading by only one order in $\rho$. In scattering theory, these tend to lead to logarithmic corrections to phases.}
\begin{equation}
\iota^* g - \rho^{-4} \mathrm{d} \rho^2- \rho^{-2} g_{\partial X} \in \rho^2 C^\infty(\operatorname{Sym} {}^{\mathrm{sc}} T^* ([0,\bar{\rho})_\rho \times \partial X) ),
\label{eq:conic_def}
\end{equation}
where ${}^{\mathrm{sc}} T^* X$ is the vector bundle over $X$ whose smooth sections are given by $C^\infty(X)\rho^{-2} \mathrm{d} \rho$, $C^\infty(X)\rho^{-1} \omega$ for $\omega \in \Omega^1(\partial X)$ and linear combinations thereof. That is, $g$ differs from the exactly conic metric $\rho^{-4} \mathrm{d} \rho^2 + \rho^{-2} g_{\partial X}$ by suitably decaying terms. In the exact Euclidean case, $g_{\partial X}$ is the standard metric (or any scalar multiple thereof) on the 2-sphere at infinity.

The first component of $\iota^{-1}$ serves as a boundary-defining function (bdf). That is, there exists a bdf $\rho \in C^\infty(X; [0,\infty) )$ such that $\rho(\iota(\varrho,\theta) ) = \varrho$ for all $\varrho \in [0,\bar{\rho}]$. That this is a bdf means that $\rho^{-1}(\{0\}) = \partial X$ and that $\mathrm{d} \rho$ is nonvanishing on $\partial X$ (and also that $\rho$ is nonnegative). Going forwards, we will identify $[0,\bar{\rho}]_\rho \times \partial X$ with its image under $\iota$.
We will use the notation $\dot{X}[R] = [0, R^{-1} ]_\rho \times \partial X_\theta$,
and this can be considered as a subset of $X$ as long as $R>\bar{\rho}^{-1}$. The subscripts here signal preferred variable names used to parameterize each factor, and similar notation is used throughout below.

We now describe the construction of $M$ in a bit more detail.
As a starting point, let $C$ denote the ``cylinder'' $C = [0,\infty]_t\times  X$. Consider the mwc
\begin{equation}
	M/ \mathrm{parF}=[C; \{\infty\} \times \partial X]  = C^\circ \cup \Sigma \cup \mathrm{nf} \cup \mathrm{dilF}_0 \cup \mathrm{kf}_0
\end{equation}
resulting from performing a polar blowup of the corner $\{\infty\} \times \partial X\subset C$ of $C$. Here, $\Sigma = \{t=0\}$, and the remaining three boundary hypersurfaces $\mathrm{nf},\mathrm{dilF}_0,\mathrm{kf}_0$ are the lift of $[0,\infty]_t\times \partial X$, the front face of the blowup, and the lift of $\{\infty\}\times X$, respectively. Then, $M$ can be constructed in terms of $M/ \mathrm{parF}$ as $M = [M/ \mathrm{parF}; \mathrm{dilF}_0 \cap \mathrm{kf}_0 ]$,
which is the result of performing a polar blowup of the corner $\mathrm{dilF}_0 \cap \mathrm{kf}_0$ of $[C; \{\infty\} \times \partial X]$.
This construction of $M$ is depicted in \Cref{fig:construction}. The notation $M/\operatorname{parF}$ indicates that this mwc is, as a topological space, the quotient resulting from collapsing $\operatorname{parF}$.


There are a number of other compactifications of $\bbR^+_t\times X^\circ$ via mwcs used in this paper.
In the next section, we use
\begin{equation}
C_1 = [0,\infty]_{t/r^2}\times X.
\end{equation}
Topologically, this is just diffeomorphic to $C$ (i.e.\ they are both cylinders), but it differs as a compactification of $\bbR_t^+\times X$.
We refer to the boundary hypersurfaces $\{t/r^2=\infty\}$, $\{r=\infty\}$, and $\{t=0\}$ of this mwc as $\mathrm{kf}$, $\mathrm{parF}_1$, $\Sigma_1$ respectively. As the notation suggests, a small neighborhood of $\mathrm{kf}$ in $C_1$ is identifiable with a neighborhood of $\mathrm{kf}$ in $M$, and the interiors of $\mathrm{parF}_1,\Sigma_1$ are identifiable with their counterparts in $\mathrm{parF},\Sigma$, respectively. An alternative construction of $M$ involves blowing up the lower corner of $C_1$ and then blowing up the new lower corner of the resultant mwc, which we call $M/\mathrm{nf}$.

Another compactification of note is $M/\mathrm{dilF}$, which is the result of blowing down $\mathrm{dilF}$ in $M$.
In \S\ref{ap:minimality}, we address the question as to whether \Cref{thm:B} holds on $M/\mathrm{f}$ for $\mathrm{f} \in \{\mathrm{nf},\mathrm{dilF}\}$. The upshot (which holds also for $\mathrm{parF}$, but requires a different argument) is that \Cref{thm:B} fails on both. (A similar analysis also applies to \Cref{thm:A}, but we do not present it.)

In the Euclidean case, we can also define $M/\mathrm{kf}$ using the coordinates $1/t\in [0,\infty)$, $x_j/\smash{t^{1/2}}\in \bbR$ near the blown-down locus, but if bound states are present then it can be immediately concluded from \Cref{thm:B} that exponential-polyhomogeneity does not hold on this compactification.
So, $M$ is minimal among mwcs on which solutions of the Schr\"odinger equation with Schwartz initial data have the desired form.

\begin{figure}[h!]
	\begin{center}
		\begin{tikzpicture}[scale=3]
		\filldraw[fill=lightgray!20] (0,0) -- (1,0) -- (1,1) -- (0,1) -- cycle;
		\node at (.5,.5) {$C$};
		\node at (.5,-.1) {$\Sigma=\{t=0\}$};
		\node at (.5,1.1) {$\{t=\infty\}$};
		\node at (-.4,.5) {$[0,\infty]_t\times \partial X$};
		\draw[->, darkred] (.05,.05) -- (.3,.05) node[above] {$1/r$};
		\draw[->, darkred] (.05,.05) -- (.05,.35) node[right] {$t$};
		\draw[->, darkred] (.95,.95) -- (.7,.95) node[below] {$1/r$};
		\draw[->, darkred] (.95,.95) -- (.95,.65) node[left] {$1/t$};
		\filldraw[darkblue] (0,1) circle (.65pt);
		\filldraw[darkblue] (1,1) circle (.65pt);
		\node[darkblue] at (-.35,1) {$\{\infty\}\times \partial X$};
		\end{tikzpicture}
		\qquad\qquad
		\begin{tikzpicture}[scale=3]
		\filldraw[fill=lightgray!20] (0,0) -- (1.5,0) -- (1.5,.55) -- (1.15,1) -- (.35,1) -- (0,.55) -- cycle;
		\node at (.75,-.1) {$\Sigma$};
		\node at (.75,.5) {$[C; \{\infty\} \times \partial X]$};
		\node at (1.6,.3) {$\mathrm{nf}$};
		\node[darkblue] at (1.5,.825) {$\mathrm{dilF}_0$};
		\node at (.75,1.075) {$\mathrm{kf}_0$};
		\draw[->, darkred] (.05,.05) -- (.3,.05) node[above] {$1/r$};
		\draw[->, darkred] (.05,.05) -- (.05,.3) node[right] {$t$};
		\draw[->, darkred] (1.45,.55) -- (1.325,.7) node[left] {$t/r$};
		\draw[->, darkred] (1.45,.55) -- (1.45,.35) node[left] {$1/t$};
		\draw[->, darkred] (.357,.95) -- (.6,.95) node[below] {$1/r$};
		\draw[->, darkred] (.357,.95) -- (.213,.75) node[right] {$r/t$};
		\filldraw[darkgreen] (1.15,1) circle (.65pt);
		\filldraw[darkgreen] (.35,1) circle (.65pt);
		\node[darkgreen] at (-.03,1.05) {$\mathrm{dilF}_0\cap \mathrm{kf}_0$};
		\end{tikzpicture}
	\end{center}
	\caption{The cylinder $C = [0,\infty]_t\times X$ and the blowup $M/\mathrm{parF}=[C; \{\infty\} \times \partial X]$ constructed in the process of constructing $M$. The submanifolds to be blown up are depicted in {\color{darkblue} blue} and {\color{darkgreen} green}.}
	\label{fig:construction}
\end{figure}
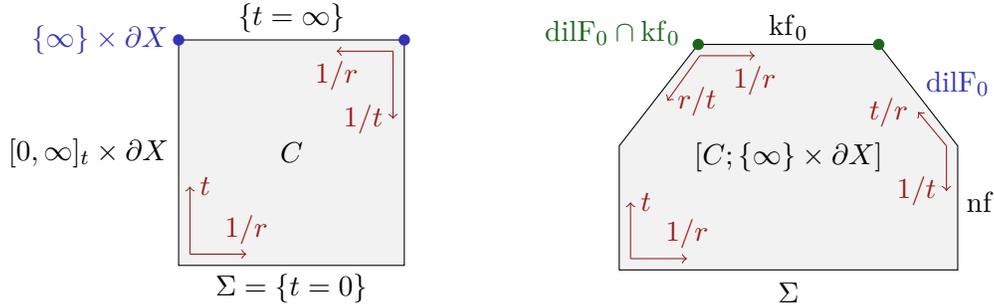

\subsection{Sketch of proof}
\label{subsec:sketch}

Consider the differential operator $P = \Delta_g + iA\cdot \nabla_g + 2^{-1} i \nabla_g\cdot A + V$.
The hypotheses are such that $P: C_{\mathrm{c}}^\infty(X^\circ) \to L^2(X,g)$
defines an essentially self-adjoint operator. The closure is the map $H^2(X,g)\to L^2(X,g)$ given by restricting $P: \calD'(X)\to \calD'(X)$, defined in the sense of distributions, to the $L^2$-based Sobolev space $H^2(X,g)$.
The spectrum $\sigma(P) = \sigma_{\mathrm{pp}}(P)\cup \sigma_{\mathrm{ac}}(P)$ of $P$ consists of finitely many negative eigenvalues and a continuous spectrum on the whole nonnegative real axis, so, if nonempty,  $\sigma_{\mathrm{pp}}(P) = \{-E_1,\cdots,-E_N\}$
for some $0<E_N \leq \cdots \leq E_1$, and $\sigma_{\mathrm{ac}}(P) = [0,\infty)$.
Note the absence of embedded eigenvalues (recall that we are assuming the nonexistence of a bound state at zero energy)  or of singular continuous spectrum. The absence of embedded eigenvalues for $E>0$ follows from \cite[Theorem 17.2.8]{Hormander} as in \cite[\S10]{MelroseSC}. That $E_N \neq 0$ is the assumption that no bound state exists at zero energy.
Here $N\in  \bbN$, with $N=0$ corresponding to the absence of pure-point spectrum.

Let $\Pi:\operatorname{Borel}(\bbR)\to \calL(L^2(X,g))$ denote the spectral measure of $P$.
So, for each Borel set $S\subset \bbR$, $\Pi(S)$ is a projection operator on $L^2(X)$.

Via the functional calculus, there exists a 1-parameter subgroup $U:t\mapsto U(t) \in \operatorname{U}(L^2(X,g))$ such that the solution $u(t,x) : \bbR_t\times X_x\to \bbC$ to \cref{eq:Schrodinger} is given by $u(t,x) = (U(t) f)(x)$,
and $U(t)$ is given by
\begin{equation}
U(t) = \int_{-\infty}^\infty e^{i E t} \dd \Pi(E) =e^{i Pt},
\end{equation}
this integral being well-defined e.g. when applied to an element of $\calS(X)$. One fact that falls out of this formalism (and the ellipticity of $P+1$) is that, for $f\in \calS(X)$, then $U(t)f \in C^\infty(\bbR_t; \calS(X))$. This already gives \Cref{thm:B} in any neighborhood of $\Sigma$ disjoint from all of the other boundary hypersurfaces of $M$ besides $\mathrm{nf}$. 

For each $n\in \{1,\ldots,N\}$, let $\phi_n$ be an $L^2(X,g)$-normalized bound state with $P\phi_n = -E_n \phi_n$, such that $\phi_1,\cdots,\phi_N$ are orthogonal. Via a standard elliptic estimate -- e.g. ellipticity in Melrose's $\operatorname{Diff}_{\mathrm{sc}}(X)$ \cite{MelroseSC} -- we have $\phi_n \in \calS(X)$ for each $n$. Let
\begin{equation}
\Pi_{\mathrm{pp}}(E) = \sum_{n=1}^N \delta(E+E_n) \phi_n  \langle \phi_n , - \rangle
\label{eq:misc_015}
\end{equation}
be the spectral projection onto the pure-point spectrum. (The Dirac-$\delta$ is $\delta(E+E_n)$ because the eigenvalue was defined as $-E_n$.)
Here, $\langle-,- \rangle : L^2(X,g)\times L^2(X,g)\to \bbC$ denotes the $L^2(X,g)$-inner product anti-linear in the first slot and linear in the second.

Stone's formula says that the spectral projection onto the continuous spectrum, $\Pi_{\mathrm{ac}}(E) = \Pi(E) - \Pi_{\mathrm{pp}}(E)$, which is supported on $[0,\infty)_E$, has Radon--Nikodym derivative given by
\begin{equation}
\dd \Pi_{\mathrm{ac}}(E) = \frac{1}{2\pi i} (R(E+i0) - R(E-i0)) \dd E,
\label{eq:misc_016}
\end{equation}
where, for each $E\notin \sigma(P)$, $R(E): L^2(X,g) \to H^2(X,g)$
denotes the resolvent $R(E) = (P-E)^{-1}$, and where, for each $E>0$, $R(E\pm i0) = \lim_{\epsilon \to 0^+} R(E\pm i \epsilon)$, these limits existing in the strong operator topology between suited weighted Sobolev spaces.

Combining \cref{eq:misc_016} and \cref{eq:misc_015},
\begin{equation}
U(t) =  \sum_{n=1}^N e^{-iE_n t} \phi_n  \langle \phi_n ,- \rangle +  \frac{1}{2\pi i} \int_0^\infty e^{i E t}  (R(E+i0) - R(E-i0)) \dd E,
\end{equation}
the integral being absolutely convergent in the strong sense. So, for any $f \in \calS(X)$,
\begin{equation}
(U(t)f)(x) =  \sum_{n=1}^N e^{-iE_n t} \phi_n(x)  \langle\phi_n , f \rangle  +  \frac{1}{2\pi i} \int_0^\infty e^{i E t}  (R(E+i0) - R(E-i0)) f(x) \dd E,
\label{eq:misc_019}
\end{equation}
where, for each $x\in X^\circ$, the integral is absolutely convergent.\footnote{One conclusion of the analysis in \S\ref{sec:sad} is that the integrand in \cref{eq:misc_019} is, for each fixed $t$ and $x$, continuous in $E$ and rapidly decaying as $E\to\infty$. This is one way to see that the integral is absolutely convergent. With a bit more work, this can be turned into a proof that \cref{eq:misc_019} gives the solution to the Cauchy problem. } 

The terms in the first sum in \cref{eq:misc_019} are easily analyzed --- see \Cref{prop:bound}. The crux of our problem is to analyze the oscillatory integral $I(t,x) = I_+(t,x) - I_-(t,x)$, where
\begin{equation}
I_\pm(t,x) = \int_0^\infty e^{i E t}  R(E\pm i0) f(x) \dd E = 2 \int_0^\infty e^{i \sigma^2 t} R(\sigma^2 \pm i 0) f(x) \sigma \dd \sigma .
\label{eq:misc_hty}
\end{equation}
(These integrals turn out to be conditionally convergent if $t\neq 0$, but not absolutely convergent because $R(E\pm i0) f(x)$ is only $O(1/E)$ as $E\to\infty$. We will quickly pass from these oscillatory integrals to ones whose integrands are Schwartz as $E\to\infty$, so this subtlety is not important.)

The key input, coming from \cite{Vasy, VasyLagrangian}\cite{HintzPrice}\cite{Full}\footnote{Unfortunately, the order employed here for listing index sets in the superscript of the spaces $\calA^{\bullet}$ is the opposite of that used in the references. Our justification for the convention here was that it is natural to list the boundary hypersurfaces from left-to-right as they are depicted in \Cref{fig:Xspres}. This also has the advantage that the least important face $\infty\mathrm{f}$ is last.} is a detailed analysis of the output $ e^{\mp i \sigma r} R(\sigma^2 \pm i0) e^{\pm i \sigma r}h(x) : \bbR^+_\sigma\times X_x\to \bbC $ of the ``conjugated (limiting) resolvent'' $ e^{\mp i \sigma r} R(\sigma^2 \pm i0) e^{\pm i \sigma r}$  for $h\in \calS(X)$ on the mwc
\begin{equation}
X^{\mathrm{sp}}_{\mathrm{res}} = [ [0,\infty]_{\sigma} \times  X ; \{0\}\times \partial X] \hookleftarrow \bbR^+_\sigma\times X.
\end{equation}
Label its faces $\mathrm{zf},\mathrm{tf},\mathrm{bf},\infty\mathrm{f}$, as in \Cref{fig:Xspres}.\footnote{The mwc $X^{\mathrm{sp}}_{\mathrm{res}}$ differs from the mwc $X^+_{\mathrm{res}}$ in \cite{HintzPrice}\cite{Full} because the former contains the face $\infty\mathrm{f}$, whereas $X^+_{\mathrm{res}}$ is non-compact.}  The analysis that we need is summarized in \S\ref{sec:sad}. So, the integrands of $I_\pm$ are well-understood.

The main piece of analysis developed within this particular paper is the production of asymptotic expansions for oscillatory integrals of the same form as $I_\pm$:
\begin{theorem}[Main lemma]
	Let $\calE,\calF,\calG$ denote index sets and $\alpha,\beta,\gamma \in \bbR\cup \{\infty\}$, and suppose that $\min\{\alpha,\Re j: (j,k)\in \calE\}>-2$.
	Let $\phi\in \calA^{(\calE,\alpha),(\calF,\beta),(\calG,\gamma),\infty}(X_{\mathrm{res}}^{\mathrm{sp}})$.
	Then, letting
	\begin{equation}
	I_\pm [\phi](t,x) =2\int_{0}^\infty e^{i \sigma^2 t \pm i \sigma r } \phi(\sigma,x) \sigma \dd \sigma : \bbR_t^+ \times X_x\to \bbC,
	\end{equation}
	we have $I_+[\phi] \in \calA^{(\calE/2+1,\alpha/2+1),(\calF+2,\beta+2),(0,0)}(C_1)$, and for any $\chi\in C_{\mathrm{c}}^\infty(\bbR)$ such that $0\notin \operatorname{supp}(1-\chi)$, a decomposition of $I_-[\phi] $ of the form $I_-[\phi] =	\exp( -i(1 -\chi(t))/(4 t \rho^2)) I_{\mathrm{osc}}[\phi] + I_{-,\mathrm{phg}}[\phi]$
	for some functions
	\begin{equation}
	I_{-,\mathrm{phg}}[\phi] \in \calA^{(\calE/2+1,\alpha/2+1),(\calF+2,\beta+2),(0,0)}(C_1)
	\label{eq:misc_025}
	\end{equation}
	and $I_{\mathrm{osc}}[\phi] \in \calA^{(\calE/2+1,\alpha+1),(\calF+2,\beta+2), (\calG+1/2,\gamma+1/2),\infty,(0,0)}(M)$.
	Here, the index sets on $C_1$ are specified in the order $\mathrm{kf}$, $\mathrm{parF}_1$, and $\Sigma_1$, respectively. The index sets on $M$ are specified in the order $\mathrm{kf}$, $\mathrm{parF}$, $\mathrm{dilF}$, $\mathrm{nf}$, and $\Sigma$.
	
	\label{thm:D}
\end{theorem}
\begin{remark*}
	The reason for the marked difference between the asymptotics of $I_-$ and $I_+$ is that we are restricting attention to $t\geq 0$. If one wants to understand the asymptotics for $t\leq 0$, the roles of $I_\pm$ are switched.
\end{remark*}
\begin{remark*}
	Note that we are assuming in \Cref{thm:D} that $\phi$ is Schwartz at $\infty\mathrm{f}$. The resolvent output $R(\sigma^2 \pm i0) f$ is certainly not rapidly decaying, let alone Schwartz, as $\sigma\to\infty$, unless $f=0$. (If $w=R(\sigma^2 \pm i0) f$ were Schwartz as $\sigma\to\infty$, then $f=(P -\sigma^2)w$ would be too, but since $f$ is independent of $\sigma$ this is only possible if $f=0$.) However, it turns out that the part of $R(\sigma^2 \pm i0) f$ that is not Schwartz at $\infty\mathrm{f}$ does not depend on the sign $\pm$, which means that its contributions to $I_-,I_+$ exactly cancel when forming the difference $I = I_+ - I_-$. It therefore does not contribute to the asymptotic analysis of the Cauchy problem. The precise argument can be found in the next section.
	
	In \cite{HintzPrice}, Hintz instead studied the inhomogeneous problem with zero initial data, for which $R(\sigma^2 \pm i0)f$ is replaced by $R(\sigma^2 \pm i0) f(E,x)$ for $f(E,x)\in \calS(\bbR_E\times X)$, $E=\sigma^2$. Then, high energy estimates easily imply that $R(\sigma^2 \pm i0) f(E,x)$ is Schwartz at $\infty\mathrm{f}$. Hintz then relates the solution of the Cauchy problem to the solution of the inhomogeneous problem.
	The analogous argument also works here. We preferred to study the Cauchy problem directly because it aids getting explicit formulas for the asymptotic profile.
	\label{rem:high_energy}
\end{remark*}

\Cref{thm:D} yields \Cref{thm:B}. This deduction is contained in \S\ref{sec:sad}. Besides some straightforward but tedious computations, one additional piece of argument is required: we must explain why the non-oscillatory contribution $I_{\mathrm{phg}}$ is Schwartz at the lower corner of $C_1$. Unfortunately, \Cref{thm:D} cannot be improved in this regard, as a numerical example in \S\ref{sec:example} illustrates. Instead, we will prove Schwartzness at the lower corner of $C_1$ in the specific context of the proof of \Cref{thm:B} using a direct argument. An alternative microlocal proof, using the PDE, is provided in \S\ref{sec:microlocal}. So, the exponential-polyhomogeneity of the resolvent output on $X_{\mathrm{res}}^{\mathrm{sp}}$ -- the conclusions of \cite{HintzPrice}\cite{Full} -- is not by itself completely sufficient to yield \Cref{thm:B}.

\begin{remark*}
	When \cref{eq:Schrodinger} is studied with initial data which is only polyhomogeneous with some finite rate of decay, $u$ cannot be Schwartz at $\mathrm{nf}$, because this would imply that the initial data is Schwartz. Our expectation in this case is that $I_{\mathrm{phg}}$ ends up not being Schwartz at the lower corner of $C_1$, but the rest of the asymptotic analysis is analogous, including the description of the resolvent output (except we get a new polyhomogeneous term at $\mathrm{bf}$ that does not depend on the sign $\pm$ in the limiting resolvent, so does not affect the asymptotic analysis of the Cauchy problem). So, there is a good reason why \Cref{thm:D} \emph{cannot} have in its conclusion that $I=I_+-I_-$ is Schwartz at $\mathrm{nf}$, and why, to prove this Schwartzness in the context of \Cref{thm:B}, we need to use the PDE and the Schwartzness of the initial data. 
\end{remark*}

We will prove \Cref{thm:D} over the course of three sections, \S\ref{sec:low}, \S\ref{sec:high},  \S\ref{sec:mid}. The idea is to write
\begin{equation}
I_\pm[\phi] = I_\pm[\chi_{\mathrm{low}} \phi]+I_\pm[\chi_{\mathrm{tf}\cap\mathrm{bf}} \phi]+I_\pm[\chi_{\mathrm{high}} \phi]
\label{eq:misc_gab}
\end{equation}
for $\chi_{\mathrm{low}}, \chi_{\mathrm{tf}\cap \mathrm{bf}} ,\chi_{\mathrm{high}} \in  C^\infty(X_{\mathrm{res}}^{\mathrm{sp}})$ a partition of unity on $X_{\mathrm{res}}^{\mathrm{sp}}$ such that the support of $\chi_{\mathrm{low}}$ is disjoint from $\mathrm{bf}\cup \infty \mathrm{f}$, hence supported at low energies and bounded $r\sigma$, the support of $\chi_{\mathrm{tf}\cap \mathrm{bf}}$ is disjoint from $\mathrm{zf}\cup \infty \mathrm{f}$, and the support of $\chi_{\mathrm{high}}$  is disjoint from $\mathrm{zf}\cup \mathrm{tf}$, hence supported at high energy.
The low energy contribution $I_{\pm}[\chi_{\mathrm{low}} \phi]$ is analyzed in \S\ref{sec:low}, the high energy contribution $I_{\pm}[\chi_{\mathrm{high}} \phi]$ is analyzed in \S\ref{sec:high}, and the final contribution $I_{\pm}[\chi_{\mathrm{tf}\cap\mathrm{bf}}\phi]$ is analyzed in \S\ref{sec:mid}. Each of the functions $\varphi = \chi_\bullet \phi$ is polyhomogeneous on a mwc with only one corner, and this allows the analysis to be reduced to the study of elementary oscillatory integrals, each of which will be proven to be of exponential-polyhomogeneous type on $M$. The mwc $M$ has four corners, there are three $\chi_\bullet$ to consider, and there are two choices of sign $\pm$, so spread throughout this paper are $4\times 3\times 2=24$ basic lemmas which establish the exponential-polyhomogeneous form of one of the $I_\pm[\chi_\bullet \phi]$ near one of the corners of $M$. (Some of these are stated together as one lemma, and some are spread over multiple lemmas. The ``24'' number should therefore not be taken literally.) Our main task is to prove these basic lemmas. Each is straightforward, but altogether their proof requires a fair amount of work.

The deduction of \Cref{thm:D} from the work in \S\ref{sec:low}, \S\ref{sec:high},  \S\ref{sec:mid} is contained in \S\ref{sec:main}. The relationship between the index sets appearing in the exponential-polyhomogeneity of $I_\pm[\phi]$ on $M$ and the index sets appearing in the exponential-polyhomogeneity of $\phi$ on $X^{\mathrm{sp}}_{\mathrm{res}}$ shows that the asymptotics of $I_\pm[\phi]$ near each corner of $M$ are closely related to the asymptotics of $\phi$ near some corner of $X^{\mathrm{sp}}_{\mathrm{res}}$. More precisely, if we ignore terms which end up being Schwartz when we apply \Cref{thm:D} to the proof of \Cref{thm:B}, then, in \Cref{thm:B}, the asymptotic expansions of our solution $u$ near each corner of $M$ can be computed from the asymptotic expansions of the resolvent output near the corresponding corner of $X^{\mathrm{sp}}_{\mathrm{res}}$. This correspondence is depicted in \Cref{fig:M_colored}. 

\begin{figure}
	\begin{tikzpicture}[scale=2.5]
	\filldraw[lightgray!20] (0,1) -- (0,.5) -- (.5,0) -- (1.5,0) -- (1.5,1) -- cycle;
	\draw (0,1) -- (0,.5) -- (.5,0) -- (1.5,0) -- (1.5,1);
	\draw[dashed] (1.5,1) -- (0,1);
	\draw[->, darkred] (1.45,.05) -- (1.45,.35) node[left] {$\rho$};
	\draw[->, darkred] (1.45,.05) -- (1.15,.05) node[above] {$1/\sigma $};
	\draw[->, darkred] (.53,.05) -- (.33,.25) node[right] {$\;\,\rho/\sigma$};
	\draw[->, darkred] (.53,.05) -- (.85,.05) node[above] {$\sigma$};
	\draw[->, darkred] (.05,.53) -- (.25,.33) node[above] {$\;\;\;\sigma/\rho$};
	\draw[->, darkred] (.05,.53) -- (.05,.73) node[right] {$\rho$};
	\node  at (.15,.15) {tf};
	\node  at (1,-.1) {bf};
	\node  at (-.1,.75) {zf};
	\node  at (1.65,.5) {$\infty$f};
	\end{tikzpicture}
	\qquad\qquad
	\begin{tikzpicture}[scale=2.5]
	\filldraw[lightgray!20] (0,1) -- (0,.5) -- (.5,0) -- (1.5,0) -- (1.5,1) -- cycle;
	\filldraw[darkgreen!20] (.75,0) -- (1.5,0) -- (1.5,1) -- (.75,1) -- cycle;
	\draw (0,1) -- (0,.5) -- (.5,0) -- (1.5,0) -- (1.5,1);
	\filldraw[darkblue!20, opacity=.5] (.35,.15) -- (1.2,1) -- (0,1) -- (0,.5) -- cycle;
	\filldraw[darkred!20, opacity=.5] (.15,.35) -- (.55,.75) -- (.95,.75) -- (.95,0) -- (.5,0) -- cycle;
	\draw (0,1) -- (0,.5) -- (.5,0) -- (1.5,0) -- (1.5,1);
	\draw[dashed] (1.5,1) -- (0,1);
	\node[color=darkred] at (.6,.25) {$\chi_{\mathrm{tf}\cap \mathrm{bf}}$};
	\node[color=darkgreen] at (1.2,.45) {$\chi_{\mathrm{high}}$};
	\node[color=darkblue] at (.25,.75) {$\chi_{\mathrm{low}}$};
	\node  at (.15,.15) {tf};
	\node  at (1,-.1) {bf};
	\node  at (-.1,.75) {zf};
	\node  at (1.65,.5) {$\infty$f};
	\end{tikzpicture}
	\caption{The mwc $X_{\mathrm{res}}^{\mathrm{sp}}$, with an atlas of coordinate charts (left), and the supports of the cutoffs ${\color{darkblue}\chi_{\mathrm{low}}},{\color{darkred} \chi_{\mathrm{tf}\cap \mathrm{bf}}},{\color{darkgreen}\chi_{\mathrm{high}}}$ (right). Since $E=\sigma^2$, ``high'' means high energy, and ``low'' means low energy (and radii).}
	\label{fig:Xspres}
\end{figure}
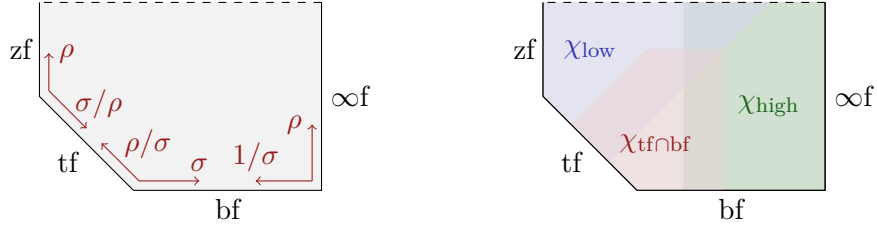

\section{Proof of \Cref{thm:B}, \Cref{thm:forced} from \Cref{thm:D}}
\label{sec:sad}

The following theorem regarding the low-energy behavior of the resolvent output is contained in the conjunction of \cite{HintzPrice}\cite{Full}:
\begin{proposition}
	There exist some index sets $\calE_0\subseteq  \bbN^{\geq 3} \times \bbN$ and $\calF_0\subseteq \bbN^{\geq 2} \times \bbN$, with $(2,1)\in \calF_0$, such that, for any $f(\sigma,x)\in C^\infty([0,\infty)_\sigma;\calS(X))$, there exist $u_0 \in \calA^{(1,0)\cup \calF_0}(X)$ and $u_1 \in \calA^{(1,1)\cup \calF_0}(X)$ (which are allowed to depend on the sign $\pm$ if $f$ depends on $\sigma$, but which are independent of the sign if $f$ is independent of $\sigma$) such that
	\begin{equation}
	e^{\mp i \sigma r} R( \sigma^2 \pm i0) e^{\pm i \sigma r}f(\sigma,x)  =  u_0(x) \pm i \sigma u_1(x)+ \phi_\pm(\sigma,x) = \varphi_\pm(\sigma,x)
	\label{eq:phi_def}
	\end{equation}
	for
	\begin{align}
	\phi_{\pm}(\sigma,x) &\in \calA^{(2,1)\cup \calE_0,\calF_0,(1,1)\cup \calF_0}_{\mathrm{loc}}(X_{\mathrm{res}}^{\mathrm{sp}}\backslash \infty\mathrm{f} ) \\
	\varphi_{\pm}(\sigma,x) &\in \calA^{(0,0)\cup (2,1)\cup \calE_0,(1,0)\cup  \calF_0,(1,0)}_{\mathrm{loc}}(X_{\mathrm{res}}^{\mathrm{sp}}\backslash \infty\mathrm{f} ).
	\end{align}
	Here, the index sets are specified at the faces $\mathrm{zf},\mathrm{tf},\mathrm{bf}$ respectively.
	In fact, $u_0(x) = P^{-1} f(0,x)$ and $u_1 = - P^{-1} L P^{-1} f(0,x) \mp i P^{-1} f'(0,x) $, where $f'(\sigma,x) = \partial_\sigma f(\sigma,x)$ and $L$ is as above.
\end{proposition}
\begin{proof}
	The properties of $\varphi_\pm$ come from the main theorem of \cite{Full}. The conormality of $\phi_\pm$ comes from \cite[Theorem 3.1]{HintzPrice}, which says $\phi_\pm \in \smash{\calA_{\mathrm{loc}}^{2-,2-,1-}(X_{\mathrm{res}}^{\mathrm{sp}}\backslash \infty\mathrm{f})}$. Polyhomogeneity, with the stated index sets, then comes from the fact that everything else in \cref{eq:phi_def} besides $\phi_\pm$ is already known to be polyhomogeneous.
\end{proof}

Strengthening this slightly:
\begin{proposition}
	Let $f\in \calS(\bbR_E\times X)$ or $f\in \calS(X)$. For any $\epsilon>0$, there exist
	\begin{align}
		\phi_{\mathrm{even}} &\in \calA^{(2,1)\cup  \calE_0,\calF_0,\infty}(X_{\mathrm{res}}^{\mathrm{sp}} \backslash \infty\mathrm{f}), \\ \varphi_{\mathrm{even}} &\in \calA^{(0,0)\cup (2,1)\cup \calE_0,(1,0)\cup  \calF_0,(1,0)}_{\mathrm{loc}}(X_{\mathrm{res}}^{\mathrm{sp}}\backslash \infty\mathrm{f})
	\end{align}
	such that
	\begin{align}
	\begin{split} 
	R(\sigma^2 \pm i0)f(E,x) &= e^{\pm i \sigma r- \epsilon \sigma^2} (u_0(x) \pm i \sigma u_1(x) ) + \phi_{\mathrm{even}}(\sigma,x)+ e^{\pm i \sigma r} \phi_\pm(\sigma,x) \\
	&= e^{\pm i \sigma r} \varphi_\pm + \varphi_{\mathrm{even}} 
	\end{split}
	\label{eq:phi_def2}
	\end{align}
	for some $\phi_\pm \in \calA^{(2,0)\cup  \calE_0,\calF_0,(1,1)\cup \calF_0,\infty}(X_{\mathrm{res}}^{\mathrm{sp}} )$ and $\varphi_\pm \in \calA^{(0,0)\cup (2,1)\cup \calE_0,(1,0)\cup  \calF_0,(1,0),\infty}_{\mathrm{loc}}(X_{\mathrm{res}}^{\mathrm{sp}})$
	(differing from the $\varphi_\pm,\phi_\pm$ in \cref{eq:phi_def}), where the final `$\infty$' means Schwartz behavior as $\sigma\to\infty$.
	Moreover, 
	\begin{align}
	u_0(x) &= P^{-1} \tilde{f}(0,x) = P^{-1}f(0,x) \\
	u_1(x) &= - P^{-1} L P^{-1} \tilde{f}(0,x) \mp i P^{-1} \tilde{f}'(0,x) =  - P^{-1} L P^{-1} f(0,x) - P^{-1} (r f(0,x)),
	\end{align}
	where $\tilde{f} = e^{\mp i \sigma r} f(E,x)$.
	\label{prop:phidef2}
\end{proposition}
Note that $u_0,u_1,\phi_{\mathrm{even}},\varphi_{\mathrm{even}}$ do not end up depending on the sign $\pm$. (This is the evenness to which the subscript ``even'' refers.)
\begin{proof}
	In order to apply \cref{eq:phi_def} to analyze the situation near $\sigma=0$ (and really just away from $\infty\mathrm{f}$), we rewrite:
	\begin{equation}
	e^{\mp i \sigma r } R(\sigma^2 \pm i0)f(E,x) = e^{\mp i \sigma r } R(\sigma^2 \pm i0)  e^{\pm i \sigma r} \tilde{f}(\sigma,x)
	\end{equation}
	for $\tilde{f}(\sigma,x) \in C^\infty([0,\infty)_\sigma; \calS(X))$ given by $\tilde{f}(\sigma,x) = e^{\mp i \sigma r} f(E,x)$.
	We can then apply \cref{eq:phi_def} with $\tilde{f}$ in place of $f$. This immediately gives the portion of the proposition involving $\varphi$, except for the behavior at high energy.

	Because of the essential self-adjointness of $P$, the $\sigma\to 0^+$ expansion of $R(\sigma^2 \pm i 0) f(E,x)$ at $\mathrm{zf}^\circ$ must have the form
	\begin{equation}
	R(\sigma^2 \pm i 0) f(E,x) \sim \sum_{(j,k)\in (0,0) \cup (2,1) \cup \calE_0} (\pm i \sigma )^j \log^k(\pm i\sigma ) \phi_{j,k}(x)
	\end{equation}
	for some $\phi_{j,k} \in C^\infty(\mathrm{zf}^\circ)$, where the key point is that $\phi_{j,k}$ does not depend on the choice of sign. See \cite[p. 27]{HintzPrice}. The $(j,k)=(2,1)$ term has the form $- \sigma^2 (\log(\sigma) \pm i \pi/2 ) \phi_{2,1}(x)$. Let $\psi \in C^\infty(X^{\mathrm{sp}}_{\mathrm{res}})$ be a cutoff Schwartz at $\mathrm{bf}\cup\infty\mathrm{f}$ and identically $=1$ near $\mathrm{zf}$. Then, we can take $\phi_{\mathrm{even}} = -\sigma^2 \log(\sigma) \psi \phi_{2,1}(x)$ near $\sigma=0$. The properties of $\phi_\pm$ at low energy then follow.

	The formulas for $u_0,u_1$ follow from the corresponding formulas for the terms in \cref{eq:phi_def}. This completes the analysis of the low-energy behavior.

	If $f\in \calS(\bbR_E\times X)$, the usual high-energy bounds \cite[Lemma 2.10]{HintzPrice} imply that
	\begin{equation}
		e^{\mp i \sigma r } R(\sigma^2 \pm i0)f(E,x)\in \calA^{(1,0),\infty}_{\mathrm{loc}}( (0,\infty]_\sigma \times X_x),
		\label{eq:high_energy}
	\end{equation}
	where the `$\infty$' denotes Schwartz behavior as $\sigma \to \infty$. So, \cref{eq:phi_def2} also holds away from $\sigma=0$ if we take
	\begin{equation}
	e^{\pm i \sigma r} \phi_\pm = R(\sigma^2 \pm i0) f(E,x) -  e^{\pm i \sigma r- \epsilon \sigma^2} (u_0(x) \pm i \sigma u_1(x) ) - \phi_{\mathrm{even}}(\sigma,x)
	\end{equation}
	as a global definition, and the thusly defined $\phi_\pm$ lies in the desired function spaces. Likewise, if we define  
	\begin{equation}
	\varphi_\pm = e^{\mp i \sigma r} R(\sigma^2 \pm i 0)f(E,x), 
	\end{equation}
	then the high energy estimates show that $\varphi_\pm$ lie in the desired function spaces. 
	So, in this case, we can take $\varphi_{\mathrm{even}} = 0$. To summarize, because $f$ is Schwartz as $E\to\infty$, even the weak high-energy bounds in \cite[Lemma 2.10]{HintzPrice} suffice to show that the resolvent output has Schwartz behavior at high energy.
	
	If $f\in \calS(X)$, then, viewing $f$ as a function on $X\times \bbR_\sigma$, its large-$\sigma$ semiclassical wavefront set lies in the zero section of the semiclassical cotangent bundle, where $P-\sigma^2$ is elliptic, which means (via the semiclassical parametrix) that there exists some $v \in \langle \sigma \rangle^{-2} L^\infty(\bbR_\sigma;\calS(X))$ such that $(P-\sigma^2)v - f \in \calS(\bbR_E\times X)$. Then, 
	\begin{equation} 
		R(\sigma^2 \pm i0)f = v + R(\sigma^2 \pm i0) g
	\end{equation} 
	for $g= f-(P-\sigma^2)v  \in \calS(\bbR_E\times X)$, as follows via the characterization of the output of the limiting resolvent in terms of the function spaces it lies in. We already know (by the previous paragraph) that $R(\sigma^2 \pm i0) g$ has the desired form, and $v$ does not depend on $\pm$, so it can be absorbed into $\phi_{\mathrm{even}},\varphi_{\mathrm{even}}$. 
\end{proof}

So, in the expansion of the spectral projection $(R(\sigma^2 +i0) - R(\sigma^2-i0))f(E,x)$, all of the terms with even $j$ up to the first logarithmic term cancel, and we even get some cancellation involving the logarithmic part of that term:
\begin{equation}
(R(\sigma^2 +i0) - R(\sigma^2-i0))f(E,x) \sim 2 i \sigma  \phi_{1,0}(x) - \pi i \sigma^2\phi_{2,1}(x) + \sum_{(j,k)\in \calE_0 } \sigma^j (\log \sigma)^k \tilde{\phi}_{j,k}(x)
\end{equation}
for some $\tilde{\phi}_{j,k} \in C^\infty(\mathrm{zf}^\circ)$.

We are now in a position to prove \Cref{thm:B}. So, let $u$ be as in \cref{eq:Schrodinger}. By \cref{eq:misc_019}, its absolutely continuous part is given by $(2\pi i)^{-1}I(t,x)$, where $I=I_+-I_-$.
The integral $I$ in \cref{eq:misc_hty} is given by $I(t,x) = I_+[\varphi_+] - I_-[\varphi_-]$,
where
\begin{equation}
\varphi_\pm \in \calA^{(0,0)\cup (2,1)\cup \calE_0,(1,0)\cup \calF_0,(1,0),\infty}(X_{\mathrm{res}}^{\mathrm{sp}})
\end{equation}
are as in \cref{eq:phi_def2}. Note that $\varphi_{\mathrm{even}}$ does not contribute.

Each $I_\pm[\varphi_\pm]$  has the form described by our main lemma, \Cref{thm:D}, and it turns out that the combination $I_{\mathrm{phg}}=I_+[\varphi_+]-I_{-,\mathrm{phg}}[\varphi_-]$ must be Schwartz at $\mathrm{dilF}\cup \mathrm{nf}$. We provide a proof using microlocal tools in \S\ref{sec:microlocal}, but the simplest way to see this is that each term in the Taylor series of $I_{\mathrm{phg}}$ at $\Sigma=\{t=0\}$ differs from that of the solution $u$ by a Schwartz function (coming from $I_{\mathrm{osc}}$, which is stated in \Cref{thm:D} to be Schwartz at nf, and from the sum over bound states in \cref{eq:misc_loo}, in which each bound state is Schwartz at $\mathrm{nf}$). On the other hand, the initial data is Schwartz, so \cref{eq:Schrodinger} implies that each term in the Taylor series of $u$ at $\Sigma$ is Schwartz. So, each term in the expansion of $I_{\mathrm{phg}}$ at $\Sigma$ is Schwartz. Since $I_{\mathrm{phg}}$ is polyhomogeneous already on $C_1$, this implies Schwartzness at $\mathrm{dilF}\cup \mathrm{nf}$ when viewed as a function on $M$.
So, in fact
\begin{multline}
e^{ir^2 (1-\chi(t))/(4 t)} I(t,x) \in  \calA^{(1,0)\cup (3/2,0)\cup(2,1)\cup (1+\calE_0/2),(3,0)\cup(\calF_0+2),(3/2,0),\infty,(0,0)}(M) \\ \subseteq (t+i\epsilon)^{-3/2} \calA^{(-1/2,0)\cup (0,0)\cup (1/2,1)\cup \calE,(0,0)\cup \calF,(0,0),\infty,(0,0)}(M)
\label{eq:misc_hen}
\end{multline}
where $\calE = (1,0) \cup(-1/2+ \calE_0/2) \subset (2^{-1}\bbN^{\geq 2})\times \bbN$ and $\calF = \calF_0-1 \subset \bbN^{\geq 1} \times \bbN $.

So, combining \cref{eq:misc_019} and \cref{eq:misc_hen}, we get the desired \cref{eq:misc_loo} for
\begin{equation}
u_{\mathrm{phg}} \in \calA^{(-1/2,0)\cup (0,0)\cup(1/2,1)\cup \calE, (0,0)\cup \calF,(0,0),\infty,(0,0)}(M)
\label{eq:misc_yoy}
\end{equation}
defined by $u_{\mathrm{phg}} = (t+i\epsilon)^{3/2} e^{ir^2 (1-\chi(t))/(4 t)} I(t,x)/(2\pi i)$.

The rest of \Cref{thm:B} requires a slightly more refined analysis of $u_{\mathrm{phg}}$. To this end, write
\begin{equation}
I(t,x) = 2 \int_{-\infty}^{\infty}e^{i \sigma^2 t + i \sigma r - \epsilon \sigma^2} (u_0(x) + i\sigma u_1(x)) \sigma \dd \sigma  + I_+[\phi_+] - I_-[\phi_-],
\label{eq:misc_02c}
\end{equation}
where
$\phi_\pm \in \calA^{(2,0)\cup \calE_0,\calF_0,(1,1) \cup \calF_0,\infty}(X_{\mathrm{res}}^{\mathrm{sp}})$ are as in \cref{eq:phi_def2}.
Note that the $\phi_{\mathrm{even}}$ term in \cref{eq:phi_def2} does not contribute.

The first term in \cref{eq:misc_02c} is explicitly computable:
\begin{multline}
\exp \Big( \frac{i r^2 }{ 4 (t+i \epsilon )} \Big)
\int_{-\infty}^{\infty}
e^{i \sigma^2 t + i \sigma r - \epsilon \sigma^2}
(u_0(x) + i \sigma u_1(x)) \sigma \dd \sigma  
=
-\sqrt{\pi i}
\frac{r u_0(x)}{2 (t+i\epsilon)^{3/2}} \\ 
+\sqrt{\pi i} 
\frac{i(r^2 + 2 i t - 2 \epsilon )u_1(x)}
{4 (t+i\epsilon)^{5/2}} \in (t+i\epsilon)^{-3/2} \calA^{(0,0),(0,0)\cup (1,1)\cup \calF_0,(0,1) \cup \calF,(-1,1)\cup(\calF-1),(0,0)}(M) .
\label{eq:misc_02v}
\end{multline}

Applying \Cref{thm:D} to $I[\phi] = I_+[\phi_+] - I_-[\phi_-]$, the conclusion is that $I[\phi] =	\exp( -i(1 -\chi(t))/(4 t \rho^2)) I_{\mathrm{osc}}[\phi] + I_{\mathrm{phg}}[\phi]$
for
\begin{multline}
I_{\mathrm{osc}}[\phi]\in  \calA^{(2,0) \cup (5/2,0)\cup (1+\calE_0/2),\calF_0+2,(3/2,1) \cup (\calF_0+1/2),\infty,(0,0)}(M) \\
= (t+i\epsilon)^{-3/2} \calA^{(1/2,0)\cup \calE,\calF,(0,1)\cup \calF,\infty,(0,0)}(M)
\label{eq:misc_0gg}
\end{multline}
and some $I_{\mathrm{phg}}[\phi] \in (t+i\epsilon)^{-3/2} \calA^{(1/2,0)\cup \calE,\calF,(0,0)}(C_1)$, which must be Schwartz at the low corner.  Combining \cref{eq:misc_02c}, \cref{eq:misc_02v}, and \cref{eq:misc_0gg}, the result is that
\begin{equation}
	u_{\mathrm{phg}} \in \calA^{(0,0)\cup (1/2,0)\cup \calE, (0,0)\cup  \calF,(0,1)\cup \calF,(-1,1)\cup (\calF-1),(0,0)}(M).
\end{equation}
Combining this with \cref{eq:misc_yoy}, we get
\begin{equation}
u_{\mathrm{phg}} \in \calA^{(0,0)\cup (1/2,0)\cup \calE, (0,0)\cup  \calF,(0,0),\infty,(0,0)}(M),
\label{eq:misc_brk}
\end{equation}
which was what was claimed in \Cref{thm:B}.

We need to verify that $u_{\mathrm{phg}}$ has the claimed behavior at $\mathrm{parF}\cup \mathrm{kf}$.
The analysis above showed that
\begin{multline}
u_{\mathrm{phg}} = -\frac{\sqrt{\pi i}}{2\pi i}  \exp \Big( \frac{i (1-\chi(t)) r^2}{4t} - \frac{ir^2}{4(t+i\epsilon)} \Big) \Big[ ru_0(x)   - \frac{i(r^2 + 2 i t - 2 \epsilon )u_1(x)}{2 (t+i\epsilon)} \Big] \\ + \calA^{(1/2,0)\cup \calE,\calF,(0,1)\cup \calF,(-1,1)\cup(\calF-1),(0,0)}(M) .
\label{eq:misc_444}
\end{multline}
This can be simplified using that (I)
\begin{equation}
r^2 (t+i\epsilon)^{-1} u_1(x) \in \calA^{(1,0),(-1,1)\cup (\calF-1),(0,0)}(C)\subseteq \calA^{\calE,(1,1)\cup (\calF+1),(0,1)\cup \calF,(-1,1)\cup (\calF-1),(0,0)}(M),
\end{equation}
(II) $(1-\chi(r/t)) ru_0(x) \in  \calA^{\infty,\infty,(0,0)\cup \calF,(0,0)\cup \calF,(0,0)}(M)$ and
\begin{equation}
	(1-\chi(r^2/t)) u_1(x) \in \calA^{\infty,(1,1)\cup \calF_0,(1,1)\cup \calF_0,(1,1)\cup \calF_0, (0,0)}(M),
\end{equation}
and (III) the exponential in \cref{eq:misc_444} fails to be smooth only at $\mathrm{nf}$:
\begin{equation}
	\exp \Big( \frac{i (1-\chi(t)) r^2}{4t} - \frac{ir^2}{4(t+i\epsilon)} \Big) \in C^\infty(M\backslash \mathrm{nf}).
\end{equation}
Combining these observations with \cref{eq:misc_444} yields
\begin{equation}
u_{\mathrm{phg}} = -\frac{\sqrt{\pi i}}{2\pi i}  ( \chi(r/t)ru_0 + \chi(r^2/t)  u_1(x))+ \calA^{(1/2,0)\cup \calE,\calF,(0,1)\cup \calF,(-1,1)\cup (\calF-1),(0,0)}(M\backslash \mathrm{nf}),
\label{eq:misc_lpp}
\end{equation}
where, since we are removing $\mathrm{nf}$ from consideration, the index set $(-1,1)\cup (\calF-1)$ does not signify anything.

Since $\chi(r/t)r u_0(x) \in \calA^{(0,0),(0,0)\cup \calF, (0,0)\cup \calF,\infty,\infty}(M) $ and $ \chi(r^2/t) u_1(x) \in \calA^{(0,0),(1,1) \cup \calF,\infty,\infty,\infty}(M)$, combining \cref{eq:misc_lpp} with \cref{eq:misc_brk} shows that the error term in \cref{eq:misc_lpp} lies in
\begin{multline}
\calA^{(1/2,0)\cup \calE,\calF,(0,1) \cup \calF,(-1,1) \cup (\calF-1),(0,0)}(M\backslash \mathrm{nf})   \cap  \big(\calA^{(0,0)\cup (1/2,0)\cup \calE, (0,0) \cup  \calF,(0,0),\infty,(0,0)}(M) \\ +\calA^{(0,0),(0,0)\cup \calF, (0,0)\cup \calF,\infty,\infty}(M)  +  \calA^{(0,0),(1,1)\cup \calF,\infty,\infty,\infty}(M)\big)\\
=\calA^{(1/2,0) \cup \calE, \calF,(0,0)\cup \calF,\infty,(0,0)}(M).
\end{multline}
We have therefore improved \cref{eq:misc_lpp} to \cref{eq:lo}, which finishes the deduction of \Cref{thm:B}. 

\begin{remark*}
	Below, we will compute, for all $\phi_\pm$ as in \Cref{thm:D},
	full asymptotic expansions of $I_{\mathrm{phg}}[\phi]$ at $\mathrm{dilF}\cup \mathrm{nf}$.
	As already mentioned, it is not always the case, in this generality, that $I_{\mathrm{phg}}[\phi]$ is Schwartz there --- see \S\ref{sec:example}.
	It may therefore seem a bit miraculous that, as stated above, Schwartzness does hold when $\phi_\pm(\sigma,x) = R(\sigma^2 \pm i 0)f(x)$ for $f\in \calS(X)$. In principle, it should be possible to prove this fact by verifying that all of the terms in the expansions below vanish at $\mathrm{dilF}\cup \mathrm{nf}$. Here (in this section and in \S\ref{sec:microlocal}), we have provided two alternative proofs.
\end{remark*}

The proof of \Cref{thm:forced} is completely analogous, except one has to add to $f(x)$ a term $g(E,x)\in \calS(\bbR_E\times X)$ that is the inverse Fourier transform of the forcing $F(t,x)\in \calS(\bbR_t\times X)$. All the lemmas above were written in the level of generality needed to handle this.

\begin{figure}
	\begin{center}
		\begin{tikzpicture}[scale=3]
		\filldraw[fill=lightgray!20] (-1,.5) -- (-.75,0) -- (.75,0) -- (1,.5) -- (.66,1) -- plot[domain=0:180] ({.25*cos(\x)}, {1-.25*sin(\x)}) -- (-.66,1) -- (-1,.5);
		\filldraw[darkgreen!20] (-.83,.75)-- (-1,.5) -- (-.75,0) -- (.75,0) -- (1,.5) -- (.83,.75) to[out=200, in=-20] cycle;
		\fill[darkred!30, fill opacity=.5] (.91,.63) to[out=200, in=-20] (-.91,.63) -- (-.66,1) -- (-.375,1) to[out=270, in=180] (0,.65) to[out=0,in=270]  (.375,1) -- (.66,1) -- cycle;
		\fill[darkblue!30, fill opacity=.5] (.5,1) -- plot[domain=0:180] ({.25*cos(\x)}, {1-.25*sin(\x)}) -- (-.5,1) to[out=-90, in=180] (0,.5) to[out=0, in=-90] cycle;
		\draw(-1,.5) -- (-.75,0) -- (.75,0) -- (1,.5) -- (.66,1) -- plot[domain=0:180] ({.25*cos(\x)}, {1-.25*sin(\x)}) -- (-.66,1) -- (-1,.5);
		\node at (1,.25) {$\mathrm{nf}$};
		\node at (.98,.8) {$\mathrm{dilF}$};
		\node at (.475,1.1) {$\mathrm{parF}$};
		\node at (0,.85) {$\mathrm{kf}$};
		\node at (0,-.1) {$\Sigma$};
		\node[darkgreen] at (0,.2) {from $\chi_{\mathrm{high}}$};
		\draw[dashed, darkred] (-.7,.8) -- (-1,.8) node[darkred, left] {from $\chi_{\mathrm{tf}\cap \mathrm{bf}}$};
		\draw[dashed, darkblue] (-.3,.9) -- (-.3,1.1) node[darkblue, above] {from $\chi_{\mathrm{low}}$};
		\end{tikzpicture}
	\end{center}
	\caption{The mwc $M$ again, but with a covering indicating the regions in which asymptotics of $u$ are governed by the asymptotics of $R(\sigma^2 \pm i0)f$ in the support of  ${\color{darkblue}\chi_{\mathrm{low}}},{\color{darkred} \chi_{\mathrm{tf}\cap \mathrm{bf}}},{\color{darkgreen}\chi_{\mathrm{high}}}$, respectively. Cf.\ \Cref{fig:Xspres}. }
	\label{fig:M_colored}
\end{figure}

\section{Low energy contribution}
\label{sec:low}

We now analyze $I_\pm[\phi](t,x) = 2\int_0^\infty e^{i \sigma^2 t \pm i \sigma / \rho(x) } \phi(\sigma,x) \sigma \dd \sigma$ for $\phi$ polyhomogeneous on $X_{\mathrm{res}}^{\mathrm{sp}}$ and supported away from $\mathrm{bf}\cup \infty\mathrm{f}$. This therefore constitutes the \emph{low energy} contribution to our overall integrals.
We begin with a few geometric preliminaries.
Let $\lambda = \sigma /\rho(x)$, so that the map $\bbR^+ \times X^\circ \ni (\sigma,x) \mapsto (\lambda,x) \in \bbR^+\times X^\circ$ extends to a diffeomorphism
\begin{equation}
  \iota:X_{\mathrm{res}}^{\mathrm{sp}} \backslash (\mathrm{bf} \cup \infty\mathrm{f})  \to [0,\infty)_\lambda \times X.
\end{equation}
Let $\varphi = \phi \circ \iota^{-1} \in C^\infty((0,\infty)_\lambda \times X^\circ_x)$. Then, $\phi$ being supported away from  $\mathrm{bf}\cup \infty\mathrm{f}$ is equivalent to $\operatorname{supp} \varphi \Subset [0,\infty)_\lambda \times X$. That is, $\varphi(\lambda,-)$ vanishes identically if $\lambda$ is sufficiently large. In terms of $\varphi$,
\begin{equation}
	2^{-1} I_\pm[\phi](t,x) = \int_0^\infty e^{i \sigma^2 t \pm i \sigma /\rho(x) } \varphi(\sigma / \rho(x),x)  \sigma  \dd \sigma  = \rho(x)^2 \int_0^\infty e^{i \lambda^2 t\rho(x)^2 \pm i \lambda}\varphi(\lambda,x) \lambda \dd \lambda.
	\label{eq:34b}
\end{equation}

In order to express the spacetime asymptotics of $I_\pm[\phi]$, it is convenient to work with the compactification
\begin{equation}
	C_1 = [0,\infty]_\tau \times X_x\hookleftarrow \bbR_t^+\times X^\circ_x
\end{equation}
defined using $\tau = t \rho(x)^2$. As mwcs, $C_1\cong C$, but these differ as compactifications of spacetime. The three boundary hypersurfaces of $C_1$ are $\smash{[0,\infty]_\tau} \times \partial X$, $\{ \infty\}\times X_x$, and $\{0\}\times X_x$. Note that
\begin{equation}
	M/ \mathrm{nf}\cong  [C_1 ; \{0\}\times \partial X],
\end{equation}
which is a precise way of saying that $C_1$ results from $M$ by blowing down both $\mathrm{nf}$ and $\mathrm{dilF}$. This blowdown identifies $\mathrm{kf}$ with $\{\infty\}_\tau\times X$, $\mathrm{parF}\backslash \mathrm{dilF}$ with $(0,\infty]_\tau \times \partial X$, $\Sigma \backslash \mathrm{nf}$ with $\{0\}\times X^\circ$, and maps $\mathrm{nf}\cup \mathrm{dilF}$ to the corner $\{0\}_\tau\times \partial X$.

So, in order to specify asymptotics of $I_\pm[\phi]$ on $M$, it suffices to specify them on $C_1$.
The main proposition of this section, most of the details of the proof of which are relegated to \Cref{prop:fund}, below, reads:
\begin{proposition}
	Suppose that $\varphi(\lambda,x) \in \calA^{(\calE,\alpha)}_{\mathrm{c}} ( [0,\infty)_\lambda; \calA^{(\calF,\beta)}(X_x))$ for some index set $\calE\subseteq \{z\in \bbC: \Re z > -2\}\times \bbN$, $\alpha \in \bbR^{>-2} \cup \{\infty\}$, index set $\calF\subset \bbC\times \bbN$, and $\beta \in \bbR\cup \{\infty\} $. Then,
	\begin{equation}
		I_\pm[\phi](t,x) \in \calA^{(\calE/2+1,\alpha/2+1),(\calF+2,\beta+2),(0,0)}(C_1),
	\end{equation}
	where $\calE/2+1$ is the index set at $\{\infty\}_\tau\times X$, $\calF+2$ is the index set at $[0,\infty]_\tau \times \partial X$, and $(0,0)$ is the index set at $\{0\}\times X$.
	\label{prop:low}
\end{proposition}
See below for notational conventions regarding the Fourier transform.

\begin{remark}
	The proof shows that the expansion of $I_\pm[\phi]$ at $[0,\infty]_\tau \times \partial X$, i.e.\ as $r\to\infty$, is just
	\begin{equation}
	I_\pm[\phi](t,x) \sim \sum_{(j,k)\in \calF, \Re j \leq \beta } \rho(x)^{j+2} \log^k \rho(x) \calF_{\xi \to \tau}(\Theta(\xi) e^{\pm i \xi^{1/2} } \varphi_{j,k}(\xi^{1/2},\theta) ) (\tau),
	\label{eq:misc_029}
	\end{equation}
	where $\varphi_{j,k}(\lambda,\theta) \in \calA_{\mathrm{c}}^{(\calE,\alpha)}([0,\infty)_\lambda\times \partial X_\theta)$ are the coefficients in the polyhomogeneous expansion of $\varphi(\lambda,x)$ at $[0,\infty)_\lambda \times \partial X$, i.e.\ as $x\to\partial X$ .

	Similarly, if we let $\varphi^{j,k}(x) \in \calA^{(\calF,\beta)}(X_x)$ denote the coefficients in the $\lambda\to 0^+$ expansion of $\varphi(\lambda,x)$, then the $\tau\to \infty$ expansion of $I_\pm[\phi]$ is
	\begin{equation}
	I_\pm[\phi](t,x)
	\sim
	\rho(x)^2
	\sum_{(j,k)\in \calE,\ \Re j\leq \alpha}
	\Bigg[
	\sum_{\substack{j_0\geq 0,\ K\geq k\\ (j-j_0,K)\in \calE}}
	\frac{(\pm i)^{j_0}}{j_0!\,2^K}
	\varphi^{j-j_0,K}(x)
	c_{j/2,K,k}
	\Bigg]
	\tau^{-j/2-1}\log^k\tau .
	\label{eq:misc_jeb}
	\end{equation}
	If $\lambda^j \log^k(\lambda) \varphi^{j,k}(x)$ denotes the leading term in the $\lambda\to 0^+$ expansion of $\varphi(\lambda,x)$, then the leading term in the $\tau\to \infty$ expansion of $I_\pm[\phi]$ is given by $I_\pm[\phi] \sim \rho(x)^2 2^{-k}\varphi^{j,k} i^{1+j/2} (-1)^k \Gamma(1+j/2) \tau^{-j/2-1} \log^k(\tau)$.
	\label{rem:Ilow}
\end{remark}
\begin{proof}
	Rewriting the integral in terms of $\xi = \lambda^2= \sigma^2 / \rho(x)^{2}$, we have $I_\pm(t,x)  = \tilde{I}_\pm(t\rho(x)^2,x)$ for
	\begin{equation}
	\tilde{I}_\pm(\tau,x) = \rho(x)^2 \int_0^\infty e^{i \xi \tau} \tilde{\varphi}_\pm(\xi,x)  \dd \xi = \rho(x)^2 \calF_{\xi\to \tau}(\Theta(\xi) \tilde{\varphi}_\pm(\xi,x))(\tau),
	\label{eq:misc_026}
	\end{equation}
	where $\tilde{\varphi}_\pm(\xi,x) = e^{\pm i \xi^{1/2}} \varphi(\xi^{1/2},x)$. Because $\varphi(\lambda,x) \in \calA^{(\calE,\alpha)}_{\mathrm{c}} ( [0,\infty)_\lambda; \calA^{(\calF,\beta)}(X_x))$, we have
	\begin{equation}
	\tilde{\varphi}_\pm(\xi,x) \in \calA^{(\calE/2,\alpha/2)}_{\mathrm{c}}([0,\infty)_\xi; \calA^{(\calF,\beta)}(X_x)).
	\end{equation}
	So, via \Cref{prop:fund}, the Fourier transform on the right-hand side of \cref{eq:misc_026} lies in the function space $\calA^{(\calE/2+1,\alpha/2+1)}(\overline{\bbR}_\tau; \calA^{(\calF,\beta)}(X_x)) $.

	The form of the expansions given follow from \Cref{prop:fund}. Indeed, the large-$\tau$ expansions are stated as part of that proposition: letting $\tilde{\varphi}_{\pm;j/2,k} \in \calA^{(\calF,\beta)}(X)$ denote the coefficients in the expansion of $\smash{e^{\pm i \xi^{1/2}}\varphi(\xi^{1/2},x)}$ as $\xi\to 0^+$, then the expansion of $\rho(x)^{-2} I_\pm[\phi]$ at $\{\infty\}_\tau\times X$ is given by
	\begin{equation}
		\rho(x)^{-2} I_\pm[\phi](t,x) \sim  \sum_{(j,k)\in \calE, \Re j \leq \alpha } |\tau|^{-j/2-1} \log^k |\tau| \Big[ \sum_{K\geq k\text { s.t. }(j,K)\in \calE} \tilde{\varphi}_{\pm,j/2,K}(x) c_{j/2,K,k} \Big] ,
		\label{eq:kf_exp}
	\end{equation}
	where the $c_\bullet = c_{\bullet;+}$'s are given by \cref{eq:cdef} below. 
	Since
	\[
	\tilde{\varphi}_{\pm;j/2,K}(x)
	=
	\sum_{\substack{j_0=0\\(j-j_0,K)\in\calE}}^\infty
	\frac{(\pm i)^{j_0}}{j_0!\,2^K}
	\varphi^{j-j_0,K}(x),
	\]
	\cref{eq:misc_jeb} follows.  Note that this sum is finite, since $\varphi^{j-j_0,K}$ vanishes if $j_0$ is too large.

	Also, letting $\varphi_{j,k}(\lambda) \in \calA_{\mathrm{c}}^{(\calE,\alpha)}([0,\infty)_\lambda\times \partial X_\theta)$ be the coefficients in the polyhomogeneous expansion of $\varphi(\lambda,x)$ at $[0,\infty)_\lambda \times \partial X$, and defining $\varphi_\gamma$ by
	\begin{equation}
		\varphi(\lambda,x) = \varphi_{\gamma}(\lambda,x)+ \sum_{(j,k)\in \calF, \Re j \leq \gamma} \rho(x)^j \log^k \rho(x) \varphi_{j,k}(\lambda,\theta),
	\end{equation}
	we have $\varphi_{\gamma}(\lambda,x) \in  \calA^{(\calE,\alpha)}_{\mathrm{c}} ( [0,\infty)_\lambda; \calA^{\gamma}(X_x))$. \Cref{prop:fund} then says that the error in truncating the expansion in \cref{eq:misc_029} to order $\gamma$ lies in $\calA^{(\calE/2+1,\alpha/2+1),\gamma+2,(0,0)}(C_1)$,
	where the $\gamma+2$ is the order at $[0,\infty]_\tau \times \partial X$. Since $\gamma$ can be any real number $\leq \beta$, we conclude that \cref{eq:misc_029} holds.
\end{proof}

\subsection{Fourier transforms of polyhomogeneous functions on the half-line}

Our convention for the Fourier transform $\calF:\calS'(\bbR)\to \calS'(\bbR)$ is
\begin{equation}
\calF  \phi(\tau) = \int_{-\infty}^{+\infty} e^{i\xi \tau}   \phi(\xi) \dd \xi .
\end{equation}
We will also write $\calF   \phi(\tau)$ as $\calF_{\xi\to \tau}(  \phi(\xi))(\tau)$ when it is useful to name the dual variable, `$\xi$' in this case.

The main proposition of this subsection is:
\begin{proposition}
	Suppose that $\calX = \bigcap_{n\in \bbN} \calX_n$ is an intersection of Banach spaces over $\bbC$, with $\calX_{n+1}\subseteq \calX_n$ continuously.
	Fix $\alpha \in (-1,\infty) \cup \{\infty\}$ and an index set $\calE \subset \{z\in \bbC: \Re z >-1\} \times \bbN$, so that
	\begin{equation}
	\calA^{(\calE,\alpha)}_{\mathrm{c}}([0,\infty);\calX) \subseteq L^1(\bbR;\calX). \label{eq:E_28}
	\end{equation}
	Then, if $  \phi \in \calA^{(\calE,\alpha)}_{\mathrm{c}}([0,\infty);\calX)$,
	the Fourier transform $\calF   \phi$ satisfies $\calF   \phi(\tau) \in \calA^{(\calE+1,\alpha+1) }(\overline{\bbR}_\tau;\calX)$. Moreover, if $  \phi_{j,k}\in \calX$ are the coefficients in the polyhomogeneous expansion
	\begin{equation}
	  \phi(\xi) \sim  \sum_{(j,k)\in \calE ,\Re j \leq \alpha}   \phi_{j,k} \xi^j \log^k \xi
	\end{equation}
	of $  \phi(\xi)$ as $\xi\to 0^+$, then
	\begin{equation}
	\calF   \phi(\tau) \sim \sum_{(j,k)\in \calE, \Re j \leq \alpha }  \Big[ \sum_{K\geq k \text{ s.t } (j,K)\in \calE}   \phi_{j,K} c_{j,K,k} \Big] |\tau|^{-j-1} \log^k |\tau|
	\label{eq:misc_044}
	\end{equation}
	is the polyhomogeneous expansion of $\calF\phi$ as $\tau\to \pm \infty$, where the $c_\bullet$'s are given by \cref{eq:cdef}.
	\label{prop:fund}
\end{proposition}

Cf.\ \cite[Cor.\ 2.26]{Hintz3b}, which deals with polyhomogeneous functions on the compactified half-line.

Recall that the `c' subscript means that $  \phi \in
\calA^{(\calE,\alpha)}_{\mathrm{c}}([0,\infty);\calX)$ implies that there exists some $\xi_0>0$ such that $  \phi (\xi)=0$ for all $\xi\geq \xi_0$.

For $  \phi \in \calA^{(\calE,\alpha)}_{\mathrm{c}}[0,\infty)$, we let $  \phi(-\xi)=0$ for $\xi>0$, this being implicit in \cref{eq:E_28}.

\begin{proof}
	For simplicity, we prove the claim when $\calX=\bbC$. The general case is completely analogous, as the assumptions on $\calE$, as well as the compact support of $\phi$, guarantee Bochner integrability throughout the argument.

	Let $\beta \in (-1,\infty)$ satisfy $\beta\leq \alpha$. (If $\alpha<\infty$, then there is no reason not to take $\beta=\alpha$.) We can write
	\begin{equation}
	  \phi(\xi) =   \phi^{(\beta)}(\xi)+ \sum_{(j,k)\in \calE, \, \Re j \leq \beta }   \phi_{j,k} \xi^j \log^k \xi
	\label{eq:misc_031}
	\end{equation}
	for $  \phi_{j,k}\in \bbC$ which do not depend on $\beta$, where $  \phi^{(\beta)} \in \calA^{\beta}_{\mathrm{loc}}([0,\infty))$.
	Because $\calE$ is an index set, the sum here is finite. (In the future, we will simply use that sums of this form are finite without stating so explicitly.)

	Let $\chi \in C_{\mathrm{c}}^\infty(\bbR)$ equal $1$ identically on a neighborhood of $\{0\}\cup \operatorname{supp}   \phi$. Then,
	\begin{equation}
	\calF   \phi(\tau) = \int_0^\infty e^{i \xi \tau}    \chi(\xi)  \phi^{(\beta)}(\xi) \dd \xi+   \sum_{(j,k)\in \calE, \, \Re j \leq \beta }    \phi_{j,k} \int_0^\infty e^{i \xi \tau}\chi(\xi)  \xi^j \log^k \xi \dd \xi.
	\end{equation}
	Let $E_\beta(\tau) = \int_0^\infty e^{i \xi \tau}    \chi(\xi)  \phi^{(\beta)}(\xi) \dd \xi$, and, for each $(j,k)\in \bbC\times \bbN$, let
	\begin{equation}
		I_{j,k}[\chi](\tau)  =  \int_0^\infty e^{i \xi \tau}\chi(\xi)  \xi^j \log^k \xi \dd \xi,
	\end{equation}
	so that
	\begin{equation}
	\calF   \phi(\tau) = E_\beta(\tau) + \sum_{(j,k)\in \calE,\, \Re j\leq \beta}   \phi_{j,k} I_{j,k}[\chi](\tau).
	\end{equation}
	It follows immediately from \cite[Lemma 3.6]{HintzPrice} that $E_\beta\in \calA^{\beta+1}(\overline{\bbR}_\tau)$.
	On the other hand, $I_{j,k}[\chi]$ can be written as
	\begin{equation}
	I_{j,k}[\chi] = \frac{1}{2\pi} \calF \chi(\tau) *  \calF_{\xi\to \tau}(\Theta(\xi)\xi^j \log^k \xi).
	\end{equation}
	By \Cref{prop:phg_Fourier_comp},
	\begin{equation}
	\calF_{\xi\to \tau}(\Theta(\xi)\xi^j \log^k \xi) \in \calS'(\bbR_\tau) \cap \calA^{(j+1,k)}(\overline{\bbR}_\tau\backslash \{0\}) \subseteq \calS'(\bbR_\tau) \cap \calA^{\calE+1}(\overline{\bbR}_\tau\backslash \{0\}).
	\end{equation}
	So, by \Cref{prop:convolution_lemma}, $I_{j,k}[\chi](\tau) \in \calA^{\calE+1}(\overline{\bbR}_\tau)$.

	So, $\calF\phi(\tau) \in  \calA^{(\calE+1,\beta+1)} (\overline{\bbR}_\tau)$. Given the arbitrariness of $\beta \leq \alpha$, this implies the first clause of the proposition.

	The argument shows that the polyhomogeneous expansion of $\calF\phi$ is given, at the level of formal series, by
	\begin{equation}
		  \calF \phi \sim \sum_{(j,k)\in \calE}   \phi_{j,k} I_{j,k}[\chi](\tau).
		\label{eq:misc_050}
	\end{equation}
	By \Cref{prop:convolution_lemma}, the polyhomogeneous expansion of $I_{j,k}[\chi](\tau)$ in this limit is the same as that of $\smash{\calF_{\xi\to \tau}(\Theta(\xi)\xi^j \log^k \xi )(\tau)}$, which we compute in \Cref{prop:phg_Fourier_comp}.  Substituting this into \cref{eq:misc_050} yields  \cref{eq:misc_044}.
\end{proof}

\begin{proposition}
	For any $j\in \{z\in \bbC: \Re z>-1\}$ and $k\in \bbN$, $\calF_{\xi\to \tau}(\Theta(\xi)\xi^j \log^k \xi )$ is smooth away from the origin, and, for $\tau>0$,
	\begin{equation}
		\calF_{\xi\to \tau}(\Theta(\xi)\xi^j \log^k \xi )(\tau) = \tau^{-j-1} \sum_{\kappa=0}^k c_{j,k,\kappa} \log^\kappa \tau
		\label{eq:h1}
	\end{equation}
	for some $c_{j,k,\kappa}\in \bbC$. In fact, $c_{j,k,k} = i^{j+1} (-1)^k \Gamma(j+1)$,
	where $\Gamma:\bbC\backslash \bbZ^{\leq 0}\to \bbC$ denotes Euler's gamma function and $(\pm i)^{z} = \exp(\pm  \pi i z/2)$ for $z\in \bbC$.
	\label{prop:phg_Fourier_comp}
\end{proposition}
\begin{remark}
	The proof shows that $c_{j,k,\kappa}$ is given by
	\begin{equation}
	c_{j,k,\kappa} = i^{j+1} (-1)^\kappa \binom{k}{\kappa} \sum_{\varkappa=0}^{k-\kappa} \Big( \pm \frac{ \pi i}{2} \Big)^{k-\kappa-\varkappa} \binom{k-\kappa}{\varkappa}
	\frac{\mathrm{d}^\varkappa \Gamma(j+1) }{\mathrm{d} j^\varkappa}
	\label{eq:cdef}
	\end{equation}
	for all $j\in \{z\in \bbC: \Re z> -1\}$, $k\in \bbN$, and $\kappa \in \{0,\ldots,k\}$.
\end{remark}
\begin{proof}
	For $\tau\neq 0$, the Fourier transform $\calF_{\xi\to \tau}(\Theta(\xi)\xi^j \log^k \xi )$ is given by
	\begin{equation}
		\calF_{\xi\to \tau}(\Theta(\xi)\xi^j \log^k \xi ) = \lim_{\epsilon \to 0^+} \int_0^\infty e^{i \xi \tau -\epsilon \xi |\tau| } \xi^j \log^k (\xi) \dd \xi.
	\end{equation}
	The  integral on the right-hand side can be written
	\begin{equation}
		\int_0^\infty e^{i \xi \tau - \epsilon \xi |\tau|} \xi^j \log^k (\xi) \dd \xi = |\tau|^{-j-1} \sum_{\kappa=0}^k (-1)^\kappa \binom{k}{\kappa} \log^\kappa|\tau|  \int_0^\infty e^{\pm i \xi - \epsilon \xi} \xi^j \log^{k-\kappa}(\xi) \dd \xi,
	\end{equation}
	where the $\pm$ is the sign of $\tau$. The right-hand side has the same form as the right-hand side of \cref{eq:h1}, with an explicit formula for $c_{j,k,\kappa}$. All we need to do is compute the $\epsilon \to0^+$ limit of the integrals on the right-hand side.

	By Cauchy's integral theorem,
	\begin{align}
		\begin{split}
			\int_0^\infty e^{\pm i \xi - \epsilon \xi } \xi^j \log^k (\xi) \dd \xi &=  \int_{0}^{\pm i\infty} e^{\pm i z - \epsilon z} z^j \log^k (z) \dd z \\
			&= \pm i \int_0^\infty e^{- \xi \mp i\epsilon \xi} (\pm i\xi)^j \log^k(\pm i\xi) \dd \xi,
		\end{split}
	\end{align}
	where we are using the principal branch of the logarithm in order to fix the phase of $(\pm i \xi)^j = e^{\pm j\pi i/2} \xi^j$ and $\log(\pm i\xi) = \pm \pi i/2 + \log \xi$. So, taking $\epsilon \to 0^+$,
	\begin{equation}
		\lim_{\epsilon \to 0^+}	\int_0^\infty e^{\pm i \xi - \epsilon \xi } \xi^j \log^k (\xi) \dd \xi  = \pm  i \int_0^\infty e^{- \xi} (\pm i\xi)^j \log^k (\pm i\xi) \dd \xi.
		\label{eq:dge}
	\end{equation}
	The integral on the right-hand side can be written as
	\begin{equation}
		\frac{\mathrm{d}^k}{\mathrm{d} j^k}  \int_0^\infty e^{- \xi} (\pm i\xi)^j \dd \xi = \frac{\mathrm{d}^k }{\mathrm{d} j^k}( (\pm i)^j \Gamma(j+1)) = (\pm i)^j\sum_{\kappa=0}^k \binom{k}{\kappa} \Big( \pm \frac{\pi i}{2} \Big)^{k-\kappa}\frac{\mathrm{d}^\kappa \Gamma(j+1) }{\mathrm{d} j^\kappa}.
		\label{eq:1h3}
	\end{equation}
	Chaining together these equalities yields the proposition.
\end{proof}

Finally, to complete the proof of \Cref{prop:fund}:
\begin{lemma}
	Suppose that $\chi\in \calS(\bbR)$ is identically $1$ near the origin, and suppose that $f\in \calS'(\bbR) \cap \calA^{\calE}_{\mathrm{loc}}(\overline{\bbR}_\tau\backslash \{0\}) = \calS'(\bbR) \cap \calA^{\calE}_{\mathrm{loc}}([-\infty,0)\cup (0,\infty]_\tau)$ for some index set $\calE\subset\bbC\times\bbN$. Then,
	\begin{equation}
		(2\pi)^{-1}\calF \chi * f(\tau) - (1- \psi(\tau)) f(\tau) \in \calS(\bbR_\tau)
	\end{equation}
	for any $\psi\in C_{\mathrm{c}}^\infty(\bbR_\tau)$ identically $1$ near the origin.
	\label{prop:convolution_lemma}
\end{lemma}
\begin{proof}
	Write $f(\tau) = E(\tau) + F(\tau) $ for $E\in \calE'(\bbR)$ a compactly supported distribution and $F\in \calA^\calE(\overline{\bbR}_\tau)$. Then,
	\begin{multline}
		(2\pi)^{-1}\calF \chi * f(\tau) - (1- \psi(\tau)) f(\tau) \\ = (2\pi)^{-1}\calF\chi * F(\tau) -  F(\tau)+ (2\pi)^{-1}\calF \chi * E(\tau) - (1- \psi(\tau)) E + \psi F(\tau) .
		\label{eq:misc_059}
	\end{multline}
	It is quick to see that the last three terms are Schwartz. Indeed, $\psi F \in C_{\mathrm{c}}^\infty(\bbR)$, and necessarily $\operatorname{singsupp} E \subseteq \{0\}$, so $(1-\psi) E \in  C_{\mathrm{c}}^\infty(\bbR)$ as well. The term $\calF \chi * E(\tau)$ being Schwartz is equivalent to $\chi \calF^{-1} E$ being Schwartz. Since $E$ is compactly supported, $\calF^{-1} E$ is smooth, so $\chi \calF^{-1} E \in C_{\mathrm{c}}^\infty(\bbR)$ is Schwartz.

	The remaining term on the right-hand side, $(2\pi)^{-1}\calF\chi * F -  F$, requires a short argument. This term being Schwartz is equivalent to $(1-\chi(\xi)) \calF^{-1} F(\xi) \in \smash{\calS(\bbR_\xi)}$,
	which follows if $\calF^{-1} F$ is smooth except at the origin and Schwartz outside of some compact subset.
	We can write $F(\tau) = (1+\tau^2)^j F_0(\tau)$ for some $F_0 \in \calA^{1}(\overline{\bbR})$ and $j\in \bbN$. Because
	\begin{equation}
		1+\Delta_\xi = \calF_{\tau\to \xi}^{-1} \circ M_{1+\tau^2} \circ \calF_{\xi \to \tau}
	\end{equation}
	preserves the space of tempered distributions on the real line that are Schwartz except at the origin, it suffices to consider the case $F\in \calA^1(\overline{\bbR})$. Then, $\partial_\tau^{k+\ell}  (\tau^k F(\tau))\in \calA^{1}(\overline{\bbR}_\tau) \subseteq L^2(\bbR_\tau)$ for all $k,\ell\in \bbN$. Taking the inverse Fourier transform  yields
	\begin{equation}
		\xi^{k+\ell} \partial_\xi^k \calF^{-1} F(\xi)\in L^2(\bbR_\xi)
	\end{equation}
	for all $k,\ell\in \bbN$, which implies that $\calF^{-1} F$ is Schwartz except at the origin, as desired.
\end{proof}

\subsection{Bound states}

The contribution from bound states is a finite sum of functions $w: \bbR_t\times X_x\to \bbC $ of the form $w(t,x) = e^{-iEt} \varphi(x)$ for some $E\in \bbR$ and Schwartz $\varphi \in \calS(X)$.

\begin{proposition}
	If $v(t,x) = e^{-iEt} \varphi(x)$ for some $E>0$ and Schwartz $\varphi \in \calS(X)$, then $v$ is of exponential-polyhomogeneous type on $M$ and Schwartz at $\mathrm{nf}\cup \mathrm{dilF} \cup \mathrm{parF}$.
	\label{prop:bound}
\end{proposition}
\begin{proof}
	On the cylinder $C = [0,\infty]_t\times X$, each such $v$ is already of exponential-polyhomogeneous type, with
	\begin{equation}
	v \in e^{-i E t} \bigcap_{k\in \bbN}  \rho(x)^k C^\infty(C) \subseteq e^{- i E t}\bigcap_{k\in \bbN}  \varrho_{\mathrm{nf}}^k \varrho_{\mathrm{dilF}}^k \varrho_{\mathrm{parF}}^k C^\infty(M),
	\end{equation}
	where we used $C^\infty(C)\subset C^\infty(M)$, which follows from the construction of $M$ via blowing up some corners of $C$.

	We can choose $\varrho_{\mathrm{kf}} = t^{-1} \rho(x)^{-1}(\rho(x)+1/(t \rho(x)))^{-1}$, $\varrho_{\mathrm{parF}} = \rho(x)+1/(t \rho(x))$, and $\varrho_{\mathrm{dilF}} = \rho(x)(\rho(x)+1/(t \rho(x)))^{-1}$, as follows from the construction of $M$ from $M/\operatorname{parF}$, as described in the introduction.  So, $t = \varrho_{\mathrm{dilF}}^{-1} \varrho_{\mathrm{parF}}^{-2}\varrho_{\mathrm{kf}}^{-1}$ is polyhomogeneous on $M$.
\end{proof}

\section{Radiation-field analogue and the high energy contribution}
\label{sec:high}

Fix an index set $\calE\subset \bbC\times \bbN$ and $\alpha\in \bbR \cup \{\infty\}$.
Suppose that $\phi \in \calS( \bbR_\sigma; \calA_{\mathrm{loc}}^{(\calE,\alpha)}(X))$ vanishes identically in $\{\sigma< \sigma_0\} \subset \bbR_\sigma\times X$ for some $\sigma_0>0$. We denote the set of such functions as
\begin{equation}
\dot{\calS}( [\sigma_0,\infty)_\sigma ; \calA_{\mathrm{loc}}^{(\calE,\alpha)}(X)).
\end{equation}
In this section, we analyze
\begin{equation}
I_{\pm}[\phi] (t,x)=\int_0^\infty e^{i \sigma^2 t \pm  i \sigma r(x)} \phi(\sigma,x) \dd \sigma,
\label{eq:ethan_Ipm}
\end{equation}
where $r(x) = \rho(x)^{-1}$.
The argument is a straightforward application of the method of stationary phase, as we will see. The phase appearing in the oscillatory integral is $\theta_\pm(t,x;\sigma) = \sigma^2 t \pm \sigma r$, which has derivative $\partial_\sigma \theta_\pm(t,x;\sigma) = 2\sigma t \pm r$.
Remembering that $\rho>0$ and $\sigma,t\geq 0$, $\partial_\sigma \theta_+$ is nonvanishing, while $\partial_\sigma \theta_-$ vanishes at the ``critical'' frequency $\sigma_{\mathrm{crit}} = r/2t$. Thus, following the oscillatory integral $I_\pm[\phi]$ along level sets of $r/t \in C^\infty(\mathrm{dilF}^\circ)$, an observer either sees rapid decay or else asymptotics in accordance with the stationary phase expansion.

For $I_+[\phi]$, the method of nonstationary phase also yields rapid decay at $\mathrm{nf}$. It turns out that $I_-[\phi]$ also decays rapidly at $\mathrm{nf}$.
The reason is that, in this asymptotic regime, $r\to\infty$ and $t/r\to 0$, which means that $\sigma_{\mathrm{crit}}\to \infty$.
As $\phi(\sigma ,-)$ decays rapidly as $\sigma \to\infty$, the data in the stationary phase approximation decays rapidly as well. A less careful version of this reasoning (valid only in $\mathrm{nf}^\circ$) is that, in $\mathrm{nf}^\circ$, only $r$ is a large parameter, so the relevant portion of the phase is $\theta_{\pm,0} = \pm r/2\sigma$,
whose gradient $\partial_\sigma \theta_{\pm,0}$ is nonvanishing, so the method of nonstationary phase applies, regardless of the sign.

\subsection{Nonstationary case of sign}

We first turn to the nonstationary case. Actually, in addition to discussing $I_+[\phi]$, we discuss the contribution $I_{-,\mathrm{non}}[\phi(\sigma,-),\psi]$ to $I_{-}[\phi]$, where
\begin{equation}
	I_{-,\mathrm{non}}[\phi,\psi] (t,x)=\int_{0}^\infty e^{i \sigma^2 t -  i \sigma r(x)} \Big[ 1 - \psi\Big( \sigma - \frac{r}{2t} \Big)\Big] \phi(\sigma,x) \dd \sigma \in C^\infty(\bbR^+_t\times X^\circ_x),
\end{equation}
where $\psi\in C_{\mathrm{c}}^\infty(\bbR)$ is identically $1$ in some neighborhood of the origin and, for convenience, $\operatorname{supp} \psi \Subset (-\sigma_0,\sigma_0)$, where $\sigma_0$ is chosen such that $\phi(\sigma,-)=0$ whenever $\sigma\leq \sigma_0$.

In the next proposition, let $\calA^{\infty,\infty,(0,0)}(C)$ denote the set of smooth functions on $C = [0,\infty]_t\times  X$ that are Schwartz at $\partial C\backslash \{t=0\}$. Such functions are smooth on $M$ and Schwartz at all faces except $\Sigma = \mathrm{cl}_M \{t=0\}$.

\begin{proposition}
	If $\phi \in \dot{\calS} ( [\sigma_0,\infty)_\sigma ; \calA^{(\calE,\alpha)}(X))$, then $I_+[\phi],I_{-,\mathrm{non}}[\phi]\in \calA^{\infty,\infty,(0,0)}(C)$.
	\label{prop:high_nonstat}
\end{proposition}
\begin{proof}
	For any $R>0$, we have $e^{ \pm i E^{1/2} r(x)} \phi(E^{1/2},x) \in \dot{\calS}( [\sigma_0^2,\infty) ; C^\infty(\{r(x)<R\} ))$. So, since the Fourier transform has the mapping property
	\begin{equation}
		\calF_{E\to t}: \calS( \bbR_{E} ; C^\infty(\{r(x)<R\} )) \to \calS( \bbR_{t} ; C^\infty(\{r(x)<R\} )),
	\end{equation}
	we deduce, since 
	\begin{equation*} 
		I_+[\phi](t,x) = \calF_{E \to t}( e^{i E^{1/2} r(x)} \phi(E^{1/2},x) /(2E^{1/2}))(t),
	\end{equation*} 
	that $I_+[\phi](t,x) \in \calS(\bbR_t;   C^\infty (X^\circ_x))$. Similarly,  if $t$ is sufficiently large so that an $R/2t$-neighborhood of $\operatorname{supp} \psi$ is still a subset of $(-\sigma_0,\sigma_0)$, then if $r(x)<R$,
	\begin{equation}
	I_{-,\mathrm{non}}[\phi](t,x) = \calF_{E \to t}( e^{-i E^{1/2} r(x)} \phi(E^{1/2},x)/(2E^{1/2}) )(t) .
	\end{equation}
	So, $I_{-,\mathrm{non}}[\phi](t,x) \in \calS(\bbR_t;   C^\infty (X^\circ_x))$.
	So, in order to prove the proposition, it suffices to restrict attention to any neighborhood in $C$ of $[0,\infty]_t\times \partial X$, at least one of which is identifiable with $[0,\infty]_t\times \dot{X}[R]$ for $\dot{X}[R]=(R,\infty]_r \times \partial X_\theta$.

	Let $\calA^{\infty,\infty,(0,0)}_{\mathrm{loc}}(\dot{C}[R])$ denote the set of smooth functions on $\dot{C}[R] = [0,\infty]_t\times \dot{X}[R]$
	Schwartz at $\partial \dot{C}[R]\backslash \{t=0\}$. It suffices to prove that
	\begin{equation}
	I_+[\phi] \in \calA^{\infty,\infty,(0,0)}_{\mathrm{loc}}(\dot{C}[R]).
	\end{equation}

	Let $L\in \operatorname{Diff}_{\mathrm{b}}(\dot{X})$, i.e.\ $L$ is a differential operator in the $C^\infty(\dot{X})$-algebra generated by vector fields tangent to $\partial X$. For any $j\in \bbN$, we can write $\smash{\partial_t^j} L I_+[\phi] = I_+[\phi_{j,L}]$ for some
	\begin{equation}
	\phi_{j,L} \in \calS(\bbR_E; \calA_{\mathrm{loc}}^{(\calE,\alpha)}(\dot{X}))
	\end{equation}
	vanishing identically in $\{E<E_0=\sigma^2_0\}$.
	In order to prove that $I_+[\phi] \in \dot{\calS}$, it suffices to prove that
	\begin{equation}
	I_+[\phi_{j,L}]  \in \langle t +r \rangle^{-K}  L^\infty_{\mathrm{loc}}( [0,\infty]_t \times \dot{X})
	\label{eq:high_inductive}
	\end{equation}
	for every $K\in \bbZ$. Here, $L^\infty_{\mathrm{loc}}( [0,\infty]_t \times \dot{X})$ is the set of functions $f(t,r,\theta)$ on $\bbR^+_t\times \dot{X}_{r,\theta}$, such that, for each $r_0>0$, there exists some $C_{r_0}>0$ such that $|f(t,r,\theta)| \leq C_{r_0}$ whenever $r\geq r_0$.

	More generally, we show that the bound \cref{eq:high_inductive} holds for $I_+[\psi]$ whenever $
	\psi(t,E,r,\theta) \in C^\infty(\bbR^+_t\times \bbR_E \times \dot{X}_{r,\theta})$
	is vanishing identically on $\{E<E_0\}$ and satisfies the following bounds: there exists some $J\in \bbR$ such that, for all $k,K'\in \bbN$, and for all $Q \in \operatorname{Diff}_{\mathrm{sc}}(\dot{X})$,
	\begin{equation}
	\frac{\partial^k}{\partial E^k} Q\psi(t,E,r,\theta) \in  \langle E \rangle^{-K'} \langle t+r \rangle^{J} L^\infty(\bbR^+_t\times \bbR_E \times (\dot{X}_{r,\theta} \cap \{r\geq r_0\}) )
	\label{eq:high_ass}
	\end{equation}
	for all $r_0>0$. For the $\phi_{j,L}$ above, this holds with $J>0$ sufficiently large such that $	\calA_{\mathrm{loc}}^{(\calE,\alpha)}(\dot{X}) \subset \langle r \rangle^{J} L^\infty_{\mathrm{loc}}(\dot{X})$,
	but it will be useful to consider other values of $J$.

	Applying \cref{eq:high_ass} with $k=0$ and $K'=2$ yields
	\begin{equation}
	|I_+[\psi](t,r,\theta)| \leq \int_{E_0}^\infty |\psi(t,E,r,\theta)| \dd E  \in  \langle t+r  \rangle^{J}  L^\infty_{\mathrm{loc}}( [0,\infty]_t \times \dot{X}),
	\end{equation}
	so \cref{eq:high_inductive} holds with $K=-J$. This is the base case of the inductive argument.

	Let $\hat{K}\in \bbN$.
	Suppose we have shown that \cref{eq:high_inductive}  holds for $K = - J + \tilde{K}$ for all $\tilde{K} \in \{0,\ldots,\hat{K}\}$, whenever $\psi$ satisfies \cref{eq:high_ass}. The inductive step, which once handled completes the argument, is to show that the bound holds also for
	\begin{equation}
	K = - J + \hat{K} + 1,
	\end{equation}
	i.e.\ that $I_+[\psi] \in \langle t+r\rangle^{J - \hat{K}-1} L^\infty_{\mathrm{loc}}([0,\infty]_t\times \dot{X})$. Writing
	\begin{equation}
	I_+[\psi] =-i \int_{E_0}^\infty \Big( \frac{\partial \theta_+}{\partial E}\Big)^{-1}  \Big( \frac{\partial}{\partial E}e^{i E t + i \sigma(E) r} \Big) \psi(t,E,-) \dd E
	\label{eq:ethan_104}
	\end{equation}
	and integrating by parts, the result is $I_+[\psi] = I_+[\psi_1 ]+ I_+[\psi_2]$ for
	\begin{equation}
	\psi_1 = i\Big(\frac{\partial \theta_+}{\partial E}\Big)^{-1} \frac{\partial \psi}{\partial E}, \qquad
	\psi_2= -i\Big(\frac{\partial \theta_+}{\partial E}\Big)^{-2} \frac{\partial^2 \theta_+}{\partial E^2} \psi .
	\end{equation}
	Note that $\partial^2_E \theta_+ = - r/(4 E^{3/2})$. Thus, \Cref{lem:theta_controledness} implies that for each $\nu\in \{1,2\}$, all $k,K'\in \bbN$, and for all $Q \in \operatorname{Diff}_{\mathrm{sc}}(\dot{X})$,
	\begin{equation}
	\frac{\partial^k}{\partial E^k} Q\psi_\nu(t,E,r,\theta) \in  \langle E \rangle^{-K'} \langle t+r \rangle^{J-1} L^\infty(\bbR^+_t\times \bbR_E \times (\dot{X}_{r,\theta} \cap \{r\geq r_0\}) )
	\end{equation}
	for all $r_0>0$. In other words, each of $\psi_1,\psi_2$ also satisfies \cref{eq:high_ass}, except with a lower value of $J$. Thus, by the phrasing of the inductive hypothesis,
	\begin{equation}
	I_+[\psi_\nu] \in  \langle t+r\rangle^{J - \hat{K}-1} L^\infty_{\mathrm{loc}}([0,\infty]_t\times \dot{X}) \subseteq \langle t+r\rangle^{J - \hat{K}-1} L^\infty_{\mathrm{loc}}([0,\infty]_t\times \dot{X}),
	\end{equation}
	as desired.
	
	The same argument applies to $I_{-,\mathrm{non}}[\phi](t,x)$, since the cutoff $1-\psi(\sigma-r/(2t))$ cuts off away from the critical energy, where $\partial \theta_-/\partial E$ vanishes. 
\end{proof}

\begin{lemma}
	For any $j,k\in \bbN$,  and for any $r_0,\sigma_0>0$, there exists a constant $C = C(j,k,\sigma_0,r_0)>0$ such that
	\begin{equation}
	\Big| \frac{\partial^j  \partial ^k}{\partial \sigma ^j  \partial r^k } \frac{\sigma}{2\sigma t+r} \Big| \leq \frac{C \langle \sigma \rangle}{\langle t+r \rangle^{k+1} }
	\label{eq:ethan_108}
	\end{equation}
	holds for all $t > 0$, $r\geq r_0$, and $\sigma \geq \sigma_0$.
	\label{lem:theta_controledness}
\end{lemma}
\begin{proof}
	We have $\partial_\sigma^j \partial_r^k (\sigma /(2\sigma t+r))  = (-1)^k k! \partial_\sigma^j (\sigma/(2\sigma t+r)^{k+1})$,
	and
	\begin{equation}
	\frac{\partial^j}{\partial \sigma^j } \frac{\sigma}{(2\sigma t+r)^{k+1} } =
	\begin{cases}
 \frac{ \sigma}{(2\sigma t +r)^{k+1}} & (j=0), \\ 
	\frac{(k+j-1)! }{k!} \frac{j (-2t)^{j-1} }{(2\sigma t+r)^{k+j}} + \frac{(k+j)!}{k!} \frac{(-2t)^j \sigma}{(2\sigma t +r)^{k+j+1}}  & (j\geq 1).\\ 
	\end{cases} 
	\label{eq:ethan_j37}
	\end{equation}
	The first term on the right-hand side in the $j\geq 1$ case satisfies the required estimate, as, for $j\geq 1$,
	\begin{equation}
	0\leq  t^{j-1} (2\sigma t+r)^{-k-j} \leq  (2 \sigma_0)^{-j+1} (2\sigma_0 t + r)^{-k-1} \leq C_{j,k}\langle t+r \rangle^{-k-1}
	\end{equation}
	for some $C_{j,k}>0$. The second term in \cref{eq:ethan_j37} is under control as well, as $0 \leq t^j (2 \sigma t + r)^{-k-j-1} \leq (2\sigma_0)^{-j} (2\sigma_0 t+r)^{-k-1} \leq C_{j,k} \langle t+r \rangle^{-k-1}$,
	for a possibly different $C_{j,k}$. So, the bound \cref{eq:ethan_108} follows.
\end{proof}

\subsection{Stationary remainder}
We now turn to the remaining contribution $I_{-,\mathrm{stat}}[ \phi(\sigma,-),\psi ]$ to $I_-[\phi]$, where
\begin{equation}
	I_{-,\mathrm{stat}}[\phi,\psi](t,x) = \int_0^\infty e^{i \sigma^2 t -  i \sigma r(x)} \psi\Big( \sigma - \frac{r}{2t} \Big) \phi(\sigma,x) \dd \sigma \in C^\infty(\bbR^+_t\times X^\circ_x).
	\label{eq:misc_069}
\end{equation}
The main proposition of this section says:
\begin{proposition}
	Given $\phi \in \dot{\calS}([0,\infty)_\sigma ; \calA^{(\calE,\alpha)}(X) )$ vanishing near zero, $\psi\in C_{\mathrm{c}}^\infty(\bbR)$ satisfying $\operatorname{supp} \psi \Subset (-\sigma_0,\sigma_0)$, $0\notin \operatorname{supp}(1-\psi)$, and given  $\chi \in C_{\mathrm{c}}^\infty(\bbR)$ identically $1$ near the origin,
	\begin{equation}
		I_{-,\mathrm{stat}}[\phi,\psi] \in e^{- i (1-\chi(t)) r^2/(4t)} \calA^{\infty,\infty, (\calE+1/2,\alpha+1/2),\infty,\infty}(M),
	\end{equation}
	with $I_{-,\mathrm{stat}}[\phi,\psi]$ vanishing identically in $\mathrm{cl}_M\{r/(2t) \leq \epsilon\}$ for some $\epsilon>0$.
	\label{prop:high_stat}
\end{proposition}
\begin{proof}
	Since $\phi(\sigma,x)=0$ for $\sigma\leq \sigma_0$, and since we chose $\psi$ such that $\operatorname{supp} \psi \Subset (-\sigma_0,\sigma_0)$, and therefore $\operatorname{supp} \psi \subset (-\sigma_0+\epsilon,\sigma_0-\epsilon)$ for some $\epsilon>0$, the integral $I_{-,\mathrm{stat}}[\phi,\psi](t,x)$ is vanishing in $\{r/2t \leq \epsilon\}$. Thus, we work on the sub-mwc
	\begin{equation}
		M_{\epsilon,R}=M \cap \mathrm{cl}_{M} \{r/2t > \epsilon, r(x)> R\}
	\end{equation}
	for $R>0$ sufficiently large such that we can identify $X\cap \mathrm{cl}_X\{r(x)>R\}$ with $\dot{X}[R] = [0,R^{-1})_{\rho} \times \partial X_\theta$ via a choice of boundary collar $\dot{X}[R] \hookrightarrow X$.

	We can write $M_{\epsilon,R} = M_{\epsilon,R}^\circ \cup U_0 \cup U$, where, for any $T>0$ satisfying $T\epsilon > R$ and $R_0>R$ satisfying $R_0/(4T) > \epsilon$,
	\begin{equation}
		U_0 \cong  [0,2T)_t\times \dot{X}[R_0], \qquad
		U \cong (T,\infty]_t\times  ( 2\epsilon,\infty]_{r/t} \times \partial X_\theta.
	\end{equation}
	That is:
	\begin{itemize}
		\item the map $[0,2T)_t\times \dot{X}[R_0] \hookrightarrow M_{\epsilon,R}$, applying the boundary collar to the right factor, is a diffeomorphism onto $U_0$, and
		\item  the composition
		\begin{equation}
		(T,\infty)_t\times  (2\epsilon,\infty)_{r/t} \times \partial X_\theta  \hookrightarrow (T,\infty)_t\times \dot{X}[R]_{r,\theta} \hookrightarrow M^\circ_{\epsilon,R},
		\end{equation}
		where the first map sends $(t,\hat{r},\theta)\mapsto (t,(t\hat{r},\theta))$ and the second map applies the boundary collar, extends to a diffeomorphism $ (T,\infty]_t\times  ( 2\epsilon,\infty]_{r/t} \times \partial X_\theta\to U$.
	\end{itemize}
	So, in order to conclude the proposition, it suffices to prove that $I_{-,\mathrm{stat}}[\phi,\psi](t,r,\theta)\in \calS_{\mathrm{loc}} ( [0,\infty)_t\times \dot{X})$ and
	\begin{equation}
		I_{-,\mathrm{stat}}[\phi,\psi](t,\hat{r} t,\theta)\in e^{-i \hat{r}^2t/4} \calA_{\mathrm{loc}}^{(\calE+1/2,\alpha+1/2),\infty }( (T,\infty]_t\times (2\epsilon,\infty]_{\hat{r}} \times \partial X_\theta ),
	\end{equation}
	where $\calE+1/2$ is the index set at $t=\infty$ and the $\infty$ denotes Schwartz behavior at $\hat{r}=\infty$.
	These claims are proven below. The first is in \Cref{lem:aden}, and the second is in \Cref{lem:ruth}.
\end{proof}

A modification of \cref{eq:misc_069},
\begin{equation}
	I_{-,\mathrm{stat}}[\phi,\psi](t,r,\theta) = \int_0^\infty e^{i \sigma^2 t -  i \sigma r} \psi\Big( \sigma - \frac{r}{2t} \Big) \phi(\sigma,r,\theta) \dd \sigma \in C^\infty(\bbR^+_t\times \dot{X}^\circ[R]_{r,\theta} )
\end{equation}
defines a function $I_{-,\mathrm{stat}}[\phi,\psi]:\bbR_t\times \dot{X}[R] \to \bbC$  for any $\phi \in \calS( \bbR_\sigma; \calA_{\mathrm{loc}}^{(\calE,\alpha)}(\dot{X}))$, for any index set $\calE$ and any $\alpha\in \bbR\cup \{\infty\}$, and for any $\psi \in C_{\mathrm{c}}^\infty(\bbR)$.

\begin{proposition}
	For any $\phi \in \calS( \bbR_\sigma; \calA_{\mathrm{loc}}^{(\calE,\alpha)}(\dot{X}))$ and $\psi\in C_{\mathrm{c}}^\infty(\bbR)$, the function $I_{-,\mathrm{stat}}[\phi,\psi]$ satisfies
	\begin{equation}
		I_{-,\mathrm{stat}}[\phi,\psi] \in \calS_{\mathrm{loc}} ( [0,\infty)_t\times \dot{X}),
	\end{equation}
	i.e.\ is Schwartz at both boundary hypersurfaces $\{t=0\}$ and $[0,\infty)_t\times \partial \dot{X}$.
	\label{lem:aden}
\end{proposition}
\begin{proof}
	Using the rapid decay of $\phi(\sigma,-)$ as $\sigma\to\infty$ in some weighted $L^\infty$-space $\langle r \rangle^J L^\infty_{\mathrm{loc}}(\dot{X})$, we have, for all $r_0>0$ and $r \geq r_0$,
	\begin{equation}
		|I_{-,\mathrm{stat}}[\phi,\psi](t,r,\theta) | \leq \lVert \psi \rVert_{L^1} \sup_{\sigma \geq r/2t - \sigma_0} |\phi(\sigma,r,\theta)| \preceq \Big\langle \frac{r}{2t} - \sigma_0 \Big\rangle^{-K} \langle r \rangle^{J},
	\end{equation}
	which holds for some $J\geq 0$ and all $K\geq 0$, where the constant involved depends on $r_0$. Since $\langle r/2t - \sigma_0 \rangle \succeq \langle r \rangle \langle t^{-1} \rangle$ (locally), we conclude that $I_{-,\mathrm{stat}}[\phi,\psi] \in \langle r \rangle^{-\infty} \langle t^{-1} \rangle^{-\infty} L^\infty_{\mathrm{loc}} ( [0,\infty)_t\times \dot{X})$.

	In order to control derivatives, we use the identities
	\begin{align}
	\partial_t I_{-,\mathrm{stat}}[\phi,\psi] &= i I_{-,\mathrm{stat}}[\sigma^2 \phi,\psi] + (r/2t^2) I_{-,\mathrm{stat}}[\phi,\psi'] \\
	\partial_r I_{-,\mathrm{stat}}[\phi,\psi] &= -i I_{-,\mathrm{stat}}[\sigma \phi,\psi] - (1/2t) I_{-,\mathrm{stat}}[\phi,\psi'] + I_{-,\mathrm{stat}}[\partial_r \phi,\psi].
	\end{align}
	Applying these inductively, and applying the $L^\infty$-bounds derived in the previous paragraph, it can be concluded that
	\begin{equation}
	\partial_t^j \partial_r^k I_{-,\mathrm{stat}}[\phi,\psi] \in \langle r \rangle^{-\infty} \langle t^{-1} \rangle^{-\infty}  L^\infty_{\mathrm{loc}} ( [0,\infty)_t\times \dot{X})
	\end{equation}
	for all $j,k\in \bbN$. So, $	I_{-,\mathrm{stat}}[\phi,\psi] \in \calS_{\mathrm{loc}} ( [0,\infty)_t\times \dot{X})$.
\end{proof}

\begin{proposition} For $\phi \in \calS( \bbR_\sigma; \calA_{\mathrm{loc}}^{(\calE,\alpha)}(\dot{X}))$ vanishing near $\sigma=0$ and $\psi$ as above,
	\begin{equation}
		I_{-,\mathrm{stat}}[\phi,\psi](t,t\hat{r},\theta) \in e^{-i r^2/(4t)} \calA_{\mathrm{loc}}^{(\calE+1/2,\alpha+1/2),\infty }((0,\infty]_t\times (0,\infty]_{\hat{r}} \times \partial X_\theta ).
	\end{equation}
	When $\calE\subset \bbN\times \bbN$, the $t\to\infty$ expansion is given by
	\begin{align}
		I_{-,\mathrm{stat}}[\phi,\psi](t,t\hat{r},\theta) &\sim e^{-i r^2/(4t)} \sum_{j=0}^\infty \frac{\Gamma(j+1/2)}{(2j)! (-i t)^{j+1/2}} \phi^{(2j)}(\hat{r}/2,t\hat{r},\theta), \\
		\begin{split}
		&\sim \frac{e^{-i r^2/(4t)}}{\sqrt{-i t}}  \sum_{(j,k)\in \calE} \Big[ \sum_{j_0=0}^j \hat{r}^{-j_0} \frac{\Gamma(j-j_0+1/2)}{(2(j-j_0))! (-i)^{j-j_0} }   \\
		&\qquad\qquad\qquad\qquad \times \sum_{k_0} \log^{k_0}(\hat{r}) \binom{k+k_0}{k_0}  \phi^{2(j-j_0)}_{j_0,k+k_0} \Big]  t^{-j} \log^k (t),
		\label{eq:dilF_exp}
		\end{split}
	\end{align}
	where $\phi^{(m)}(\sigma,r,\theta) = \partial_\sigma^m \phi(\sigma,r,\theta)$ for $m\in \bbN$, and $\phi^{(m)} = 0$ if $m<0$, and where
	\begin{equation}
		\phi^{(m)}(\sigma,r,\theta) \sim \sum_{(j,k)\in \calE} \phi^{(m)}_{j,k}(\sigma,\theta) r^{-j} \log^k(r)
	\end{equation}
	is the polyhomogeneous expansion of $\phi^{(m)}$ at $r=\infty$, i.e.\ at $\mathrm{bf}$. Here, each $\phi^{(m)}_{j,k}(\sigma,\theta)$ is in $\dot{\calS}([\sigma_0, \infty) ; C^\infty(\partial X_\theta))$.
	\label{lem:ruth}
\end{proposition}
\begin{proof}
	Let $\tilde{I}_{-,\mathrm{stat}}[\phi,\psi](t,r,\theta) = e^{i r^2/(4t)} I_{-,\mathrm{stat}}[\phi,\psi](t,r,\theta)$. In terms of $\hat{r}=r/t$, this can be written
	\begin{equation}
	\tilde{I}_{-,\mathrm{stat}}[\phi,\psi](t,\hat{r}t,\theta) =  \int_0^\infty e^{it (\sigma-\hat{r}/2)^2 } \psi(\sigma-\hat{r}/2) \phi(\sigma,t\hat{r},\theta)  \dd \sigma.
	\end{equation}
	Our goal is to prove that this lies in $ \calA_{\mathrm{loc}}^{(\calE+1/2,\alpha+1/2),\infty } ((0,\infty]_t\times (0,\infty]_{\hat{r}}\times \partial X_\theta )$.

		For each $K\in \bbN$, Taylor's theorem says that
	\begin{equation}
		\phi(\sigma,r,\theta) = \sum_{k=0}^K \frac{1}{k!} \Big( \sigma - \frac{\hat{r}}{2} \Big)^k \phi^{(k)}\Big( \frac{\hat{r}}{2},r,\theta\Big)  + \frac{1}{K!} \int_{0}^{\sigma-\hat{r}/2}  \Big(\sigma - \frac{\hat{r}}{2}-\delta\Big)^K \phi^{(K+1)}\Big( \frac{\hat{r}}{2}+\delta,r,\theta \Big)  \dd \delta,
	\end{equation}
	where the superscript on $\phi$ refers to differentiation in the first slot.

	Thus, $\tilde{I}_{-,\mathrm{stat}}[\phi,\psi] = \sum_{k=0}^K \phi^{(k)}(\hat{r}/2,r,\theta) I_{-,\mathrm{stat},k}[\psi] + I_{-,\mathrm{stat},K,\mathrm{rem}}[\phi,\psi ]$ for
	\begin{align}
		\begin{split}
			I_{-,\mathrm{stat},k}[\psi] &= \frac{1}{k!} \int_{-\infty}^\infty  e^{it(\sigma-\hat{r}/2)^2} \Big( \sigma - \frac{\hat{r}}{2} \Big)^k  \psi\Big( \sigma - \frac{\hat{r}}{2} \Big)  \dd \sigma \\
			&= \frac{1}{k!} \int_{-\infty}^\infty  e^{it\delta^2} \delta^k  \psi(\delta)  \dd \delta,
		\end{split}  \\
		I_{-,\mathrm{stat},K,\mathrm{rem}}[\phi,\psi] &= \frac{1}{K!} \int_{-\infty}^\infty  e^{it\Delta^2}\psi(\Delta) \Big[\int_{0}^{\Delta}  (\Delta-\delta)^K \phi^{(K+1)}\Big( \frac{\hat{r}}{2}+\delta,r,\theta \Big) \dd \delta \Big]\dd \Delta .
	\end{align}
	The stationary phase approximation suffices to show that $I_{-,\mathrm{stat},k}[\psi] \in t^{-(k+1)/2} C^\infty((0,\infty]_t)$. In fact, since $\psi=1$ identically near the origin, the difference
	\begin{equation}
		k! I_{-,\mathrm{stat},k}[\psi](t)-
		\begin{cases}
			0 & (k\text{ odd}), \\
			(- it)^{-(k+1)/2} \Gamma((k+1)/2) & (\text{otherwise}),
		\end{cases}
	\end{equation}
	is, for large $t$, Schwartz.

	Since $\phi^{(k)}(\hat{r}/2,r,\theta)=\phi^{(k)}(\hat{r}/2,\hat{r} t,\theta)$ lies in $\calS(\bbR_{\hat{r}}; \calA_{\mathrm{loc}}^{(\calE,\alpha)}( (0,\infty ]_t ) )$, it follows that
	\begin{equation}
		\phi^{(k)}(\hat{r}/2,r,\theta) I_{-,\mathrm{stat},k}[\psi] \in \calS(\bbR_{\hat{r}}; t^{-1/2} \calA_{\mathrm{loc}}^{(\calE,\alpha)}( (0,\infty ]_t ) ).
	\end{equation}

	On the other hand, \Cref{lem:adam} shows that
	\begin{equation}
		|I_{-,\mathrm{stat},K,\mathrm{rem}}[\phi,\psi](t,\hat{r}t,\theta)| \in t^{-\lfloor (K+1)/2 \rfloor } \hat{r}^{-\infty} \calA_{\mathrm{loc}}^{0,0}((0,\infty]_t \times (0,\infty]_{\hat{r}}\times \partial X_\theta ).
	\end{equation}
	Combining everything, $I_{-,\mathrm{stat}}[\phi,\psi]\in  \calS(\bbR_{\hat{r}}; \calA_{\mathrm{loc}}^{(\calE+1/2,\min\{\alpha+1/2,\lfloor (K+1) /2 \rfloor \})} (0,\infty ]_t ) )$. Since $K$ can be taken arbitrarily large, the result follows.
\end{proof}

\begin{lemma}
	For each $J,K\in \bbN$, $\phi \in \calS(\bbR_\sigma ; \calA_{\mathrm{loc}}^0(\dot{X}))$, and $\psi \in C_{\mathrm{c}}^\infty(\bbR)$, consider the function $\calI_{J,K}[\phi,\psi]:\bbR^+_t\times \dot{X}^\circ_{r,\theta} \to \bbC$ given by
	\begin{equation}
	\calI_{J,K}[\phi,\psi] =\int_{-\infty}^\infty  e^{it\Delta^2} \Delta^J \psi(\Delta) \Big[\int_{0}^{\Delta}  (\Delta-\delta)^K \phi\Big( \frac{r}{2t}+\delta,r,\theta \Big) \dd \delta \Big]\dd \Delta.
	\end{equation}
	Then,  $\calI_{J,K}[\phi,\psi] \in t^{-\lfloor (J+K+1)/2 \rfloor }(t/r)^\infty L^\infty_{\mathrm{loc}}((0,\infty]_t \times \dot{X}_{r,\theta})$.

	In fact,
	$\calI_{J,K}[\phi,\psi](t,\hat{r}t,\theta) \in t^{-\lfloor (J+K+1)/2 \rfloor } \hat{r}^{-\infty} \calA_{\mathrm{loc}}^{0,0}((0,\infty]_t \times [0,\infty)_{\hat{r}}\times \partial X_\theta ).
	$
	\label{lem:adam}
\end{lemma}

\begin{proof}
	We first prove the $L^\infty$-bounds. For the sake of the inductive argument below, it is convenient to note that the definition of $\calI_{J,K}$ makes sense also when $J\in \bbZ$ as long as $J+K\geq -1$. Indeed, 
	\begin{equation*}
	\calI_{J,K}[\phi,\psi] =\int_{-\infty}^\infty  e^{it\Delta^2} \Delta^{J+K+1} \psi(\Delta) \Big[\int_{0}^{1}  (1-s)^K \phi\Big( \frac{r}{2t}+\Delta s,r,\theta \Big) \dd s \Big]\dd \Delta,
	\end{equation*}
	which makes sense under the claimed conditions. 
	
	For $J+K=-1,0$, we have
	\begin{multline}
	|\calI_{J,K}[\phi,\psi]| \leq \max\{1,\sup  \operatorname{supp} \psi ,-\inf \operatorname{supp} \psi \} \lVert \psi \rVert_{L^1} \\ \times \operatorname{sup}_{\min\{0,\inf \operatorname{supp} \psi \}\leq \delta \leq \max\{0,\sup \operatorname{supp} \psi\}} \Big|  \phi\Big( \frac{r}{2t}+\delta,r,\theta \Big) \Big| \\ \in  (t/r)^\infty L^\infty_{\mathrm{loc}}((0,\infty]_t \times \dot{X}_{r,\theta}).
	\end{multline}
	So, the desired estimate holds in this case. 
	
	To handle the $J+K\geq 1$ case, we integrate-by-parts, starting from
	\begin{equation}
	2i t\calI_{J,K}[\phi,\psi] = \int_{-\infty}^\infty \Big( \frac{\partial}{\partial \Delta} e^{it\Delta^2} \Big) \Delta^{J-1} \psi(\Delta) \Big[\int_{0}^{\Delta}  (\Delta-\delta)^K \phi\Big( \frac{r}{2t}+\delta,r,\theta \Big) \dd \delta \Big]\dd \Delta.
	\end{equation}
	Integrating-by-parts yields $-2it \calI_{J,K}[\phi,\psi]=(J-1)\calI_{J-2,K}[\phi,\psi] + \calI_{J-1,K}[\phi,\psi'] + K \calI_{J-1,K-1}[\phi,\psi]$ if $K\neq 0$ and
	\begin{equation}
	-2it \calI_{J,0}[\phi,\psi]=(J-1)\calI_{J-2,0}[\phi,\psi] + \calI_{J-1,0}[\phi,\psi'] + \tilde{\calI}_{J-1}[\phi,\psi]
	\end{equation}
	otherwise,
	where the last of these functions is defined by \cref{eq:misc_hjk}. Note that all of the $\calI_{\bullet,\bullet}$ terms on the right-hand side are defined, with smaller $J+K$.

	So, $\calI_{J,K}[\phi,\psi] \in t^{-\lfloor (J+K+1)/2 \rfloor }(t/r)^\infty L^\infty_{\mathrm{loc}}((0,\infty]_t \times \dot{X}_{r,\theta})$ follows inductively (using \Cref{lem:was_missing} to handle the $\tilde{\calI}_{J-1}$ term).

	In order to control derivatives, we use $L\calI_{J,K}[\phi,\psi]=\calI_{J,K}[L\phi,\psi]$, which holds for all $L\in \operatorname{Diff}(\partial X_\theta)$, and the identities
	\begin{align}
	\partial_{\hat{r}} \calI_{J,K}[\phi,\psi](t,\hat{r}t,\theta) &= 2^{-1}\calI_{J,K}[\partial_\sigma \phi(\sigma,r,\theta),\psi](t,\hat{r}t,\theta) + \hat{r}^{-1}\calI_{J,K}[r\partial_r\phi(\sigma,r,\theta),\psi](t,\hat{r}t,\theta) \\
	\partial_t \calI_{J,K}[\phi,\psi](t,\hat{r}t,\theta) &= i\calI_{J+2,K}[\phi,\psi](t,\hat{r}t,\theta) + t^{-1}\calI_{J,K}[r \partial_r \phi(\sigma,r,\theta),\psi](t,\hat{r}t,\theta) .
	\end{align}
	Using these inductively, and using the $L^\infty$-bounds already proven, the final clause of the lemma follows.
\end{proof}

\begin{lemma}\label{lem:was_missing}
	For each $J\in \bbN$, $\phi \in \calS(\bbR_\sigma ; \calA_{\mathrm{loc}}^0(\dot{X}))$, and $\psi \in C_{\mathrm{c}}^\infty(\bbR)$, consider the function $\tilde{\calI}_{J}[\phi,\psi]:\bbR^+_t\times \dot{X}^\circ_{r,\theta} \to \bbC$ given by
	\begin{equation}
	\tilde{\calI}_{J}[\phi,\psi](t,r,\theta) =\int_{-\infty}^\infty  e^{it\Delta^2} \Delta^{J} \psi(\Delta) \phi\Big( \frac{r}{2t}+\Delta,r,\theta \Big) \dd \Delta.
	\label{eq:misc_hjk}
	\end{equation}
	Then, for each $K\in \bbN$, we have $\tilde{\calI}_{J}[\phi,\psi](t,\hat{r}t,\theta) \in  t^{-\lfloor J/2 \rfloor} \hat{r}^{-K} \calA_{\mathrm{loc}}^{0,0}((0,\infty]_t \times (0,\infty]_{\hat{r}}\times \partial X_\theta )$.
\end{lemma}
\begin{proof}
	We first prove the $L^\infty$-bounds.
	If $J=0$, then
	\begin{equation}
	|\tilde{\calI}_{J}[\phi,\psi]| \leq \lVert \psi \rVert_{L^2} \operatorname{sup}_{\Delta \in \operatorname{supp} \psi} |\phi( (r/2t) + \Delta,r,\theta) | \in (t/r)^\infty L^\infty_{\mathrm{loc}}((0,\infty]_t \times \dot{X}_{r,\theta}) .
	\end{equation}
	If $J\geq 1$, then integration-by-parts yields
	\begin{equation}
	-2 i t\tilde{\calI}_{J}[\phi,\psi] = (J-1)\tilde{\calI}_{J-2}[\phi,\psi]+ \tilde{\calI}_{J-1}[\phi,\psi'] +  \tilde{\calI}_{J-1}[\phi',\psi],
	\end{equation}
	where $\phi'(\sigma,r,\theta) = \partial_\sigma \phi (\sigma,r,\theta)$. Applying this inductively allows the deduction of $\tilde{\calI}_{J}[\phi,\psi] \in t^{-\lfloor J/2 \rfloor} (t/r)^K L^\infty_{\mathrm{loc}}((0,\infty]_t \times \dot{X}_{r,\theta})$ from the $J=0$ case.

	To deduce conormality, we want to prove that the same $L^\infty$-bounds apply to the quantity $(t \partial_t)^j  \partial_{\hat{r}}^k L \tilde{\calI}_{J}[\phi,\psi](t,\hat{r}t,\theta)$ for every $j,k\in \bbN$ and $L\in \operatorname{Diff}(\partial X_\theta)$. Using the identities $L \tilde{\calI}_{J}[\phi,\psi]= \tilde{\calI}_{J}[L\phi,\psi]$,
	\begin{align}
	\partial_{\hat{r}}\tilde{\calI}_{J}[\phi,\psi](t,\hat{r}t,\theta) &= 2^{-1}  \tilde{\calI}_{J}[\partial_\sigma \phi(\sigma,r,\theta) ,\psi](t,\hat{r}t,\theta) + \hat{r}^{-1}\tilde{\calI}_{J}[r \partial_r \phi(\sigma,r,\theta),\psi](t,\hat{r}t,\theta), \\
	\partial_t \tilde{\calI}_{J}[\phi,\psi](t,\hat{r}t,\theta) &= i \tilde{\calI}_{J+2}[\phi,\psi](t,\hat{r}t,\theta) + t^{-1}\tilde{\calI}_{J}[r \partial_r \phi(\sigma,r,\theta),\psi](t,\hat{r}t,\theta),
	\end{align}
	these bounds follow from those already proven.
\end{proof}

\section{Remaining contribution}
\label{sec:mid}

Finally, we examine $I_\pm[\phi](t,x) = 2\int_0^\infty e^{i \sigma^2 t\pm i \sigma r(x)} \phi(\sigma,r,\theta) \sigma  \dd \sigma $ for $\phi$ polyhomogeneous on $\dot{X}^{\mathrm{sp}}_{\mathrm{res}}$ and supported near $\mathrm{tf} \cap \mathrm{bf}$. Specifically, we consider the case $\phi(\sigma,r,\theta) = \varphi(\sigma,\sigma r,\theta)$ for
\begin{equation}
	\varphi(\sigma,\lambda,\theta) \in \calA_{\mathrm{c}}^{(\calE,\alpha),(\calF,\beta)}([0,\Sigma)_\sigma \times (\Lambda,\infty]_\lambda\times \partial X_\theta)
\end{equation}
for some $\Sigma,\Lambda>0$, index sets $\calE,\calF$, and $\alpha,\beta\in \bbR$. Here, $\calE$ is the index set at $\sigma = 0$, i.e.\ at $\mathrm{tf}$, and $\calF$ is the index set at $\lambda=\infty$, i.e.\ at $\mathrm{bf}$.
In order to formulate the asymptotics of $I_\pm$, it is useful to work with the manifold $\dot{M} / \mathrm{nf}$, which is defined analogously to $M/\mathrm{nf}$ with $\dot{X}$ in place of $X$, and whose faces we label correspondingly.
Recall that $\dot{C}_1 = [0,\infty]_\tau \times \dot{X}$.
We summarize the results of this section in the following proposition:
\begin{proposition}
	Given the setup above, $I_+[\phi]\in \calA^{\infty,(\calE+2,\alpha+2),(0,0)}_{\mathrm{loc}}(\dot{C}_1)$, where the index sets are specified at $\{\tau=\infty\}$, $\mathrm{parF}$, and $\{\tau=0\}$ respectively, and $I_-[\phi] = \exp(-i r^2/(4t) )   \tilde{I}_-[\phi] + I_{-,\mathrm{phg}}[\phi]$
	for some
	\begin{equation}
	\tilde{I}_-[\phi] \in    \calA^{\infty,(\calE+2,\alpha+2),(\calF+1/2,\beta+1/2)}_{\mathrm{loc}}(\dot{M}/ \mathrm{nf})
	\label{eq:h33gv1}
	\end{equation}
	and $I_{-,\mathrm{phg}}[\phi]\in \calA^{\infty,(\calE+2,\alpha+2),(0,0)}_{\mathrm{loc}}(\dot{C}_1)$, where the index sets on $\dot{M}/\mathrm{nf}$ are specified at $\mathrm{kf}$, $\mathrm{parF}$, and $\mathrm{dilF}$ (and one has smoothness at $\Sigma$).
	Moreover, if $\chi \in C_{\mathrm{c}}^\infty(\bbR)$ satisfies $\operatorname{supp} \chi \Subset (-1,1)$, then $\chi(2t \rho(x) \Sigma ) \tilde{I}_-[\phi]$ is Schwartz.
	\label{prop:corner}
\end{proposition}
\begin{proof}
	The statement for $I_+[\phi]$ comes immediately from \Cref{prop:tfbf_large_time} and \Cref{prop:tfbf_I+rem}.

	The partial compactifications $\bbR^+_t\times \dot{X} \hookrightarrow (\dot{M}/ \mathrm{nf}) \backslash ( \mathrm{dilF} \cup \Sigma)$ and $\bbR^+_t\times \dot{X} \hookrightarrow \dot{C}_1 \backslash \Sigma $ are equivalent, in the sense that the identity map on the interior extends to a diffeomorphism $(\dot{M}/ \mathrm{nf}) \backslash ( \mathrm{dilF} \cup \Sigma) \cong \dot{C}_1 \backslash \Sigma$.
	So, \Cref{prop:tfbf_large_time} tells us that
	\begin{equation}
		I_-[\phi]\in \calA^{\infty,(\calE+2,\alpha+2)}_{\mathrm{loc}}((\dot{M}/ \mathrm{nf}) \backslash (\mathrm{dilF}\cup \Sigma)),
	\end{equation}
	where the index sets are specified at $\mathrm{kf},\mathrm{parF}$, respectively.
	On the other hand, the partial compactifications $\bbR^+_t\times \dot{X}\hookrightarrow (\dot{M} / \mathrm{nf}) \backslash \mathrm{kf}$ and $\bbR^+_t\times \dot{X}\hookrightarrow [0,\infty)_\tau \times (0,\infty]_s\times \partial X_\theta $ given by $(t,r,\theta) \mapsto (t/r^2,t/r,\theta)$ are equivalent. So \Cref{prop:I-} tells us that
	\begin{equation}
	\tilde{I}_-[\phi] \in  \calA_{\mathrm{loc}}^{(\calE+2,\alpha+2),(\calF+1/2,\beta+1/2)}((\dot{M}/\mathrm{nf}) \backslash \mathrm{kf}).
	\end{equation}
	This gives \cref{eq:h33gv1}. 
	
	The claimed space for $I_{-,\mathrm{phg}}[\phi]$ follows by combining the
	description of $I_{-,\mathrm{phg}}$ in \Cref{prop:I-} away from $\mathrm{kf}$
	with the description of $I_-[\phi]$ near $\mathrm{kf}$ given by
	\Cref{prop:tfbf_large_time}; the oscillatory term is Schwartz at $\mathrm{kf}$.

	The last clause of this proposition follows from the last clause of \Cref{prop:I-}.
\end{proof}

\subsection{Control for very large time}
The following establishes control near $\mathrm{kf}$:
\begin{proposition}
	For $\varphi(\sigma,\lambda,\theta) \in \calA_{\mathrm{c}}^{(\calE,\alpha),(\calF,\beta)}([0,\Sigma)_\sigma \times (\Lambda,\infty]_\lambda\times \partial X_\theta)$ and $\phi(\sigma,r,\theta) = \varphi(\sigma,\sigma r,\theta)$, we have $I_\pm[\phi]( t,x) = \rho^2 \bar{I}_\pm[\varphi](t \rho^2,x)$ for some
	\begin{equation}
		\bar{I}_\pm[\varphi](\tau,x) \in \calA_{\mathrm{loc}}^{\infty,(\calE,\alpha) }(\dot{C}_1\backslash \Sigma ),
		\label{eq:misc_121}
	\end{equation}
	where $\dot{C}_1\backslash \Sigma = (0,\infty]_\tau \times \dot{X}$, where $\calE$ is the index set at $(0,\infty]_\tau \times \partial X_\theta$, i.e.\ as $r\to\infty$, and the $\infty$ denotes Schwartz behavior as $\tau\to\infty$. Moreover, the expansion at $(0,\infty]_\tau\times \partial X_\theta$ is given by
	\begin{equation}
	\bar{I}_\pm[\varphi](\tau,x) \sim  \sum_{(j,k)\in \calE, \Re j \leq \alpha}  \frac{2}{r^j} \sum_{\kappa=0}^k  (-1)^\kappa\binom{k}{\kappa} \log^\kappa(r) \int_\Lambda^\infty e^{i\lambda^2 \tau \pm i \lambda}  \varphi_{j,k}(\lambda,\theta) \lambda^{1+j} \log^{k-\kappa}(\lambda) \dd \lambda
	\label{eq:misc_120}
	\end{equation}
	where $\varphi(\sigma,\lambda,\theta) \sim \sum_{(j,k)\in \calE,\Re j\leq \alpha} \varphi_{j,k}(\lambda,\theta)  \sigma^j (\log \sigma)^k$ is the polyhomogeneous expansion of $\varphi$ as $\sigma \to 0^+$, so that
	\begin{equation}
			\varphi_{j,k}(\lambda,\theta) \in \calA_{\mathrm{c}}^{(\calF,\beta)}((\Lambda,\infty]_\lambda\times \partial X_\theta).
	\end{equation}
	The integrals on the right-hand side of \cref{eq:misc_120} are well-defined oscillatory integrals (though not necessarily absolutely convergent), e.g.\ via formal integration-by-parts.
	\label{prop:tfbf_large_time}
\end{proposition}
\begin{proof}
	We have $I_\pm[\phi](t,x) = \rho^2 \bar{I}_\pm[\varphi](t\rho^2,x)$ for $\bar{I}_\pm[\varphi](\tau,x)$ defined by
	\begin{equation}
		\bar{I}_\pm[\varphi](\tau,x) = 2\int_\Lambda^\infty e^{i \lambda^2 \tau \pm i \lambda } \varphi(\lambda /r, \lambda,\theta) \lambda \dd \lambda.
	\end{equation}
	Defining $\varphi_\gamma(\sigma,\lambda,\theta) = \sigma^{-\gamma} (\varphi(\sigma,\lambda,\theta) - \sum_{(j,k)\in \calE,\Re j\leq \gamma} \varphi_{j,k}(\lambda,\theta)  \sigma^j \log^k \sigma)$
	for $\gamma\in \bbR$ with $\gamma\leq \alpha$, we have
	\begin{equation}
		\varphi_\gamma \in \smash{\calA_{\mathrm{c}}^{0,(\calF,\beta)}}([0,\Sigma)_\sigma \times (\Lambda,\infty]_\lambda\times \partial X_\theta).
	\end{equation}
	Let $\chi\in C_{\mathrm{c}}^\infty(\bbR)$ be identically $1$ on $[-1,1]$. Then, we can write $\varphi(\sigma,\lambda,\theta) = \chi(\sigma \Sigma^{-1}) \varphi(\sigma,\lambda,\theta)$ for all $\sigma ,\lambda>0$ and $\theta\in \partial X$, so
	\begin{equation}
		\bar{I}_\pm[\varphi](\tau,x) = \sum_{(j,k)\in \calE, \Re j \leq \gamma}  \frac{2}{r^j} \sum_{\kappa=0}^k (-1)^\kappa \binom{k}{\kappa} \log^\kappa(r)  \bar{I}_{\pm,j,k-\kappa }[\varphi](\tau,x) + r^{-\gamma} \bar{I}_{\pm,\gamma}[\varphi](\tau,x)
	\end{equation}
	for
	\begin{multline}
		\bar{I}_{\pm,j,k-\kappa}[\varphi](\tau,x) = \int_\Lambda^\infty e^{i\lambda^2 \tau \pm i \lambda} \chi \Big( \frac{\lambda}{r\Sigma} \Big) \varphi_{j,k}(\lambda,\theta) \lambda^{1+j} \log^{k-\kappa}(\lambda) \dd \lambda, \\
		\bar{I}_{\pm,\gamma}[\varphi](\tau,x) = \int_\Lambda^\infty e^{i\lambda^2 \tau \pm i \lambda} \chi \Big( \frac{\lambda}{r\Sigma} \Big) \varphi_{\gamma}\Big(\frac{\lambda}{r},\lambda,\theta\Big) \lambda^{1+\gamma} \dd \lambda .
	\end{multline}
	By \Cref{lem:tfbf_large_time}, $\bar{I}_{\pm,j,k}(\tau,x) \in \calA_{\mathrm{loc}}^{\infty,(0,0)}((0,\infty]_\tau \times \dot{X}_x)$
	and $\bar{I}_{\pm,\gamma}(\tau,x) \in \calA_{\mathrm{loc}}^{\infty,0}((0,\infty]_\tau \times \dot{X}_x)$, where the $(0,0)$ denotes the index set at $(0,\infty]_\tau \times\partial \dot{X}_\theta$ and $\infty$ denotes Schwartz behavior as $\tau\to\infty$.

	So, $\bar{I}_\pm[\varphi](\tau,x) \in  \calA_{\mathrm{loc}}^{\infty,(\calE,\gamma) }(\dot{C}_1\backslash \Sigma )$, and since $\gamma\leq \alpha$ was arbitrary we conclude \cref{eq:misc_121}.
	The explicit expansion follows from the argument above and the second half of \Cref{lem:tfbf_large_time}.
\end{proof}

\begin{lemma}
	For $\varphi \in \bigcup_{K\in \bbN} \calA_{\mathrm{c}}^{0,-K}([0,\Sigma)_\sigma\times (\Lambda,\infty]_\lambda\times \partial X_\theta)$, for $j\in \bbC$ and $k\in \bbN$, and for $\chi \in C_{\mathrm{c}}^\infty(\bbR)$, consider
	\begin{equation}
		\calI_{\pm,j,k} [\varphi,\chi ](\tau,r,\theta) = \int_\Lambda^\infty e^{i \lambda^2 \tau \pm i \lambda} \chi \Big( \frac{\lambda}{r\Sigma} \Big) \varphi\Big(\frac{\lambda}{r},\lambda,\theta\Big) \lambda^j \log^k( \lambda ) \dd \lambda .
	\end{equation}
	Then, $\calI_{\pm,j,k}[\varphi,\chi](\tau,x) \in \calA_{\mathrm{loc}}^{\infty,0}((0,\infty]_\tau \times \dot{X}_x)$, where $0$ is the order at $(0,\infty]_\tau \times \partial X$, i.e.\ as $r\to\infty$, and the $\infty$ denotes Schwartz behavior as $\tau\to\infty$.
	If $\varphi(\sigma,\lambda,\theta)=\varphi(\lambda,\theta)$ does not depend on $\sigma$, and if $0\notin \operatorname{supp}(1-\chi)$, then
	\begin{equation}
		\calI_{\pm,j,k}[\varphi,\chi](\tau,x) \in \calA_{\mathrm{loc}}^{\infty,(0,0)}((0,\infty]_\tau \times \dot{X}_x),
	\end{equation}
	and moreover
	\begin{equation}
		\calI_{\pm,j,k}[\varphi,\chi](\tau,x) - \int_\Lambda^\infty e^{i \lambda^2 \tau \pm i \lambda}  \varphi(\lambda,\theta) \lambda^j \log^k (\lambda ) \dd \lambda \in \calA_{\mathrm{loc}}^{\infty,\infty}((0,\infty]_\tau \times \dot{X}_x).
		\label{eq:jjjj}
	\end{equation}
	The second term on the right-hand side is a well-defined oscillatory integral, even though it may not be absolutely convergent.
	\label{lem:tfbf_large_time}
\end{lemma}
\begin{proof}
	It suffices to consider the case  $\varphi \in \calA_{\mathrm{c}}^{0,0}([0,\Sigma)_\sigma\times (\Lambda,\infty]_\lambda\times \partial X_\theta)$, as we can write $\calI_{\pm,j,k}[ \varphi,\chi]=\calI_{\pm,j+K,k}[ \lambda^{-K} \varphi,\chi]$.

	We first prove that $\calI_{\pm,j,k}[\varphi,\chi](\tau,x) \in  \tau^{-\infty} L^\infty_{\mathrm{loc}}((0,\infty]_\tau \times \dot{X}_x)$. First of all, if $\Re j<-1$, then $\calI_{\pm,j,k}[\varphi,\chi](\tau,r,\theta) \in   L^\infty_{\mathrm{loc}}((0,\infty]_\tau \times \dot{X}_x)$, as follows immediately from an $L^\infty$-bound. Using
	\begin{equation}
			2i\tau \calI_{\pm,j,k} [\varphi,\chi ](\tau,r,\theta) =  \int_\Lambda^\infty \Big[ \frac{\partial}{\partial \lambda} e^{i \lambda^2 \tau } \Big]  e^{\pm i \lambda } \chi \Big( \frac{\lambda}{r\Sigma} \Big)\varphi\Big(\frac{\lambda}{r},\lambda,\theta\Big) \lambda^{j-1} \log^k (\lambda)  \dd \lambda,
	\end{equation}
	integrating-by-parts yields
	\begin{multline}
		- 2 i \tau \calI_{\pm,j,k} [\varphi,\chi ] = \pm i \calI_{\pm,j-1,k}[\varphi,\chi] + r^{-1} \Sigma^{-1} \calI_{\pm,j-1,k}[\varphi,\chi']+ \calI_{\pm,j-2,k}[\sigma \partial_\sigma \varphi(\sigma,\lambda,\theta) ,\chi] \\ + \calI_{\pm,j-2,k}[\lambda \partial_\lambda \varphi(\sigma,\lambda,\theta) ,\chi]  + (j-1) \calI_{\pm,j-2,k}[\varphi,\chi] + k \calI_{\pm,j-2,k-1}[\varphi,\chi],
		\label{eq:misc_134}
	\end{multline}
	with that last term omitted when $k=0$.
	Each term on the right-hand side has the same form as the original integral (possibly times an extra $L^\infty_{\mathrm{loc}}$ factor), but with $j$ with smaller real part. Since the left-hand side of \cref{eq:misc_134} has one extra factor of $\tau$, this sets up an inductive argument to conclude $O(\tau^{-\infty})$ decay from the $L^\infty_{\mathrm{loc}}$ estimate already proven in the $\Re j<-1$ case.

	Now suppose that $0\notin \operatorname{supp} \chi$. Then, an $L^\infty$-bound yields immediately that, if $\Re j<-1$, then
	\begin{equation}
		\calI_{\pm,j,k}[\varphi,\chi](\tau,x) \in   r^{\Re j+1+\epsilon} L^\infty_{\mathrm{loc}}((0,\infty]_\tau \times \dot{X}_x)
		\label{eq:brgr}
	\end{equation}
	for any $\epsilon>0$. So, in this case the inductive argument above yields additionally rapid decay as $r\to\infty$, i.e.\ that $\calI_{\pm,j,k}[\varphi,\chi](\tau,r,\theta) \in \tau^{-\infty} r^{-\infty} L^\infty_{\mathrm{loc}}((0,\infty]_\tau \times \dot{X}_x)$.

	We now prove two sets of estimates on derivatives of $\calI_{\pm,j,k}[\varphi,\chi]$. First of all, if $n,m\in \bbN$ and $L\in \operatorname{Diff}(\partial X_\theta)$, then
	\begin{equation}
		\partial_\tau^n (r\partial_r)^m  L \calI_{\pm,j,k} [\varphi,\chi](\tau,r,\theta) \in  \tau^{-\infty}  L^\infty_{\mathrm{loc}}((0,\infty]_\tau \times \dot{X}_x) .
	\end{equation}
	Secondly, if $0\notin \operatorname{supp}(1-\chi)$ and if $\varphi(\sigma,\lambda,\theta) = \varphi(\lambda,\theta)$ does not depend on $\sigma$, then, for $m\in \bbN^+$,
	\begin{equation}
	\partial_\tau^n \partial_r^m L \calI_{\pm,j,k} [\varphi,\chi](\tau,r,\theta) \in  \tau^{-\infty} r^{-\infty} L^\infty_{\mathrm{loc}}((0,\infty]_\tau \times \dot{X}_x).
	\label{eq:g3111}
	\end{equation}
	These follow from applying repeatedly the identities
	\begin{equation}
		L \calI_{\pm,j,k} [\varphi,\chi] = \calI_{\pm,j,k} [L\varphi,\chi], \qquad \partial_\tau \calI_{\pm,j,k} [\varphi,\chi] = i\calI_{\pm,j+2,k} [\varphi,\chi]
	\end{equation}
	\begin{equation}
		r\partial_r\calI_{\pm,j,k} [\varphi,\chi] =  -\calI_{\pm,j+1,k}[\sigma \partial_\sigma \varphi(\sigma,\lambda,\theta),\chi] - \Sigma^{-1}r^{-1}  \calI_{\pm,j+1,k} [\varphi,\chi' ].
		\label{eq:misc_131}
	\end{equation}
	For example, if $\varphi(\sigma,\lambda,\theta) = \varphi(\lambda,\theta)$ does not depend on $\sigma$, then the first term on the right-hand side of \cref{eq:misc_131} is zero, so $\partial_r \calI_{\pm,j,k} [\varphi,\chi] =   - \Sigma^{-1}r^{-2}  \calI_{\pm,j+1,k} [\varphi,\chi' ]$, and if $\chi=1$ identically near the origin, then $0\notin \operatorname{supp} \chi'$, so we can apply the improved $L^\infty$-estimates \cref{eq:brgr} that apply to $\calI_{\pm,j+1,k} [\varphi,\chi' ]$ in this case.

	These estimates say
	\begin{equation}
		\calI_{\pm,j,k}[\varphi,\chi](\tau,x) \in \calA_{\mathrm{loc}}^{\infty,0}((0,\infty]_\tau \times \dot{X}_x)
	\end{equation}
	and, under the stronger set of assumptions, 	$\calI_{\pm,j,k}[\varphi,\chi](\tau,x) \in \calA_{\mathrm{loc}}^{\infty,(0,0)}((0,\infty]_\tau \times \dot{X}_x)$. The fact that only the leading-order term in the $r\to\infty$ expansion contributes in the latter case follows from the Schwartz estimates \cref{eq:g3111} that apply when taking at least one derivative in $r$. To get \cref{eq:jjjj}, one only needs to verify that the leading-order term in the $r\to\infty$ expansion has the desired form, and this can be proven by integrating-by-parts until one is left working with absolutely convergent integrals. Then, the leading-order term is computed by replacing $\chi$ with $1$.
\end{proof}

\subsection{Asymptotics elsewhere, $I_+$}
To complete our discussion of $I_+$, we prove:
\begin{proposition}
		For $\varphi(\sigma,\lambda,\theta) \in \calA_{\mathrm{c}}^{(\calE,\alpha),(\calF,\beta)}([0,\Sigma)_\sigma \times (\Lambda,\infty]_\lambda\times \partial X_\theta)$ and $\phi(\sigma,r,\theta) = \varphi(\sigma,\sigma r,\theta)$, we have $I_+[\phi]( t,r,\theta) = 2\rho^2 \bar{I}_+[\varphi](t /r^2,r,\theta)$ for some
		\begin{equation}
			\bar{I}_+[\varphi](\tau,r,\theta)\in \calA_{\mathrm{loc}}^{(\calE,\alpha),(0,0)}(\dot{C}_1 \backslash \mathrm{kf}) .
			\label{eq:misc_axe}
		\end{equation}
		\label{prop:tfbf_I+rem}
\end{proposition}
\begin{proof}
	We have $I_+[\phi]( t,r,\theta) = 2\rho^2 \bar{I}_+[\varphi](t /r^2,r,\theta)$ for $
	\bar{I}_+[\varphi](\tau,r,\theta)$ defined by
	\begin{equation}
		\bar{I}_+[\varphi](\tau,r,\theta) = \int_\Lambda^\infty e^{i \lambda^2 \tau + i \lambda} \varphi\Big( \frac{\lambda}{r},\lambda ,\theta \Big) \lambda \dd \lambda .
	\end{equation}
	Now let $\varphi(\sigma,\lambda,\theta) \sim \sum_{(j,k) \in \calE, \Re j \leq \alpha}  \varphi_{j,k}(\lambda,\theta) \sigma^j \log^k \sigma $ denote the polyhomogeneous expansion of $\varphi(\sigma,\lambda,\theta)$ at $\sigma=0$, so
	\begin{equation}
		\varphi_{j,k} \in \calA_{\mathrm{c}}^{(\calF,\beta)} ((\Lambda,\infty]_\lambda\times \partial X_\theta)
	\end{equation}
	and, letting $\varphi_\gamma = \sigma^{-\gamma} (\varphi -  \sum_{(j,k) \in \calE, \Re j \leq \gamma}  \varphi_{j,k}(\lambda,\theta) \sigma^j \log^k \sigma)$, we have $\varphi_\gamma \in \calA_{\mathrm{c}}^{0,(\calF,\beta)}([0,\Sigma)_\sigma \times (\Lambda,\infty]_\lambda\times \partial X_\theta)$.
	Let $\chi\in C_{\mathrm{c}}^\infty(\bbR)$ be identically $1$ on $[-1,+1]$. Then, $\varphi(\sigma,\lambda,\theta) = \chi(\sigma \Sigma^{-1}) \varphi(\sigma,\lambda,\theta)$, so we can write
	\begin{equation}
		\bar{I}_+[\varphi](\tau,r,\theta) = \sum_{(j,k)\in \calE, \Re j \leq \gamma}  \Big( \frac{1}{r}\Big)^j \sum_{\kappa=0}^k \binom{k}{\kappa} \log^\kappa\Big(\frac{1}{r}\Big)  \bar{I}_{+,j,k-\kappa }(\tau,r,\theta) +  \Big(\frac{1}{r}\Big)^{\gamma} \bar{I}_{+,\gamma}(\tau,r,\theta)
	\end{equation}
	for
	\begin{multline}
		\bar{I}_{+,j,k-\kappa}(\tau,r,\theta) = \int_\Lambda^\infty e^{i\lambda^2 \tau+ i \lambda} \chi \Big( \frac{\lambda}{r \Sigma} \Big) \varphi_{j,k}(\lambda,\theta) \lambda^{1+j} \log^{k-\kappa}(\lambda) \dd \lambda, \\
		\bar{I}_{+,\gamma}[\varphi](\tau,r,\theta) = \int_\Lambda^\infty e^{i\lambda^2 \tau+ i \lambda} \chi \Big( \frac{\lambda }{r\Sigma} \Big) \varphi_{\gamma}\Big( \frac{\lambda }{r},\lambda ,\theta \Big) \lambda^{1+\gamma} \dd \lambda .
	\end{multline}
	We now appeal to \Cref{lem:alek} to conclude that 	$\bar{I}_+[\varphi](\tau,r,\theta)\in \calA_{\mathrm{loc}}^{(\calE,\gamma),(0,0)}(\dot{C}_1 \backslash \mathrm{kf})$. Since $\gamma\leq \alpha$ was arbitrary, we can conclude \cref{eq:misc_axe}.
\end{proof}

\begin{remark}
	Using the explicit expansions in \Cref{lem:alek}, the proof of \Cref{prop:tfbf_I+rem} shows that
	the expansion of $\bar{I}_+[\varphi](\tau,r,\theta)$ as $r\to\infty$ is given by
	\begin{equation}
	\bar{I}_+[\varphi]\sim \sum_{(j,k)\in \calE, \Re j \leq \alpha}  \Big(\frac{1}{r}\Big)^{j}  \log^k\Big(\frac{1}{r}\Big) \sum_{K\geq k, (j,K) \in \calE} \binom{K}{k}  \int_\Lambda^\infty
	e^{i \lambda^2 \tau + i\lambda} \varphi_{j,K}(\lambda,\theta) \lambda^{1+j} \log^{K-k} (\lambda) \dd \lambda.
	\label{eq:misc_tiki}
	\end{equation}
\end{remark}
\begin{lemma}
	Let $\Sigma>0$ and $\chi \in C_{\mathrm{c}}^\infty([0,\infty ))$.
	Suppose that $\gamma \in \bbC$, $k\in \bbN$, and that $\phi \in \calA^{0,0,0,0}_{\mathrm{c}}([0,\Sigma)_\sigma  \times (\Lambda,\infty]_\lambda \times (0,\infty]_r \times [0,\infty)_\tau\times \partial X_\theta )$. Consider the integral
	\begin{equation}
		\calI_{+,\gamma,k}[\phi,\chi](\tau,r,\theta) = \int_\Lambda^\infty e^{i \lambda^2 \tau + i\lambda} \phi \Big( \frac{\lambda }{r} ,\lambda,r, \tau ,\theta\Big) \chi\Big( \frac{\lambda }{r \Sigma } \Big) \lambda^\gamma \log^k(\lambda) \dd \lambda.
		\label{eq:misc_idf}
	\end{equation}
	Then, $\calI_{+,\gamma,k}[\phi,\chi](\tau,r,\theta) \in \calA^{0,0}_{\mathrm{loc}}(\dot{C}_1\backslash \mathrm{kf})$.
	If $\phi(\sigma,\lambda,r,\tau,\theta) = \phi(\lambda,\theta)$ does not depend on any of $\sigma,\tau,r$, and if $0\notin \operatorname{supp}(1-\chi)$, then
	\begin{equation}
		\calI_{+,\gamma,k}[\phi,\chi](\tau,r,\theta)
		- \int_\Lambda^\infty e^{i \lambda^2 \tau + i\lambda} \phi(\lambda,\theta) \lambda^\gamma \log^k(\lambda) \dd \lambda
		\in \calA^{\infty,(0,0)}_{\mathrm{loc}}(\dot{C}_1\backslash \mathrm{kf}),
		\label{eq:misc_kiki}
	\end{equation}
	the integral on the left-hand side being a well-defined oscillatory integral. Here, the $\infty$ denotes Schwartz behavior as $r\to \infty$ for $\tau$ fixed.
	\label{lem:alek}
\end{lemma}
\begin{proof}
	We first prove that
	\begin{equation}
	\calI_{+,\gamma,k}[\phi,\chi](\tau,r,\theta)\in L^\infty_{\mathrm{loc}}([0,\infty)_\tau\times (0,\infty]_r\times \partial X_\theta ).
	\label{eq:misc_151}
	\end{equation}
	If $\Re \gamma <-1$, then $\lVert \calI_{+,\gamma,k}[\phi,\chi] \rVert_{L^\infty} \leq \lVert \phi \rVert_{L^\infty} \lVert \chi \rVert_{L^\infty} \int_\Lambda^\infty \lambda^{\Re \gamma} (\log \lambda)^k \dd \lambda < \infty$. In order to prove the claim for $\Re \gamma\geq -1$, we use an inductive argument.
	The key is the identity $\exp(i \lambda^2 \tau + i \lambda )  = -i (2\lambda\tau+1)^{-1} \partial_\lambda \exp(i \lambda^2 \tau + i \lambda )$, hence
	\begin{equation}
		\calI_{+,\gamma,k}[\phi,\chi](\tau,r,\theta) = \int_\Lambda^\infty \Big[  \frac{-i}{2\lambda \tau+1}
		\frac{\partial}{\partial \lambda} e^{i \lambda^2 \tau + i \lambda } \Big] \phi \Big( \frac{\lambda}{r}, \lambda ,r, \tau,\theta\Big) \chi\Big( \frac{\lambda }{r \Sigma} \Big) \lambda^\gamma \log^k (\lambda) \dd \lambda.
		\label{eq:misc_idf2}
	\end{equation}
	Integrate-by-parts, noting that no boundary terms arise:
	\begin{multline}
		\calI_{+,\gamma,k}[\phi,\chi] =\calI_{+,\gamma-1,k}[\varphi,\chi] +  \calI_{+,\gamma-1,k}[\phi_0,\chi] +  \calI_{+,\gamma-1,k}[ \psi,\sigma \chi'(\sigma)] \\+  \gamma \calI_{+,\gamma-1,k}[ \psi,\chi] + k\calI_{+,\gamma-1,k-1}[ \psi,\chi]
		\label{eq:misc_k55}
	\end{multline}
	(with the last term omitted when $k=0$),
	where
	\begin{align}
		\begin{split}
			\varphi(\sigma,\lambda,r,\tau,\theta) &= \frac{i}{2\lambda \tau+1}  \Big(  \sigma \frac{\partial}{\partial \sigma} \phi(\sigma,\lambda,r,\tau,\theta) + \lambda\frac{\partial}{\partial \lambda} \phi(\sigma,\lambda,r,\tau,\theta) \Big)  \\
			&\in \calA^{0,0,0,0}_{\mathrm{c}}([0,\Sigma)_\sigma \times  (\Lambda,\infty]_\lambda\times (0,\infty]_r \times [0,\infty)_\tau \times \partial X_\theta )
		\end{split} \\
		\begin{split}
			\phi_0(\sigma,\lambda,r,\tau,\theta) &=  \frac{-2i \lambda \tau  \phi(\sigma,\lambda,r,\tau,\theta) }{(2\lambda\tau+1)^2}\in  \calA^{0,0,0,0}_{\mathrm{c}}([0,\Sigma)_\sigma \times  (\Lambda,\infty]_\lambda\times  (0,\infty]_r \times [0,\infty)_\tau \times \partial X_\theta ),
		\end{split}
	\end{align}
	and $\psi = i (2\lambda \tau+1)^{-1} \phi \in \calA^{0,0,0,0}_{\mathrm{c}}([0,\Sigma)_\sigma \times (\Lambda,\infty]_\lambda \times (0,\infty]_r   \times [0,\infty)_\tau\times \partial X_\theta )$.
	Since the right-hand side of \cref{eq:misc_k55} involves only $\gamma-1$ in place of $\gamma$, this identity can be used repeatedly to reduce the to-be-proven claim \cref{eq:misc_151} to the $\Re \gamma<-1$ case.

	A modification of this argument shows that if $0\notin \operatorname{supp} \chi$, then
	\begin{equation}
		\calI_{+,\gamma,k}[\phi,\chi](\tau,r,\theta)\in  (1/r)^\infty L^\infty_{\mathrm{loc}}([0,\infty)_\tau\times (0,\infty]_r\times \partial X_\theta ).
		\label{eq:misc_155}
	\end{equation}
	Indeed, in this case, the integrand in \cref{eq:misc_idf2} is supported for $\lambda \sim r$. Thus, an initial $L^\infty$-bound yields $\calI_{+,\gamma,k}[\phi,\chi](\tau,r,\theta)\in r^{\Re \gamma + 1} L^\infty_{\mathrm{loc}}([0,\infty)_\tau\times (0,\infty]_r\times \partial X_\theta )$. The inductive argument then shows that the same bound holds with $\gamma+K$ in place of $\gamma$, for all $K\in \bbN$, which then implies \cref{eq:misc_155}.

	In order to deduce that  $\calI_{+,\gamma,k}[\phi,\chi] \in \calA^{0,0}_{\mathrm{loc}}([0,\infty)_\tau\times (0,\infty]_r\times \partial X_\theta ) $, we want to show that
	\begin{equation}
		(r\partial_r)^j (\tau \partial_\tau)^\kappa L \calI_{+,\gamma,k}[\phi,\chi](\tau,r,\theta) \in  L^\infty_{\mathrm{loc}}([0,\infty)_\tau\times (0,\infty]_{r}\times \partial X_\theta )
		\label{eq:misc_147}
	\end{equation}
	for all $j,\kappa\in \bbN$ and $L\in \operatorname{Diff}(\partial X)$. As elsewhere, we use $ L \calI_{+,\gamma,k}[\phi,\chi] = \calI_{+,\gamma,k}[L\phi,\chi]$, and now we have $\tau \partial_\tau \calI_{+,\gamma,k}[\phi,\chi] =  i \tau \calI_{+,\gamma+2,k}[\phi,\chi] +   \calI_{+,\gamma,k}[\vartheta,\chi]$ for $\vartheta(\sigma,\lambda,r,\tau,\theta) = \tau \partial_\tau \phi(\sigma,\lambda,r,\tau,\theta)$ and $ r \partial_r \calI_{+,\gamma,k}[\phi,\chi] = \calI_{+,\gamma,k}[\varpi,\chi] - \calI_{+,\gamma,k}[\phi,\sigma \chi'(\sigma) ]$
	for $\varpi(\sigma,\lambda,r,\tau,\theta) = (- \sigma \partial_\sigma +r\partial_r)\phi(\sigma,\lambda,r,\tau,\theta)$. So, the desired bounds in \cref{eq:misc_147} follow inductively from the $L^\infty$-bounds proven above.

	Suppose now that $\phi(\sigma,\lambda,r,\tau,\theta) = \phi(\lambda,\theta)$ does not depend on any of $\sigma,\tau,r$. Then, $ \partial_\tau \calI_{+,\gamma,k}[\phi,\chi] =  i  \calI_{+,\gamma+2,k}[\phi,\chi] $ and  $\partial_r \calI_{+,\gamma,k}[\phi,\chi] = - r^{-1}  \calI_{+,\gamma,k}[\phi,\sigma \chi'(\sigma) ]$. So, a similar inductive argument to the above shows that
	\begin{align}
		\partial_\tau^\kappa L \calI_{+,\gamma,k}[\phi,\chi] &\in  L^\infty_{\mathrm{loc}}([0,\infty)_\tau\times (0,\infty]_r\times \partial X_\theta )  \\
		\partial_\tau^\kappa \partial_r^{j+1}L \calI_{+,\gamma,k}[\phi,\chi] &\in  r^{-\infty} L^\infty_{\mathrm{loc}}([0,\infty)_\tau\times (0,\infty]_r\times \partial X_\theta )
	\end{align}
	for all $j,\kappa\in \bbN$ and $L\in \operatorname{Diff}(\partial X_\theta)$. These estimates suffice to show that
	\begin{equation}
		\calI_{+,\gamma,k}[\phi,\chi](\tau,r,\theta)
	\in \calA^{(0,0),(0,0)}_{\mathrm{loc}}([0,\infty)_\tau\times (0,\infty]_r\times \partial X_\theta )
	\end{equation}
	and that only the $O(1)$ term in the $r\to\infty$ expansion is nontrivial. It can be checked, e.g.\ via the integration-by-parts argument above, that this leading term is that specified by \cref{eq:misc_kiki}.
\end{proof}

\subsection{Asymptotics elsewhere, remaining case}
Finally, to complete our discussion of $I_-$:
\begin{proposition}
	For $\varphi(\sigma,\lambda,\theta) \in \calA_{\mathrm{c}}^{(\calE,\alpha),(\calF,\beta)}([0,\Sigma)_\sigma \times (\Lambda,\infty]_\lambda\times \partial X_\theta)$ and $\phi(\sigma,r,\theta) = \varphi(\sigma,\sigma r,\theta)$, we have
	\begin{equation}
	I_-[\phi](t,x) =  I_{-,\mathrm{phg}}[\varphi](t/r^2,r,\theta) +  I_{-,\mathrm{osc}}[\varphi](t/r^2,t/r,\theta)
	\label{eq:misc_162}
	\end{equation}
	for
	\begin{equation}
	I_{-,\mathrm{phg}}[\varphi](\tau, r,\theta  ) \in r^{-2}\calA_{\mathrm{loc}}^{(\calE,\alpha),(0,0)}(\dot{C}_1\backslash \mathrm{kf} )
	\label{eq:misc_163}
	\end{equation}
	and $I_{-,\mathrm{osc}}[\varphi](\tau, s,\theta  ) \in e^{-i/4\tau} \tau^{1/2} s^{-2} \calA_{\mathrm{loc}}^{(\calE,\alpha), (\calF,\beta) }([0,\infty)_\tau\times (0,\infty]_s\times \partial X_\theta )$, where $\calE$ is the index set as $s\to\infty$ and $\calF$ is the index set as $\tau\to 0$.
	Moreover, if $\chi \in C_{\mathrm{c}}^\infty(\bbR)$ satisfies $\operatorname{supp} \chi \Subset (-1,1)$, then
	$\chi(2s\Sigma) I_{-,\mathrm{osc}}[\varphi](\tau,s,\theta) \in \calA_{\mathrm{loc}}^{\infty, \infty }([0,\infty)_\tau\times (0,\infty]_s\times \partial X_\theta ) $.
	\label{prop:I-}
\end{proposition}
\begin{proof}
	Let $\psi\in C_{\mathrm{c}}^\infty(\bbR)$ satisfy $\operatorname{supp} \psi \Subset (-1/2,1/2)$ and $0\notin \operatorname{supp}(1-\psi)$. Now define
	\begin{align}
		I_{-,\mathrm{osc}}[\varphi,\psi](\tau,s,\theta) &=\frac{2\tau^2}{s^2} \int_\Lambda^\infty e^{i\lambda^2 \tau - i \lambda} \psi(2\lambda \tau - 1) \varphi \Big( \frac{\lambda \tau}{s} , \lambda,\theta \Big)  \lambda \dd \lambda,  \label{eq:misc_198} \\
		I_{-,\mathrm{phg}}[\varphi,\psi](\tau,r,\theta) &=\frac{2}{r^2} \int_\Lambda^\infty e^{i\lambda^2 \tau - i \lambda}(1- \psi(2\lambda \tau -1))\varphi \Big( \frac{\lambda }{r} , \lambda,\theta \Big) \lambda \dd \lambda   . \label{eq:misc_197}
	\end{align}
	Then \cref{eq:misc_162} holds. We just need to check that each of these integrals lies in the expected function spaces.
	We begin with $I_{-,\mathrm{phg}}$. Let
	\begin{equation}
		\varphi(\sigma,\lambda,\theta) \sim \sum_{(j,k)\in \calE , \Re j \leq \alpha} \varphi_{j,k}(\lambda,\theta) \sigma^j \log^k \sigma
	\end{equation}
	denote the $\sigma\to 0^+$ expansion of $\varphi$, so $\varphi_{j,k}(\lambda,\theta) \in \calA_{\mathrm{c}}^{(\calF,\beta)}((\Lambda,\infty]_\lambda\times \partial X_\theta)$. Consider the function
	\begin{equation}
	 \varphi_\gamma(\sigma,\lambda,\theta)  =
	\sigma^{-\gamma} \Big[ \varphi(\sigma,\lambda,\theta) - \sum_{(j,k)\in \calE , \Re j \leq \gamma} \varphi_{j,k}(\lambda,\theta) \sigma^j \log^k \sigma \Big] \in \calA_{\mathrm{c}}^{0,(\calF,\beta)}([0,\Sigma)_\sigma \times (\Lambda,\infty]_\lambda\times \partial X_\theta).
	\end{equation}
	Then, we have
	\begin{equation}
	\frac{r^2}{2} I_{-,\mathrm{phg}}[\varphi,\psi]= \sum_{(j,k) \in \calE, \Re j \leq \gamma} \sum_{\kappa=0}^k \binom{k}{\kappa} \Big( \frac{1}{r}\Big)^j  \log^{k-\kappa} \Big( \frac{1}{r} \Big)I_{-,\mathrm{phg},1+j,\kappa}[\varphi_{j,k},\psi] + \Big( \frac{1}{r}\Big)^\gamma I_{-,\mathrm{phg},1+\gamma,0}[\varphi_\gamma,\psi],
	\label{eq:misc_167}
	\end{equation}
	where the quantity $I_{-,\mathrm{phg},\gamma,\kappa}$ is defined by \Cref{lem:I-phg}. That lemma then gives that each term on the right-hand side of \cref{eq:misc_167}, and therefore $
	I_{-,\mathrm{phg}}[\varphi,\psi]$ itself, has the required form, except with a conormal error of order $\gamma$. But since $\gamma\leq \alpha$ was arbitrary, \cref{eq:misc_163} follows.

	Moving on to the other oscillatory integral, we introduce the coordinate $\delta = \lambda \tau-1/2$. In terms of this coordinate,
	\begin{equation}
		s^2 I_{-,\mathrm{osc}}[\varphi,\psi](\tau,s,\theta) = e^{-i/4\tau} \tilde{I}_{-,\mathrm{osc}}[\varphi,(\delta+1)\psi(\delta)](\tau,s,\theta)
	\end{equation}
	and
	\begin{equation}
		\tilde{I}_{-,\mathrm{osc}}[\varphi,\psi](\tau,s,\theta) =  \int_{-\infty}^\infty e^{i\delta^2/\tau} \psi(2\delta) \varphi \Big( \frac{1}{s}\Big(\delta+\frac{1}{2} \Big) , \frac{1}{\tau} \Big(\delta+\frac{1}{2} \Big) ,\theta \Big)  \dd \delta.
	\end{equation}
	We can split this, for each $\gamma\in \bbR$,  as
	\begin{equation}
	\tilde{I}_{-,\mathrm{osc}}[\varphi,\psi]= \sum_{(j,k) \in \calE, \Re j \leq \gamma} \sum_{\kappa=0}^k \binom{k}{\kappa} (-1)^{k-\kappa} s^{-j}  \log^{k-\kappa} (s)\tilde{I}_{-,\mathrm{osc}}[\varphi_{j,k},\psi_{j,\kappa}]  + s^{-\gamma} \tilde{I}_{-,\mathrm{osc}}[\varphi_\gamma,\psi_\gamma],
	\end{equation}
	where $\psi_{j,\kappa}(2\delta) = (\delta+1/2)^j \log^\kappa(\delta+1/2)\psi(2\delta)$ and $\psi_\gamma(2\delta) = (\delta+1/2)^\gamma \psi(2\delta)$. Now let
	\begin{equation}
	\varphi_{j,k}(\lambda,\theta) \sim \sum_{(j',k') \in \calF, \Re j' \leq \beta} \lambda^{-j'} \log^{k'}(\lambda) \varphi_{j,k}^{j',k'}(\theta)
	\end{equation}
	denote the expansion of $\varphi_{j,k}(\lambda,\theta)$ as $\lambda\to\infty$, and let $\lambda^{-\gamma} \varphi_{j,k}^\gamma(\lambda,\theta)$ denote the error from truncating the expansion to $\Re j' \leq \gamma$. We use similar notation for $\varphi_\gamma$. Then, for each $\gamma,\gamma'\in \bbR$,
	\begin{multline}
	\tilde{I}_{-,\mathrm{osc}}[\varphi_{j,k},\psi_{j,\kappa}] =  \sum_{(j',k') \in \calF, \Re j' \leq \gamma'} \sum_{\varkappa=0}^{k'} \binom{k'}{\varkappa} (-1)^{k'-\varkappa}\tau^{j'} \log^{k'-\varkappa}(\tau) \tilde{I}_{-,\mathrm{osc}}[ \varphi_{j,k}^{j',k'}, \psi_{j,\kappa}^{j',\varkappa} ] \\ + \tau^{\gamma'} \tilde{I}_{-,\mathrm{osc}}[\varphi_{j,k}^{\gamma'}, \psi_{j,\kappa}^{\gamma'} ],
	\end{multline}
	\begin{multline}
	\tilde{I}_{-,\mathrm{osc}}[\varphi_{\gamma},\psi_{\gamma}] =  \sum_{(j',k') \in \calF, \Re j' \leq \gamma'} \sum_{\kappa=0}^{k'} \binom{k'}{\kappa} (-1)^{k'-\kappa}\tau^{j'} \log^{k'-\kappa}(\tau) \tilde{I}_{-,\mathrm{osc}}[ \varphi_{\gamma}^{j',k'}, \psi_{\gamma}^{j',\kappa} ] \\ + \tau^{\gamma'} \tilde{I}_{-,\mathrm{osc}}[\varphi_{\gamma}^{\gamma'}, \psi_{\gamma}^{\gamma'} ],
	\end{multline}
	where $\psi_{j,\kappa}^{j',\varkappa}(2\delta) = (\delta+1/2)^{-j'} \log^\varkappa(\delta+1/2) \psi_{j,\kappa}(2\delta)$, and similarly for the other undefined terms.

	The result then follows from \Cref{prop:Jlem}. Indeed, we have that $\tilde{I}_{-,\mathrm{osc}}$ is a sum of four types of terms:
	\begin{itemize}
		\item First, consider the ``main'' terms proportional to $s^{-j} \log^k(s) \tau^{j'} \log^{k'}(\tau) \tilde{I}_{-,\mathrm{osc}}[\varphi_{j,k}^{j',k'}, \psi_{j,\kappa}^{j',\varkappa} ]$.

		Noting that $\tilde{I}_{-,\mathrm{osc}}[\varphi_{j,k}^{j',k'}, \psi_{j,\kappa}^{j',\varkappa} ]$ does not depend on $s$, the last clause of \Cref{prop:Jlem}
		yields
		\begin{equation}
		\tilde{I}_{-,\mathrm{osc}}[\varphi_{j,k}^{j',k'}, \psi_{j,\kappa}^{j',\varkappa} ] \in \tau^{1/2} C^\infty([0,\infty)_\tau; C^\infty(\partial X_\theta) ).
		\end{equation}
		\item Now consider the terms $s^{-j} \log^k(s) \tau^\gamma \tilde{I}_{-,\mathrm{osc}}[\varphi_{j,k}^{\gamma}, \psi^\gamma_{j,\kappa} ]$. Again noting that $\tilde{I}_{-,\mathrm{osc}}[\varphi_{j,k}^{\gamma}, \psi^\gamma_{j,\kappa} ]$ does not depend on $s$, we have
		\begin{equation}
		\tilde{I}_{-,\mathrm{osc}}[\varphi_{j,k}^{\gamma}, \psi^\gamma_{j,\kappa} ] \in \tau^{1/2} \calA^0_{\mathrm{loc}}([0,\infty)_\tau \times \partial X_\theta).
		\end{equation}
		\item Now consider the terms $s^{-\gamma} \tau^j \log^k(\tau) \tilde{I}_{-,\mathrm{osc}}[\varphi_\gamma^{j,k}, \psi_\gamma^{j,k}]$.

		By the last clause of \Cref{prop:Jlem}, $\tilde{I}_{-,\mathrm{osc}}[\varphi_\gamma^{j,k}, \psi_\gamma^{j,k}] \in \tau^{1/2} \calA_{\mathrm{loc}}^{0,(0,0)}([0,\infty)_\tau \times (0,\infty]_s \times \partial X_\theta)$.
		\item Finally, in $\tau^{\gamma'} s^{-\gamma} \tilde{I}_{-,\mathrm{osc}}[\varphi_\gamma^{\gamma'},\psi_\gamma^{\gamma'}]$, we have $\tilde{I}_{-,\mathrm{osc}}[\varphi_\gamma^{\gamma'},\psi_\gamma^{\gamma'}] \in \tau^{1/2} \calA_{\mathrm{loc}}^{0,0}([0,\infty)_\tau \times (0,\infty]_s \times \partial X_\theta)$
		by \Cref{prop:Jlem}.
	\end{itemize}

Putting this all together, $\tilde{I}_{-,\mathrm{osc}}[\varphi](\tau, s,\theta  ) \in  \tau^{1/2}\calA_{\mathrm{loc}}^{(\calE,\gamma), (\calF,\gamma') }([0,\infty)_\tau\times (0,\infty]_s\times \partial X_\theta )$. Taking $\gamma\to \alpha$ and $\gamma'\to \beta$ completes the proof that $\tilde{I}_{-,\mathrm{osc}}$ lies in the desired function spaces.

If $\chi \in C_{\mathrm{c}}^\infty(\bbR)$ satisfies $\operatorname{supp} \chi \Subset (-1,1)$, then
$\chi(2s\Sigma) I_{-,\mathrm{osc}}[\varphi](\tau,s,\theta) \in \calA_{\mathrm{loc}}^{\infty, \infty }([0,\infty)_\tau\times (0,\infty]_s\times \partial X_\theta ) $, as the corresponding clause of \Cref{prop:Jlem} shows that each of the terms $\tilde{I}_{-,\mathrm{osc}}$ appearing above is Schwartz when multiplied by $\chi(2s\Sigma)$.
\end{proof}

\begin{remark}
	The proof shows that the $r\to \infty$ expansion of $I_{-,\mathrm{phg}}[\varphi,\psi]$ is given by
	\begin{multline}
	 	I_{-,\mathrm{phg}}[\varphi,\psi]\sim  \frac{2}{r^2} \sum_{(j,k)\in \calE, \Re j \leq \alpha}  \Big(\frac{1}{r}\Big)^{j}  \log^k\Big(\frac{1}{r}\Big) \sum_{K\geq k, (j,K) \in \calE} \binom{K}{k}  \int_\Lambda^\infty
	e^{i \lambda^2 \tau - i\lambda} \\ \varphi_{j,K}(\lambda,\theta) \lambda^{1+j} \log^{K-k} (\lambda) \dd \lambda,
	\end{multline}
	analogously to \cref{eq:misc_tiki}.
	The $s\to\infty$ expansion of $I_{-,\mathrm{osc}}[\varphi,\psi]$ is given by
	\begin{equation}
	I_{-,\mathrm{osc}}[\varphi,\psi] \sim \frac{e^{-i/4\tau}}{s^2} \sum_{(j,k)\in \calE, \Re j\leq \alpha} s^{-j} \log^k(s) (-1)^k \sum_{K\geq k, (j,K)\in \calE} \binom{K}{k} \tilde{I}_{-,\mathrm{osc}}[\varphi_{j,K},(\delta+1) \psi_{j,K-k}(\delta) ]  ,
	\end{equation}
	where
	\begin{equation}
	\tilde{I}_{-,\mathrm{osc}}[\varphi_{j,K}, \psi_{j,K-k} ]    = \int_{-\infty}^{\infty} e^{i \delta^2/\tau} (\delta+1/2)^{1+j} \log^{K-k}(\delta+1/2) \psi(2\delta) \varphi_{j,K} \Big(\frac{1}{\tau} \Big(\delta+\frac{1}{2}\Big),\theta \Big) \dd \delta .
	\end{equation}
	We do not write the $\tau\to 0$ expansion.
\end{remark}

\begin{lemma}
	Let $\psi\in C_{\mathrm{c}}^\infty(\bbR)$ satisfy $\operatorname{supp} \psi \Subset (-1/2,1/2)$ and $0\notin \operatorname{supp}(1-\psi)$. Fix $\Sigma,\Lambda>0$.
	Let $\phi \in \calA^{0,0,0,0}_{\mathrm{c}}([0,\Sigma)_\sigma\times (\Lambda,\infty]_\lambda \times (0,\infty]_r \times [0,\infty)_\tau \times \partial X_\theta )$, and consider, for $\gamma\in \bbC$ and $\kappa\in \bbN$,
	\begin{equation}
	I_{-,\mathrm{phg},\gamma,\kappa}[\phi,\psi](\tau,r,\theta) = \int_\Lambda^\infty e^{i\lambda^2 \tau - i \lambda}(1- \psi(2\lambda \tau -1))\phi \Big( \frac{\lambda }{r} , \lambda,r, \tau,\theta \Big)  \lambda^\gamma \log^\kappa(\lambda)\dd \lambda.
	\end{equation}
	Then, we have $I_{-,\mathrm{phg},\gamma,\kappa}[\phi,\psi](\tau, r,\theta  ) \in \calA_{\mathrm{loc}}^{0,0}([0,\infty)_\tau\times (0,\infty]_r\times \partial X_\theta ) $.
	If $\phi(\sigma,\lambda,r,\tau,\theta) = \phi(\lambda,\theta)$ depends only on $\lambda,\theta$, then this can be improved to
	\begin{equation}
		I_{-,\mathrm{phg},\gamma,\kappa}[\phi,\psi](\tau, r,\theta  ) \in \calA_{\mathrm{loc}}^{(0,0),(0,0)}([0,\infty)_\tau\times (0,\infty]_r\times \partial X_\theta ) =C^\infty(\dot{C}_1\backslash \mathrm{kf}),
	\end{equation}
	and $I_{-,\mathrm{phg},\gamma,\kappa}[\phi,\psi](\tau, r,\theta  )$ does not depend on $r$ in this case.
	\label{lem:I-phg}
\end{lemma}
\begin{proof}
	We first prove that $I_{-,\mathrm{phg},\gamma,\kappa}[\phi,\psi](\tau, r,\theta  ) \in L_{\mathrm{loc}}^{\infty}([0,\infty)_\tau\times (0,\infty]_r\times \partial X_\theta ) $. If $\Re \gamma< -1$, this follows immediately from an $L^\infty$-bound. Otherwise, we use an integration-by-parts argument as usual: $\exp(i \lambda^2 \tau - i \lambda ) = -i (2\lambda \tau -1)^{-1} \partial_\lambda \exp(i \lambda^2 \tau - i \lambda )$,
	so
	\begin{equation}
	I_{-,\mathrm{phg},\gamma,\kappa}[\phi,\psi] = \int_\Lambda^\infty \Big[\frac{-i}{2\lambda \tau-1}
	\frac{\partial}{\partial \lambda} e^{i \lambda^2 \tau - i \lambda } \Big] (1- \psi(2\lambda \tau -1))\phi \Big( \frac{\lambda}{r} , \lambda,r, \tau,\theta \Big)  \lambda^\gamma \log(\lambda)^\kappa\dd \lambda.
	\end{equation}
	Note that the integrand is well-defined, since the factor $1-\psi(2\lambda\tau-1)$ vanishes when $2\lambda \tau-1$ is sufficiently small. Integrating-by-parts, noting that no boundary terms arise,
	\begin{equation}
	I_{-,\mathrm{phg},\gamma,\kappa}[\phi,\psi] = I_{-,\mathrm{phg},\gamma-1,\kappa}[\phi_0,\psi_0] - \gamma I_{-,\mathrm{phg},\gamma-1,\kappa}[\phi_1,\psi_0] - \kappa I_{-,\mathrm{phg},\gamma-1,\kappa-1}[\phi_1,\psi_0],
	\label{eq:misc_201}
	\end{equation}
	where $\psi_0 \in C_{\mathrm{c}}^\infty((-1/2,+1/2))$ is identically $1$ near the origin and satisfies  $\operatorname{supp} \psi_0 \Subset \psi^{-1}(\{1\})$, and where
	\begin{multline}
	\phi_0(\sigma,\lambda,r,\tau,\theta) =  (1- \psi(2\lambda \tau -1)) \Big[ -\frac{2i  \lambda \tau}{(2\lambda \tau -1)^2}   + \frac{i}{2\lambda \tau-1} ( \sigma \partial_\sigma   + \lambda \partial_\lambda )\Big] \phi(\sigma,\lambda,r,\tau,\theta) \\
	- \frac{2 i  \lambda \tau}{2\lambda \tau - 1}  \psi'(2\lambda \tau - 1) \phi(\sigma,\lambda,r,\tau,\theta),
	\end{multline}
	\begin{equation}
	\phi_1(\sigma,\lambda,r,\tau,\theta) =  (1- \psi(2\lambda \tau -1)) \frac{-i\phi(\sigma,\lambda,r,\tau,\theta)}{2\lambda\tau-1}.
	\end{equation}
	Since the three functions $(1- \psi(2\lambda \tau -1)) / (2\lambda \tau - 1) , (1- \psi(2\lambda \tau -1))  \lambda \tau / (2\lambda \tau - 1)^2, \psi'(2\lambda \tau-1)\lambda \tau / (2\lambda \tau - 1)$ all lie in $ \calA^{0,0,0,0}_{\mathrm{loc}}([0,\Sigma)_\sigma   \times (\Lambda,\infty]_\lambda \times (0,\infty]_r \times [0,\infty)_\tau\times \partial X_\theta )$, we have
	\begin{equation}
	\phi_0(\sigma,\lambda,r,\tau,\theta),
	\phi_1(\sigma,\lambda,r,\tau,\theta) \in  \calA^{0,0,0,0}_{\mathrm{c}}([0,\Sigma)_\sigma \times (\Lambda,\infty]_\lambda \times (0,\infty]_r \times [0,\infty)_\tau \times \partial X_\theta ).
	\end{equation}
	Thus, each term on the right-hand side of \cref{eq:misc_201} has the same form as the original oscillatory integral, except with $\gamma-1$ in place of $\gamma$ and possibly $\kappa-1$ in place of $\kappa$, if $\kappa>0$.
	So, \cref{eq:misc_201} can be used inductively to conclude the desired $L^\infty$-bound from the $\Re \gamma<-1$ case.

	In order to prove conormality, we want to prove that
	\begin{equation}
	(\tau \partial_\tau)^j(r\partial_r)^k L I_{-,\mathrm{phg},\gamma,\kappa}[\phi,\psi](\tau, r,\theta  )  \in L_{\mathrm{loc}}^{\infty}([0,\infty)_\tau\times (0,\infty]_r\times \partial X_\theta )
	\label{eq:misc_coco}
	\end{equation}
	for all $j,k\in \bbN$ and $L\in \operatorname{Diff}(\partial X)$. The angular derivatives $L$ are handled via differentiation under the integral sign as elsewhere, and for the other directions we use the following identities:
		\begin{equation}
		\tau \partial_\tau I_{-,\mathrm{phg},\gamma,\kappa}[\phi,\psi] =  i \tau I_{-,\mathrm{phg},\gamma+2,\kappa}[\phi,\psi]  + I_{-,\mathrm{phg},\gamma,\kappa}[\varphi,\psi]
		- I_{-,\mathrm{phg},\gamma,\kappa}[\varpi ,\psi_0]
		\label{eq:misc_164}
		\end{equation}
		for $\varphi(\sigma,\lambda,r,\tau,\theta) = \tau \partial_\tau \phi(\sigma,\lambda,r,\tau,\theta)$ and $\varpi(\sigma,\lambda,r,\tau,\theta)=2\lambda\tau \psi'(2\lambda\tau-1)\phi(\sigma,\lambda,r,\tau,\theta)$, and
		\begin{equation}
		r\partial_r I_{-,\mathrm{phg},\gamma,\kappa}[\phi,\psi](\tau,r,\theta) = I_{-,\mathrm{phg},\gamma,\kappa}[\varsigma,\psi]
		\end{equation}
		for $\varsigma(\sigma,\lambda,r,\tau,\theta)=(- \sigma\partial_\sigma + r\partial_r) \phi(\sigma,\lambda,r,\tau,\theta)$.
	Each of these has the same form as the original integral, so, a similar inductive argument to the above shows that the bounds \cref{eq:misc_coco} follow from those already proven.

	Finally, suppose that  $\phi(\sigma,\lambda,r,\tau,\theta) = \phi(\lambda,\theta)$ depends only on $\lambda,\theta$. Then, \cref{eq:misc_164} can be improved to
	\begin{align}
		\partial_\tau I_{-,\mathrm{phg},\gamma,\kappa}[\phi,\psi](\tau,r,\theta) =  i I_{-,\mathrm{phg},\gamma+2,\kappa}[\phi,\psi](\tau,r,\theta) -  I_{-,\mathrm{phg},\gamma+1,\kappa}[\varPi ,\psi_0](\tau,r,\theta)
		\label{eq:misc_166}
	\end{align}
	for $\varPi =  (\lambda\tau)^{-1}\varpi $, and now we simply have $\partial_r I_{-,\mathrm{phg},\gamma,\kappa}[\phi,\psi] = 0$. Noting that each term on the right-hand side of \cref{eq:misc_166} has the same form as the original integral, the usual inductive argument yields polyhomogeneity.
\end{proof}

\begin{proposition}
	Let $\psi\in C_{\mathrm{c}}^\infty(\bbR)$ satisfy $\operatorname{supp} \psi \Subset (-1,1)$. Fix $\Sigma,\Lambda>0$, $k\in \bbN$.
	Let $\phi \in \calA^{0,0,0,0}_{\mathrm{c}}([0,\Sigma)_\sigma\times (\Lambda,\infty]_\lambda \times (0,\infty]_s \times [0,\infty)_\tau \times \partial X_\theta )$, and consider
	\begin{equation}
	\calJ_{k}[\phi,\psi](\tau,s,\theta) = \int_{-\infty}^{+\infty} e^{i \delta^2/\tau} \delta^{k}  \psi(2\delta) \phi \Big( \frac{1}{s}\Big( \delta+\frac{1}{2}\Big)  , \frac{1}{\tau}\Big( \delta+\frac{1}{2}\Big),s, \tau,\theta \Big)   \dd \delta.
	\label{eq:misc_176}
	\end{equation}
	Then,
	$\calJ_k[\phi,\psi](\tau,s,\theta) \in \tau^{(k+1)/2} \calA_{\mathrm{loc}}^{0,0}([0,\infty)_\tau \times (0,\infty]_s \times \partial X_\theta)$. If $\chi \in C_{\mathrm{c}}^\infty(\bbR)$ satisfies $\operatorname{supp} \chi \Subset (-1,1)$, then $\chi(2s \Sigma) \calJ_k[\phi,\psi]$ is Schwartz.

	If $\phi(\sigma,\lambda,s,\tau,\theta) = \phi(\sigma,s,\theta)$ does not depend on $\lambda,\tau$, then 
	\begin{equation*} 
		\calJ_k[\phi,\psi](\tau,s,\theta) \in \tau^{(k+1)/2} \calA_{\mathrm{loc}}^{0,(0,0)}([0,\infty)_\tau\times(0,\infty]_s \times \partial X_\theta),
	\end{equation*} 
	where the $(0,0)$ is the index set as $\tau\to 0$.
	\label{prop:Jlem}
\end{proposition}
\begin{proof}
	\begin{enumerate}[label=(\Roman*)]
		\item We begin by proving the weaker claim that
		\begin{equation}
		\calJ_{k}[\phi,\psi] \in \tau^{\lfloor k/2 \rfloor } L_{\mathrm{loc}}^\infty([0,\infty)_\tau\times (0,\infty]_s \times \partial X_\theta).
		\label{eq:misc_ann_0}
		\end{equation}
		It suffices to just prove 
		\begin{equation}
		\calJ_{2k}[\phi,\psi] \in \tau^{k} L_{\mathrm{loc}}^\infty([0,\infty)_\tau\times (0,\infty]_s \times \partial X_\theta), 
		\label{eq:misc_ann}
		\end{equation}
		since the case of \cref{eq:misc_ann_0} with $k$ odd follows from the case with $k$ even (applied with $\delta/2 \psi(\delta)$ in place of $\psi(\delta)$).
		As with the other integrals analyzed elsewhere in this paper, this is proven using integration-by-parts: if $k=0$, then this bound is immediate, and otherwise, if $k\geq 1$, use
		\begin{equation}
		\calJ_{2k}[\phi,\psi] = -\frac{i\tau}{2} \int_{-\infty}^{+\infty}\Big[ \frac{\partial}{\partial \delta}  e^{i \delta^2/\tau} \Big] \delta^{2k-1}  \psi(2\delta) \phi \Big( \frac{1}{s}\Big( \delta+\frac{1}{2}\Big)  , \frac{1}{\tau}\Big( \delta+\frac{1}{2}\Big),s, \tau,\theta \Big)  \dd \delta.
		\end{equation}
		So,
		\begin{multline}
		-\frac{2i}{\tau}\calJ_{2k}[\phi,\psi] = (2k-1)\calJ_{2k-2}[\phi,\psi] +  \calJ_{2k-2}[\phi,\Delta\psi'(\Delta)]  \\ +
		\calJ_{2k-2}[ \sigma\partial_\sigma \phi, (\Delta+1)^{-1}\Delta \psi(\Delta)]  +
		\calJ_{2k-2}[ \lambda\partial_\lambda \phi, (\Delta+1)^{-1}\Delta \psi(\Delta)].
		\label{eq:misc_206}
		\end{multline}
		Using this identity inductively, $\calJ_{2k}[\phi,\psi] \in \tau^k L_{\mathrm{loc}}^\infty([0,\infty)_\tau\times (0,\infty]_s \times \partial X_\theta)$ follows from the $k=0$ case.
		\item

		The next goal is to improve this to
		\begin{equation}
		\calJ_{k}[\phi,\psi] \in \tau^{(k+1)/2} L_{\mathrm{loc}}^\infty([0,\infty)_\tau\times (0,\infty]_s \times \partial X_\theta).
		\label{eq:misc_179}
		\end{equation}
		In order to improve upon the previous bounds, we expand $\phi$ in Taylor series around $\delta=0$:
		\begin{multline}
		\phi \Big( \frac{1}{s}\Big( \delta+\frac{1}{2}\Big)  , \frac{1}{\tau}\Big( \delta+\frac{1}{2}\Big),s, \tau,\theta \Big) = \sum_{j_1+j_2 \leq J} \frac{\delta^{j_1+j_2}}{j_1! j_2! s^{j_1} \tau^{j_2}} \phi^{(j_1,j_2)} \Big( \frac{1}{2s}  , \frac{1}{2\tau},s, \tau,\theta \Big) \\
		+  \int_0^\delta (\delta-\Delta)^J \sum_{j_1+j_2=J+1} \frac{1}{j_1! j_2! s^{j_1} \tau^{j_2}} \phi^{(j_1,j_2)} \Big( \frac{1}{s}\Big( \Delta+\frac{1}{2}\Big)  , \frac{1}{\tau}\Big( \Delta+\frac{1}{2}\Big),s, \tau,\theta \Big)\dd \Delta,
		\end{multline}
		where $\phi^{(j_1,j_2)}(\sigma,\lambda,s,\tau,\theta) = \partial_\sigma^{j_1} \partial_\lambda^{j_2}\phi (\sigma,\lambda,s,\tau,\theta)$.
		So, for any $J\in \bbN$,
		\begin{equation}
		\calJ_k[\phi,\psi] =\calJ_{k,J+1}[\phi,\psi]+ \sum_{j_1+j_2 \leq J} \calJ_{k,j_1,j_2}[\phi,\psi],
		\label{eq:misc_209}
		\end{equation}
		where
		\begin{equation}
		\calJ_{k,j_1,j_2}[\phi,\psi] = \frac{1}{j_1!j_2! s^{j_1}\tau^{j_2}}  \phi^{(j_1,j_2)} \Big( \frac{1}{2s}  , \frac{1}{2\tau},s, \tau,\theta \Big) \int_{-\infty}^{+\infty} e^{i \delta^2/\tau} \delta^{k+j_1+j_2}  \psi(2\delta)   \dd \delta
		\end{equation}
		and
		\begin{multline}
		\calJ_{k,J+1}[\phi,\psi] =  \int_{-\infty}^{+\infty} e^{i \delta^2/\tau} \delta^{k}  \psi(2\delta) \Big[\int_0^\delta (\delta-\Delta)^J \\ \times \sum_{j_1+j_2=J+1} \frac{1}{j_1! j_2! s^{j_1} \tau^{j_2}} \phi^{(j_1,j_2)} \Big( \frac{1}{s}\Big( \Delta+\frac{1}{2}\Big)  , \frac{1}{\tau}\Big( \Delta+\frac{1}{2}\Big),s, \tau,\theta \Big)\dd \Delta  \Big] \dd \delta.
		\label{eq:misc_210}
		\end{multline}

		The method of stationary phase (which in this case amounts to Parseval--Plancherel), applied to the integral
		\begin{equation}
		\hat{\calJ}_j[\psi] = \int_{-\infty}^{+\infty} e^{i \delta^2/\tau} \delta^{j}  \psi(2\delta)   \dd \delta,
		\end{equation}
		yields that $\hat{\calJ}_j[\psi] \in \tau^{(j+1)/2} C^\infty[0,\infty)_\tau $. So,
		\begin{equation}
		\calJ_{k,j_1,j_2}[\phi,\psi]  \in \tau^{(1+k+j_1+j_2)/2} \calA^{0,0}_{\mathrm{loc}}([0,\infty)_\tau\times (0,\infty]_s \times \partial X_\theta).
		\label{eq:misc_185}
		\end{equation}

		In order to estimate $\calJ_{k,J+1}[\phi,\psi]$, we use a similar integration-by-parts argument as before. We prove, via induction on $k+J$, that
		\begin{equation}
		\calJ_{k,J+1}[\phi,\psi] \in \tau^{\lfloor (k+J+1)/2 \rfloor}  L_{\mathrm{loc}}^\infty([0,\infty)_\tau\times (0,\infty]_s \times \partial X_\theta),
		\label{eq:misc_211}
		\end{equation}
		which is trivial in the $k+J=0$ case.
		Integrating-by-parts, we can write
		\begin{align}
			- 2i \tau^{-1}\calJ_{k,J+1}[\phi,\psi] &= (k-1)\calJ_{k-2,J+1}[\phi,\psi] + 2\calJ_{k-1,J+1}[\phi,\psi'] +J\calJ_{k-1,J}[\phi,\psi]
			\label{eq:3534143}
			\intertext{if $J\geq 1$ and, otherwise, }
			- 2i \tau^{-1}\calJ_{k,1}[\phi,\psi] &= (k-1)\calJ_{k-2,1}[\phi,\psi] + 2\calJ_{k-1,1}[\phi,\psi'] + \calJ_{k-1 }[\varphi,\tilde{\psi}]
			\label{eq:3534144}
		\end{align}
		for $\tilde{\psi}(2\delta) = (\delta+1/2)^{-J-1} \psi(2\delta)$ and
		\begin{equation}
		\varphi(\sigma, \lambda,s,\tau,\theta) = \sum_{j_1+j_2 =  1} \frac{\sigma^{j_1} \lambda^{j_2}}{j_1!j_2!} \frac{\partial^{j_1}}{\partial \sigma^{j_1}}\frac{\partial^{j_2}}{\partial \lambda^{j_2}} \phi (\sigma,\lambda,s,\tau,\theta),
		\end{equation}
		where $\calJ_{k-1 }[\varphi,\psi]$ is defined by \cref{eq:misc_176}. Some of the terms above may have the form $\calJ_{k,\bullet}$ for $k$ \emph{negative}, but the integral defining $\calJ_{k,J}$ makes sense for $k\in \bbZ$ (and $J\in \bbN$) as long as $k+J\geq -1$ (similar to the proof of \Cref{lem:adam}). The bound \cref{eq:misc_211} is as trivial when $k+J=-1$ as $k+J=0$. 

		Consider each of the terms on the right-hand sides of \cref{eq:3534143}, \cref{eq:3534144}.
		By \cref{eq:misc_ann}, we have 
		\begin{equation*} 
			\calJ_{k-1 }[\varphi,\tilde{\psi}] \in \tau^{\lfloor (k-1)/2 \rfloor}L_{\mathrm{loc}}^\infty([0,\infty)_\tau\times (0,\infty]_s \times \partial X_\theta).
		\end{equation*}
		So, if we know that
		\begin{equation}\;
		\calJ_{k-1,J+1}[\phi,\psi'] ,
		\calJ_{k-2,J+1}[\phi,\psi] , \calJ_{k-1,J}[\phi,\psi] \in \tau^{\lfloor (k+J-1)/2 \rfloor}L_{\mathrm{loc}}^\infty([0,\infty)_\tau\times (0,\infty]_s \times \partial X_\theta),
		\label{eq:jj3123}
		\end{equation}
		then we can conclude that \cref{eq:misc_211} holds. This sets up an inductive algorithm (analogous to that in the proof of \Cref{lem:adam}) to deduce the claim from the $k+J=0,-1$ base case which is already known. Indeed, the induction hypothesis gives \cref{eq:jj3123} on the nose for $\calJ_{k-2,J+1}[\phi,\psi] , \calJ_{k-1,J}[\phi,\psi]$, and 
		\begin{equation*}
		\calJ_{k-1,J+1}[\phi,\psi'] \in \tau^{\lfloor (k+J)/2 \rfloor}L_{\mathrm{loc}}^\infty([0,\infty)_\tau\times (0,\infty]_s \times \partial X_\theta) 
		\end{equation*}
		for the remaining term. This is either the required estimate or
		$\tau^{1/2}$ \emph{better} (recall that $\tau\to 0$). So, the induction goes through. 

		Returning now to \cref{eq:misc_179}, this follows from \cref{eq:misc_209} combined with \cref{eq:misc_185} and \cref{eq:misc_211}, as long as $J$ is sufficiently large.
		\item
		Having now proven the optimal $L^\infty$-bounds \cref{eq:misc_211}, we upgrade this to conormality by estimating derivatives. We can do this using the usual argument: $L\calJ_{k,J+1}[\phi,\psi] = \calJ_{k,J+1}[L\phi,\psi]$ for $L\in \operatorname{Diff}(\partial X_\theta)$, $s \partial_s \calJ_{k,J+1}[\phi,\psi](\tau,s,\theta) =  \calJ_{k,J+1}[ (-\sigma \partial_\sigma + s \partial_s) \phi (\sigma,\lambda,s,\tau,\theta),\psi]$,
		and
		\begin{multline}
		\qquad\qquad\tau \partial_\tau \calJ_{k,J+1}[\phi,\psi](\tau,s,\theta) = - \tau^{-1}i \calJ_{k+2,J+1}[\phi,\psi](\tau,s,\theta) \\ +   \calJ_{k,J+1}[ (-\lambda \partial_\lambda + \tau \partial_\tau) \phi (\sigma,\lambda,s,\tau,\theta),\psi].
		\label{eq:misc_206c}
		\end{multline}
		So, via the usual inductive argument,  $(s\partial_s)^{m} (\tau \partial_\tau)^n L\calJ_{k,J+1}[\phi,\psi] \in \tau^{\lfloor (k+J+1)/2 \rfloor}  L_{\mathrm{loc}}^\infty([0,\infty)_\tau\times (0,\infty]_s \times \partial X_\theta)$ for all $m,n\in \bbN$, which means that
		\begin{equation}
		\calJ_{k,J+1}[\phi,\psi] \in \tau^{\lfloor (k+J+1)/2 \rfloor}  \calA_{\mathrm{loc}}^{0,0}([0,\infty)_\tau\times (0,\infty]_s \times \partial X_\theta).
		\end{equation}
		\item

		Consider now $\chi(2s\Sigma) \calJ_k[\phi,\psi]$, where $\chi$ is as in the statement of the proposition. Then, $\chi(2s\Sigma) \calJ_{k,j_1,j_2}[\phi,\psi]=0$
		identically, for all $j_1,j_2\in \bbN$ (due to the assumptions on the support of $\phi$). So, the estimates above on $\calJ_{k,J+1}$, since $J$ can be made arbitrarily large, give Schwartz behavior for $\chi(2s\Sigma) \calJ_k[\phi,\psi]$ as $\tau\to\infty$.
		\item

		Suppose now that  $\phi(\sigma,\lambda,s,\tau,\theta) = \phi(\sigma,s,\theta)$ does not depend on $\lambda,\tau$. Then, $\calJ_k[\phi,\psi] = \sum_{j=0}^J \calJ_{k,j,0}[\phi,\psi] + \calJ_{k,J+1}[\phi,\psi]$, where
		\begin{align}
		\begin{split}
		\calJ_{k,j,0}[\phi,\psi] &= \frac{1}{j! s^j}  \phi^{(j)} \Big( \frac{1}{2s} ,s,\theta \Big) \int_{-\infty}^{+\infty} e^{i \delta^2/\tau} \delta^{k+j}  \psi(2\delta)   \dd \delta
		= \frac{1}{j! s^j}  \phi^{(j)} \Big( \frac{1}{2s} ,s,\theta \Big) \hat{\calJ}_{k+j}[\psi],
		\end{split} \\
		\calJ_{k,J+1}[\phi,\psi] &=  \int_{-\infty}^{+\infty} e^{i \delta^2/\tau} \delta^{k}  \psi(2\delta) \Big[\int_0^\delta   \frac{(\delta-\Delta)^J }{(J+1)!s^{J+1}} \phi^{(J+1)} \Big( \frac{1}{s}\Big( \Delta+\frac{1}{2}\Big)  , s,\theta \Big)\dd \Delta  \Big] \dd \delta,
		\end{align}
		where $\phi^{(j)} (\sigma,s,\theta) = \partial_\sigma^j \phi(\sigma,s,\theta)$. Instead of \cref{eq:misc_185}, we now have
		\begin{equation}
		\calJ_{k,j,0}[\phi,\psi] \in \tau^{(1+k+j)/2}\calA_{\mathrm{loc}}^{0,(0,0)}([0,\infty)_\tau\times(0,\infty]_s \times \partial X_\theta),
		\end{equation}
		since $\hat{\calJ}_{j+k}[\psi](\tau) \in \tau^{(j+k+1)/2}C^\infty([0,\infty)_\tau)$. On the other hand, we can improve \cref{eq:misc_206c} to
		\begin{equation}
		\partial_\tau \calJ_{k,J+1}[\phi,\psi] =-i \tau^{-2}\calJ_{k+2,J+1}[\phi,\psi](\tau,s,\theta),
		\end{equation}
		i.e.\ 
		\begin{equation*}
			\tau\partial_\tau \calJ_{k,J+1}[\phi,\psi] =-i \tau^{-1}\calJ_{k+2,J+1}[\phi,\psi](\tau,s,\theta).
		\end{equation*}
		Note that the growth of the $\tau^{-1}$ factor is canceled out by the extra $\tau$ decay of $\calJ_{k+2,J+1}$ versus $\calJ_{k,J+1}$. So, the usual inductive argument yields
		\begin{equation}
		\calJ_{k,J+1}[\phi,\psi] \in \tau^{\lfloor (k+J+1)/2 \rfloor}  \calA_{\mathrm{loc}}^{0,(0,0)}([0,\infty)_\tau\times (0,\infty]_s \times \partial X_\theta).
		\end{equation}
		Combining the estimates above and taking $J$ to be sufficiently large, we conclude the final clause of the proposition.
	\end{enumerate}
\end{proof}

\section{Proof of main lemma}
\label{sec:main}

We now turn to the proof of \Cref{thm:D}, our ``main lemma.''
Let $\calE,\calF,\calG,\alpha,\beta,\gamma,\phi$ be as in the statement of that theorem.
Let $\chi_{\mathrm{low}}, \chi_{\mathrm{tf}\cap\mathrm{bf}}, \chi_{\mathrm{high}} \in C^\infty(X_{\mathrm{res}}^{\mathrm{sp}})$ be a partition of unity as in the introduction, so $\chi_{\mathrm{low}}+ \chi_{\mathrm{tf}\cap\mathrm{bf}}+\chi_{\mathrm{high}} = 1$ and
\begin{equation}
\operatorname{supp} \chi_{\mathrm{low}} \cap (\mathrm{bf}\cup \infty \mathrm{f}) = \varnothing,\; \operatorname{supp} \chi_{\mathrm{tf}\cap\mathrm{bf}} \cap (\mathrm{zf}\cup \infty \mathrm{f}) = \varnothing, \; \operatorname{supp} \chi_{\mathrm{high}} \cap (\mathrm{zf}\cup \mathrm{tf}) = \varnothing,
\end{equation}
and we can choose that  $\operatorname{supp} \chi_{\mathrm{low}} \cap ( \operatorname{supp} \chi_{\mathrm{tf} \cap \mathrm{bf}} \cup \operatorname{supp} \chi_{\mathrm{high}}) = \varnothing$. Moreover, $\chi_{\mathrm{tf}\cap\mathrm{bf}}$ can be chosen to be supported over $\dot{X}[R]$ for some $R$, that is over a collar neighborhood of the boundary of $X$.

We split $I_\pm[\phi]$ as in \cref{eq:misc_gab}.
We analyze each piece separately. First of all, according to \Cref{prop:low},
\begin{equation}
I_\pm[\chi_{\mathrm{low}} \phi ] \in \calA^{(\calE/2+1,\alpha/2+1), (\calF+2,\beta+2),(0,0)}(C_1).
\end{equation}
On the other hand, \Cref{prop:high_nonstat} says that $I_+[\chi_{\mathrm{high}}\phi ] \in \calA^{\infty,\infty,(0,0)}$, and, together, \Cref{prop:high_nonstat} and  \Cref{prop:high_stat} say that
\begin{multline}
I_-[\chi_{\mathrm{high}}\phi] \in \calA^{\infty,\infty,(0,0)}(C) + e^{-i (1-\chi(t))r^2/(4t) } \calA^{\infty,\infty,(\calG+1/2,\gamma+1/2),\infty,\infty }(M) \\ =  e^{-i (1-\chi(t))r^2/(4t) } \calA^{\infty,\infty,(\calG+1/2,\gamma+1/2),\infty,(0,0) }(M).
\end{multline}
Regarding $I_+[\chi_{\mathrm{tf}\cap\mathrm{bf}}\phi]$, says \Cref{prop:corner} that $I_+[\chi_{\mathrm{tf}\cap\mathrm{bf}}\phi] \in \calA^{\infty,(\calF+2,\beta+2),(0,0)}(C_1)$. Regarding $I_-[\chi_{\mathrm{tf}\cap\mathrm{bf}}\phi]$, the same proposition says that $I_-[\chi_{\mathrm{tf}\cap\mathrm{bf}} \phi] = \exp(-i r^2/(4t) )   \tilde{I}_-[\chi_{\mathrm{tf}\cap\mathrm{bf}} \phi] + I_{-,\mathrm{phg}}[\chi_{\mathrm{tf}\cap\mathrm{bf}} \phi]$
for some
\begin{equation}
\tilde{I}_-[\chi_{\mathrm{tf}\cap\mathrm{bf}} \phi] \in    \calA^{\infty,(\calF+2,\beta+2),(\calG+1/2,\gamma+1/2),\infty,\infty}(M)
\end{equation}
and $I_{-,\mathrm{phg}}[\chi_{\mathrm{tf}\cap\mathrm{bf}} \phi]\in \calA^{\infty,(\calF+2,\beta+2),(0,0)}(C_1)$.
So, summing up $I_\pm[\phi] = I_\pm[\chi_{\mathrm{low}} \phi] + I_\pm[\chi_{\mathrm{tf}\cap \mathrm{bf}} \phi] + I_\pm[\chi_{\mathrm{high}} \phi]$, we conclude \Cref{thm:D}.

\appendix

\section{Microlocal supplement}
\label{sec:microlocal}

The goal of this appendix is to sketch a microlocal proof of the following proposition:
\begin{proposition}
	Suppose that $f\in \calS(X)$, and let $\varphi_\pm(\sigma,x) = e^{\mp i\sigma r} R(\sigma^2 \pm i0) f(x)$. Then, letting $I_-[\varphi_-] = \exp(-i(1-\chi(t)) r^2/(4t)) I_{\mathrm{osc}}[\varphi]+ I_{-,\mathrm{phg}}[\varphi]$ denote the decomposition of $I_-[\varphi_-]$ provided in \Cref{thm:D}, it must be the case that the linear combination
	\begin{equation}
	I_{\mathrm{phg}} = I_+[\varphi_+] - I_{-,\mathrm{phg}} [\varphi_-]
	\end{equation} is Schwartz at $\mathrm{dilF}$.
	\label{prop:app}
\end{proposition}
\begin{proof}
First of all, note that because $I_{\mathrm{phg}}$ is already polyhomogeneous on $C_1$, it suffices to prove that $I_{\mathrm{phg}}$ is Schwartz in $\mathrm{dilF}^\circ$.
Let $M_0$ denote the manifold-with-boundary resulting from taking $[\bbR_t\times X ; \{\pm \infty\}\times \partial X ]$ subtracting the boundary hypersurfaces corresponding to $\{\pm \infty\}\times X$ and blowing down the boundary hypersurfaces corresponding to $\bbR_t\times \partial X$. Concretely, this mwc is identifiable with $\bbR_{t/r}\times [0,\infty)_\rho \times \partial X_\theta$ near the boundary.
We utilize Melrose's sc-calculus on $M_0$. See \cite{VasyGrenoble} for an introduction.

To show that $I_{\mathrm{phg}}$ is Schwartz at $\mathrm{dilF}^\circ$ means to show that $\operatorname{WF}_{\mathrm{sc}}(I_{\mathrm{phg}}) = \varnothing$, where $\operatorname{WF}_{\mathrm{sc}}$ is Melrose's notion of sc-wavefront set.
Let ${}^{\mathrm{sc}}o^* M_0$ denote the zero section of the sc-cotangent bundle over $\mathrm{dilF}^\circ$.
Because $I_{\mathrm{phg}}$ is conormal,
\begin{equation}
\operatorname{WF}_{\mathrm{sc}}(I_{\mathrm{phg}}) \subseteq  {}^{\mathrm{sc}} o^*  M_0,
\label{eq:misc_z}
\end{equation}
as follows e.g.\ via repeated applications of elliptic estimates in the sc-calculus. (This follows from b-vector fields being growing sc-vector fields.)

In order to study $\operatorname{WF}_{\mathrm{sc}}(I_{\mathrm{phg}})$ further, we use the relation of $I[\varphi] = I_+[\varphi_+]- I_-[\varphi_-]$ to $u(t,x) = (U(t) f)(x)$ given by \cref{eq:misc_019}. First note that
\begin{equation}
	\operatorname{WF}_{\mathrm{sc}}(e^{i E t} \varphi(x)) = \varnothing
\end{equation}
for any $E\in \bbR$ and $\varphi \in \calS(X)$. (If this looks strange, recall that $\partial M_0$ does not contain any points where $r\neq \infty$.) So, $\operatorname{WF}_{\mathrm{sc}}(u) = \operatorname{WF}_{\mathrm{sc}}(I[\varphi])$.
Because $I[\varphi] = \exp(-i(1-\chi(t)) r^2/(4t)) I_{\mathrm{osc}}[\varphi_-]+ I_{\mathrm{phg}}$, we have
\begin{multline}
 \operatorname{WF}_{\mathrm{sc}}(I_{\mathrm{phg}}) \subseteq \operatorname{WF}_{\mathrm{sc}}(\exp( -i(1 -\chi(t))/4 t \rho^2)I_{\mathrm{osc}}[\varphi_-] ) \cup \operatorname{WF}_{\mathrm{sc}}(I[\varphi]) \\ = \operatorname{WF}_{\mathrm{sc}}(\exp( -i(1 -\chi(t))/4 t \rho^2)I_{\mathrm{osc}}[\varphi_-]) \cup \operatorname{WF}_{\mathrm{sc}}(u).
 \label{eq:misc_y}
\end{multline}
By the conormality of $I_{\mathrm{osc}}$, we have 
\begin{equation*} 
	\operatorname{WF}_{\mathrm{sc}}(\exp( -i(1 -\chi(t))/4 t \rho^2)I_{\mathrm{osc}}) \subseteq \operatorname{graph}_{\partial M_0}(-2^{-1}(r/t) \dd r + 4^{-1}(r/t)^2 \dd t)
\end{equation*}
the right-hand side being the graph over $\partial M_0$ of the 1-form $-4^{-1}\dd (r^2/t) = - 2^{-1}(r/t) \dd r + 4^{-1}(r/t)^2 \dd t$, which is a smooth, \emph{nonvanishing} section of ${}^{\mathrm{sc}} T^* M_0$. Because it is nonvanishing,
\begin{equation}
\operatorname{WF}_{\mathrm{sc}}(\exp( -i(1 -\chi(t))/4 t \rho^2)I_{\mathrm{osc}}) \cap {}^{\mathrm{sc}} o^*  M_0 = \varnothing.
\end{equation}
Combining this with \cref{eq:misc_z} and \cref{eq:misc_y}, we have $\operatorname{WF}_{\mathrm{sc}}(I_{\mathrm{phg}}) \subseteq \operatorname{WF}_{\mathrm{sc}}(u) \cap {}^{\mathrm{sc}} o^* M_0$.
The upshot is that, in order to prove that $\operatorname{WF}_{\mathrm{sc}}(I_{\mathrm{phg}})  = \varnothing$, it suffices to prove that $\operatorname{WF}_{\mathrm{sc}}(u) \cap {}^{\mathrm{sc}} o^* M_0 = \varnothing$.

In order to accomplish this, one can use a standard argument based on the splitting $u = u_++u_-$,  where $u_\pm(t,x) = 1_{\pm t>0} u(t,x)$.
As $\operatorname{WF}_{\mathrm{sc}}(u) \subseteq ( \operatorname{WF}_{\mathrm{sc}}(u_-) \cup \operatorname{WF}_{\mathrm{sc}}(u_+))$, it suffices to prove $\operatorname{WF}_{\mathrm{sc}}(u_\pm) \cap {}^{\mathrm{sc}} o^* M_0 = \varnothing$ for each choice of sign.
To this end, note that $u_\pm$ satisfy the PDE
\begin{equation}
- i \partial_t u_\pm = P u_\pm \mp i \delta(t) f(x)
\end{equation}
in the sense of distributions.
As $\operatorname{WF}_{\mathrm{sc}}(\delta(t)f(x))$ is contained at fiber infinity (as can be seen using the Fourier transform in a local coordinate patch), it is irrelevant as far as sc-wavefront set in the interior of the fibers is concerned.

We have $P = \Delta_g \bmod \operatorname{Diff}_{\mathrm{sc}}^{1,-2}(M_0) $, where the `$-2$' indicates two orders of decay. So, the principal symbol of $P$ is
\begin{equation}
p(\tau,\xi) = \tau + g^{-1}(\xi,\xi) \in C^\infty({}^{\mathrm{sc}}T^* M_0 ).
\end{equation}
Associated to this function is the Hamiltonian vector field $H_p$, which is equal to 
\begin{equation*} 
	H_p = (\partial_\tau p)\partial_t + (\partial_\xi p)\cdot \partial_x  = \partial_t + 2g^{-1}(\xi,-)
\end{equation*} 
modulo terms coming from $\partial_x g^{-1}$ (which are suppressed at $\partial M_0$ relative to the other terms and therefore irrelevant to the argument below\footnote{We can replace $g$ with the exactly conic approximation $g_0$ in $p$. Then, $p$ is still principal in the sense of sc-decay. It does cease to be principal at fiber infinity, but in this appendix we are working at bounded frequency.}). 
Let 
\begin{equation*} 
	\mathsf{H}_p = \rho H_p,
\end{equation*} 
which restricts to a vector field on ${}^{\mathrm{sc}}T^*_{\partial M_0} M_0$. The Duistermaat-H\"ormander theorem --- see \cite{VasyGrenoble} for a precise statement in the context of the sc-calculus --- then says that the portion of $\operatorname{WF}_{\mathrm{sc}}(u_\pm)$ in ${}^{\mathrm{sc}}T^*_{\partial M_0} M_0$ consists of maximally extended integral curves of $\mathsf{H}_p$. Those in the zero section ${}^{\mathrm{sc}} o^* M_0$ are of the form
\begin{equation}
\gamma_{\theta_0} =( {}^{\mathrm{sc}} o^* M_0)\cap  \{\theta=\theta_0\},
\end{equation}
as follows from the explicit formula for $\mathsf{H}_p$.
Because $u_\pm$ vanishes when $\pm t <0 $, it has no sc-wavefront set over the corresponding half of $\operatorname{dilF}^\circ$. So, $\gamma_{\theta_0}\not\subseteq \operatorname{WF}_{\mathrm{sc}}(u_\pm)$, which, by Duistermaat--H\"ormander, implies $\operatorname{WF}_{\mathrm{sc}}(u_\pm) \cap {}^{\mathrm{sc}} o^* M_0 = \varnothing$.
\end{proof}
\section{Necessity of nf and dilF}
\label{ap:minimality}

In this appendix, we summarize the role of $\mathrm{nf},\mathrm{dilF}$ and why it is not possible to blow down either while maintaining the exponential-polyhomogeneous form of solutions of the Schr\"odinger equation. We only sketch the argument, and,  for simplicity, we work in the $\dim X = 1$ case.

In this appendix, when we say polyhomogeneous, we restrict to index sets
contained in $\mathbb R\times \mathbb N$. This rules out functions like $t^i$, which oscillate as $t\to\infty$. 

\begin{lemma}
	If $\theta_1,\theta_2,a_0,a_1,a_2$ are polyhomogeneous functions on a mwc $M$ and $p\in \partial M$ are such that $\theta_1,\theta_2$ are real-valued, $a_0\in  C^\infty(M;\bbC^\times)$, $a_1$ extends continuously to a neighborhood of $p$, and this extension vanishes at $p$, and $e^{i \theta_1} (a_0+a_1) = e^{i \theta_2} a_2$,
 	then, near $p$, the difference $\theta_1 - \theta_2$ is, in a neighborhood of $p$, has no $(j,k)$ for $ j<0$ in its index set.
 	\label{lem:phase_uniqueness}
\end{lemma}
\begin{proof}
	The function $a_0 + a_1$ is nonvanishing near $p$, so $b=(a_0+a_1)^{-1}$ is well-defined there, and $b$ is polyhomogeneous, with a continuous extension to the boundary of $M$ near $p$.

	Let $\tilde{b}$ be a globally polyhomogeneous function extending continuously to all of $\partial M$ and satisfying $\tilde{b}= b$ near $p$. Then,
\begin{equation}
e^{i\theta_1-i \theta_2} = \tilde{b} a_2,
	\label{eq:misc_128}
	\end{equation}
	near $p$. Since $a_2$ is polyhomogeneous, and since $a_0+a_1$ is uniformly bounded near $p$ (which implies that the same is true for $a_2$), it must be that $a_2$ extends continuously to a neighborhood of $p$. So, the right-hand side of \cref{eq:misc_128} extends continuously to a neighborhood of $p$.

	Moreover, the extension of $a_2$ to the boundary must be nonvanishing near $p$, and likewise for $\tilde{b}$, so $\tilde{b} a_2$ is nonvanishing near $p$. This implies that the difference
	\begin{equation}
	\theta_1-\theta_2 = - i \log(\tilde{b} a_2)
	\end{equation}
	is, near $p$, polyhomogeneous with all index sets in $\{z\in \bbR:\Re z \geq 0\}\times \bbN$.
\end{proof}

Consider now the Gaussian wavepacket
\begin{equation}
G(t,x) = \frac{1}{\sqrt{1+2it}}\exp\Big( -\frac{x^2}{1+2it} \Big).
\end{equation}
This solves the free Schr\"odinger equation in 1D, $i\partial_t G =  -(1/2)\partial_x^2 G$, with Schwartz initial data.

\begin{proposition}
	The Gaussian wavepacket above is not of the form $a e^{i \theta}$ for $\theta$ real-valued and $a,\theta$ polyhomogeneous on $M/\mathrm{nf}$ or $M/\mathrm{dilF}$.
\end{proposition}
Below, we will use the same names to refer to faces of the $M/\mathrm{f}$ as the corresponding faces in $M$. For example, though $\Sigma,\mathrm{dilF}$ are disjoint boundary hypersurfaces of $M$, the identically named boundary hypersurfaces of $M/\mathrm{nf}$ meet at the blown-down locus.
\begin{proof}[Proof sketch]
	First, suppose, to the contrary, that $G =e^{i \varphi} G_0$ for some $\varphi,G_0$ polyhomogeneous on $M/\mathrm{nf}$. Looking at the $t\to 0^+$ behavior of $G$ in compact subsets worth of $r$, it must be the case that
	\begin{itemize}
		\item[$\star$] the index set $\calE_{\Sigma}$ of $\varphi$ at $\Sigma$ can be taken to contain no terms $(j,k) \in \bbC\times \bbN$ with $\Re j<0$.
	\end{itemize}
	Near the corner $\Sigma \cap \mathrm{dilF} \subset M/\mathrm{nf}$, we can use $\rho=1/r$ and $s= t/r$ as a coordinate system, where both of these are local boundary-defining functions, $\rho$ of $\mathrm{dilF}$ and $s$ of $\Sigma$. In terms of these coordinates,
	\begin{equation}
	G = \frac{\rho^{1/2}}{\sqrt{\rho+2i s }}\exp\Big( -\frac{ 1}{\rho^2+2i s \rho } \Big).
	\end{equation}
	For $s>0$, this has the form $\rho^{1/2} \exp(i/(2 s \rho)) C^\infty(\bbR^+_s\times [0,\infty)_\rho;\bbC^\times)$ near $\rho=0$. Consequently, applying \Cref{lem:phase_uniqueness},
	\begin{equation}
		\varphi =  1/(2 s \rho) \bmod \calA^{0-}(\bbR^+_s\times [0,\infty)_\rho).
		\label{eq:misc_113}
	\end{equation}
	But, for no index set $\calE_{\mathrm{nf}}$ can $\varphi \in \calA^{\calE_{\mathrm{nf}},\calE_\Sigma }(([0,\infty)_s\times [0,\infty)_\rho )$ be consistent with \cref{eq:misc_113}, since this forces $(-1,0)\in \calE_\Sigma$, in conflict with our earlier observation ($\star$). So, the supposition that $G$ has the stated form on $M/\mathrm{nf}$ is not tenable.

	\begin{figure}[t!]
		\begin{center}
			\includegraphics[scale=.55]{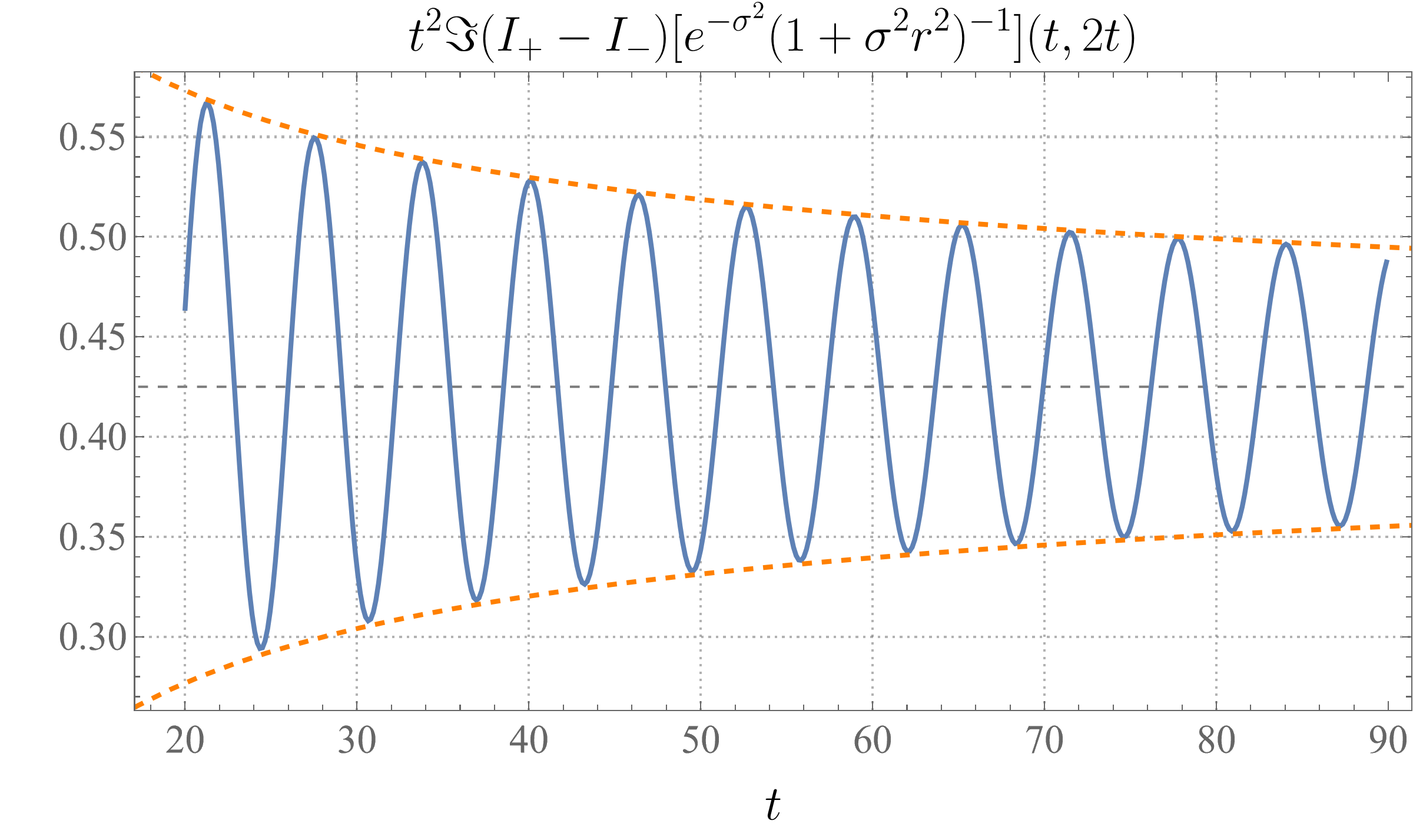}
		\end{center}
		\caption{The imaginary part $\Im t^2 I[\phi]$ of $t^2$ times $I[\phi]=I_+[\phi]-I_-[\phi]$, evaluated at $(t,r) = (t,2t)$, in blue, and its approximate envelope, in orange. The figure shows that $I[\phi]$ is the sum of the (decaying) oscillatory term  $-\exp(-i (1-\chi(t) )r^2/(4t)) ) I_{\mathrm{osc}}[\phi]$, which is oscillating with period $2\pi\approx 6$, and another term which, rather than decaying as $t\to\infty$, instead limits to
			\[
			-i\int_{-\infty}^{\infty} e^{-2i\lambda} \lambda (1+\lambda^2)^{-1} \dd \lambda \approx .425,
			\]
			in accordance with the discussion above. The approximate envelope of the oscillatory part is a curve decaying at a $t^{-1/2}$ rate almost passing through the peaks of $t^2 \Im  I[\phi]$, in accordance with how $I_{\mathrm{osc}}[\phi]$ is supposed to decay at a $t^{-5/2}$ rate.
		}
		\label{fig:numeric}
	\end{figure}

	Consider now $M/\mathrm{dilF}$.
	Since $(1+2i t)^{-1/2}$ \emph{is} polyhomogeneous on $C$, and therefore on $M/\mathrm{dilF}$, in order to prove the desired result it suffices to prove that
	we do not have $\smash{(1+2i t)^{1/2}} G = e^{i \varphi} a$
	for $\varphi$ real valued and $a,\varphi$ polyhomogeneous on $M/\mathrm{dilF}$.
	Near the corner $\mathrm{nf} \cap \mathrm{parF}$ on the blown-down manifold $M/\mathrm{dilF}$, we can use coordinates $\varrho = 1/t^{1/2}$ and $\tau = t/r^2$, with $\varrho$ a boundary-defining function of $\mathrm{parF}$ and $\tau$ a boundary-defining function of $\mathrm{nf}$. In terms of these coordinates,
	\begin{equation}
		(1+2i t)^{1/2}G = \exp(- 1/ \tau 	(\varrho^2+2i)).
	\end{equation}
	So, for $\varrho>0$, $(1+2i t)^{1/2}G$ is Schwartz as $\tau \to 0^+$. It follows that $a$ is Schwartz at $\mathrm{nf}$, however, restricting to $\varrho=0$, $(1+2i t)^{1/2}G$ has magnitude $1$, and therefore so does $a$. This contradicts the joint expandability of $a$ at the corner, and therefore polyhomogeneity.
\end{proof}

\section{A numerical example}
\label{sec:example}

We end by considering the example $I_\pm[\phi]$ for $\phi(\sigma,x) = 2^{-1} e^{- \sigma^2} (1+\sigma^2 r^2)^{-1}$. This is given by
\begin{equation}
	I_\pm[\phi](t,x) = \int_0^\infty e^{i \sigma^2 t \pm i \sigma r -\sigma^2}\frac{\sigma \dd \sigma }{1+\sigma^2 r^2}.
\end{equation}
As $\phi \in \calA^{(0,0),(0,0),(2,0)}(X_{\mathrm{res}}^{\mathrm{sp}})$, we know that $I_+[\phi] \in \calA^{(1,0)\cup (1/2,0), (2,0),(0,0) } (C_1) $ and $I_-[\phi] = \exp(-i (1-\chi(t) )r^2/(4t)) ) I_{\mathrm{osc}}[\phi] + I_{-\mathrm{phg}}$, for some
\begin{equation}
	I_{\mathrm{osc}}[\phi] \in \calA^{(1,0) \cup (1/2,0), (2,0), (5/2,0), \infty, (0,0)}(M)
\end{equation}
and $I_{-,\mathrm{phg}} \in\calA^{(1,0)\cup (1/2,0), (2,0),(0,0) } (C_1) $.  One point worth noting is that, despite $\phi$ being even, the difference $
I_+[\phi] - I_{-,\mathrm{phg}}[\phi]$ does not decay rapidly at $\mathrm{dilF}$, in contrast to what we observed in \S\ref{sec:microlocal} when $\phi$ is the output of the resolvent of a Schr\"odinger operator.
In other words, the integral
\begin{equation}
	I_+[\phi] - I_-[\phi] = \int_{-\infty}^\infty e^{i \sigma^2 t + i \sigma r - \sigma^2} \frac{\sigma \dd \sigma}{1+\sigma^2 r^2}
\end{equation}
is not purely oscillatory at $\mathrm{dilF}$, which can be seen clearly in the numerical plot \Cref{fig:numeric}.

\printbibliography

\end{document}